\pdfoutput=1

\documentclass[a4paper]{amsart}

\usepackage{amssymb,amscd}
\usepackage{amsmath,amsfonts,latexsym}
\usepackage{graphicx}
\usepackage{color}
\usepackage[english]{babel}
\usepackage[latin1]{inputenc}
\usepackage{enumerate,afterpage}
\usepackage{calc}

\usepackage{mathrsfs}
\usepackage{mathtools}
\usepackage{stmaryrd}

\usepackage{amssymb,amscd}
\usepackage{amsmath,amsfonts,latexsym}
\usepackage{graphicx}
\usepackage{color}
\usepackage[english]{babel}
\usepackage[latin1]{inputenc}
\usepackage{enumerate,afterpage}
\usepackage[all]{xy}

\newcounter{notes}%

\newtheorem{cor}{Corollary}[section]
\newtheorem{theorem}[cor]{Theorem}
\newtheorem{prop}[cor]{Proposition}
\newtheorem{lemma}[cor]{Lemma}

\newtheorem{Proposition}[cor]{Proposition}

\newtheorem{introthm}{Theorem}

\usepackage{amsthm}

\makeatletter
\newtheorem*{rep@theorem}{\rep@title}
\newcommand{\newreptheorem}[2]{%
\newenvironment{rep#1}[1]{%
 \def\rep@title{#2 \ref{##1}}%
 \begin{rep@theorem}}%
 {\end{rep@theorem}}}
\makeatother

\newreptheorem{theorem}{Theorem}

\theoremstyle{definition}
\newtheorem{defi}[cor]{Definition}

\newtheorem{example}[cor]{Example}

\newtheorem{remark}[cor]{Remark}
\theoremstyle{plain}

\def\co{\colon\thinspace}

\newcommand{\qcm}{{\mathcal QC}}
\newcommand{\pqc}{{\mathcal PQC}}

\newcommand{\cC}{{\mathcal C}}

\newcommand{\cT}{{\mathcal T}}

\newcommand{\cML}{{\mathcal M\mathcal L}}

\newcommand{\C}{{\mathbb C}}
\newcommand{\HH}{{\mathbb H}}

\newcommand{\N}{{\mathbb N}}

\newcommand{\R}{{\mathbb R}}

\newcommand{\Hyp}{\mathbb{H}}
\newcommand{\AdS}{\mathbb{A}\mathrm{d}\mathbb{S}}
\newcommand{\dS}{\mathrm{d}\mathbb{S}}

\newcommand{\SL}{\mathsf{SL}}
\newcommand{\PSL}{\mathsf{PSL}}

\newcommand{\PSO}{\mathsf{PSO}}
\newcommand{\PO}{\mathsf{PO}}

\newcommand{\PGL}{\mathsf{PGL}}

\newcommand{\ML}{\mathrm{ML}}
\newcommand{\CH}{\mathrm{CH}}

\newcommand{\rhombus}{C_{\diamond}}
\newcommand{\cro}{\mathrm{cr}}

\newcommand{\isom}{\mathsf{Isom}}

\newcommand{\im}{\Im}

\newcommand{\tr}{\mbox{\rm tr}}

\newcommand{\sh}{\mathrm{sinh}\,}
\newcommand{\ch}{\mathrm{cosh}\,}

\newcommand{\homeoqs}{\mathrm{Homeo}_{qs}(\RP^1)}

\newcommand{\Istar}{I^*}
\newcommand{\I}{I}
\newcommand{\II}{I\hspace{-0.1cm}I}
\newcommand{\III}{I\hspace{-0.1cm}I\hspace{-0.1cm}I}

\newcommand{\RR}{\mathbb R}
\newcommand{\CC}{\mathbb C}
\newcommand{\CP}{\mathbb{CP}}

\newcommand{\RP}{\mathbb{RP}}

\newcommand{\Herm}{\operatorname{Herm}}

\newcommand{\bminimatrix}[4]{\begin{bmatrix} #1 & #2 \\ #3 & #4 \end{bmatrix}  }

\newcommand{\Sa}{\mathcal S}

\newcommand{\Ein}{\mathrm{Ein}}

\newcommand{\cmp}{\mathbf{comp}}
\newcommand{\VCK}[1]{V_{#1}}

\let\oldtocsection=\tocsection

\let\oldtocsubsection=\tocsubsection

\let\oldtocsubsubsection=\tocsubsubsection

\renewcommand{\tocsection}[2]{\hspace{0em}\oldtocsection{#1}{#2}}
\renewcommand{\tocsubsection}[2]{\hspace{1em}\oldtocsubsection{#1}{#2}}
\renewcommand{\tocsubsubsection}[2]{\hspace{2em}\oldtocsubsubsection{#1}{#2}}

\AtBeginDocument{%
   \def\MR#1{}
}

\begin{document}

\title[Induced metrics on convex hulls of quasicircles]{The induced metric on the boundary of the convex hull of a quasicircle in hyperbolic and anti de Sitter geometry}

\author[F. Bonsante]{Francesco Bonsante}
\address{FB: Universit\`a degli Studi di Pavia}
\email{francesco.bonsante@unipv.it}
\urladdr{www-dimat.unipv.it/$\sim$bonsante}

\author[J. Danciger]{Jeffrey Danciger}
\address{JD: Department of Mathematics, University of Texas at Austin}
\email{jdanciger@math.utexas.edu}
\urladdr{www.ma.utexas.edu/users/jdanciger}

\author[S. Maloni]{Sara Maloni}
\address{SM: Department of Mathematics, University of Virginia}
\email{sm4cw@virginia.edu}
\urladdr{www.people.virginia.edu/$\sim$sm4cw}

\author[J.-M. Schlenker]{Jean-Marc Schlenker}
\address{JMS: Department of mathematics, University of Luxembourg}
\email{jean-marc.schlenker@uni.lu}
\urladdr{math.uni.lu/schlenker}

\thanks{Bonsante was partially supported by the Blue Sky research project ``Analytical and geometric properties of low dimensional manifolds''; Danciger was partially supported by an Alfred P. Sloan Foundation Fellowship and by NSF grants DMS-1510254 and DMS-1812216; Maloni was partially supported by NSF grant DMS-1506920 and DMS-1650811; Schlenker was partially supported by UL IRP grants GeoLoDim and NeoGeo and FNR grants INTER/ANR/15/11211745 and OPEN/16/11405402. The authors also acknowledge support from U.S. National Science Foundation grants DMS-1107452, 1107263, 1107367 ``RNMS: GEometric structures And Representation varieties'' (the GEAR Network). }

\date{\today}

\begin{abstract}
Celebrated work of Alexandrov and Pogorelov determines exactly which metrics on the sphere are induced on the boundary of a compact convex subset of hyperbolic three-space. As a step toward a generalization for unbounded convex subsets, we consider convex regions of hyperbolic three-space bounded by two properly embedded disks which meet at infinity along a Jordan curve in the ideal boundary. In this setting, it is natural to augment the notion of induced metric on the boundary of the convex set to include a gluing map at infinity which records how the asymptotic geometry of the two surfaces compares near points of the limiting Jordan curve. Restricting further to the case in which the induced metrics on the two bounding surfaces have constant curvature $K \in [-1,0)$ and the Jordan curve at infinity is a quasicircle, the gluing map is naturally a quasisymmetric homeomorphism of the circle. The main result is that for each value of~$K$, every quasisymmetric map is achieved as the gluing map at infinity along some quasicircle. We also prove analogous results in the setting of three-dimensional anti de Sitter geometry. 
Our results may be viewed as universal versions of the conjectures of Thurston and Mess about prescribing the induced metric on the boundary of the convex core of quasifuchsian hyperbolic manifolds and globally hyperbolic anti de Sitter spacetimes.
\end{abstract}

\maketitle

\tableofcontents

\section{Introduction}

\subsection{The induced metric on the boundary of a convex subset of $\HH^3$}
Let $\HH^3$ denote the three-dimensional hyperbolic space and let $\mathscr C \subset \HH^3$ be a compact convex subset with smooth boundary. By restriction from $\HH^3$, the boundary $\partial \mathscr C$ inherits a Riemannian metric which we refer to as the \emph{induced metric}, and the Gauss equation indicates that this metric has curvature $K\geq -1$. A celebrated theorem of Alexandrov  \cite{alex} and Pogorelov \cite{Po} states that, conversely, any smooth metric on the sphere with curvature $K>-1$ is the induced  metric on $\partial \mathscr C$ for some convex subset $\mathscr C \subset \HH^3$ and that, further, $\mathscr C$ is unique up to a global isometry of~$\HH^3$. This result in fact extends, by~\cite{Po}, to the general context of compact convex subsets of~$\HH^3$ whose boundary need not be smooth: any geodesic distance function on the sphere $\mathbb S^2$ with curvature $K\geq -1$ in the sense of Alexandrov is induced as the path metric on the boundary of a compact convex subset $\mathscr C \subset \HH^3$, unique up to isometry of $\HH^3$. 

A naive attempt to extend these results to arbitrary unbounded convex subsets immediately encounters problems. For instance, if $\overline \Omega \subset \partial \HH^3$ is any (not necessarily round) closed disk, then the convex hull $\mathscr C = \CH(\overline \Omega)$ is a closed half-space bounded by a convex pleated surface $\partial \mathscr C$ whose induced metric is just an isometric copy of the hyperbolic plane $\HH^2$, independent of $\overline \Omega$.
However,there does seem to be hope for extensions of the above theorems in cases that the boundary at infinity of the closed convex set $\mathscr C \subset \HH^3$ is small enough. For example, Rivin~\cite{rivin-comp} showed that any complete hyperbolic metric on the $n$-times punctured sphere is realized uniquely on the boundary of the convex hull $\mathscr C$ of $n$ points in $\partial \HH^3$ (such a $\mathscr C$ is called an ideal polyhedron).
We focus here on the situation where the boundary at infinity of $\mathscr C$ is a {\em quasicircle} $C$ (see below), so that the boundary of $\mathscr C$ is the disjoint union of two discs $\partial^+ \mathscr C$ and $\partial^- \mathscr C$. In this setting, the proper notion of induced metric for $\mathscr C$ includes not just the induced path metric on $\partial^+ \mathscr C$ and $\partial^- \mathscr C$, but also a \emph{gluing map at infinity} between $\partial^+ \mathscr C$ and $\partial^- \mathscr C$ which records how the asymptotic behavior of the two induced metrics compares for sequences going to infinity towards a point of $C$ along either surface. In the case that $\mathscr C = \CH(C)$ is the convex hull of such a quasicircle $C$, the path metrics on $\partial^+ \mathscr C$ and $\partial^- \mathscr C$ are each isometric to the hyperbolic plane and the induced metric on the boundary of $\mathscr C$ is reduced entirely to this gluing map, which turns out to be a quasisymmetric map. We will show (Theorem \ref{tm:induced-hyp} below) a partial extension of Alexandrov's theorem to this setting: 
   any quasisymmetric map of the circle is realized as the gluing map for some quasicircle. We also give a similar result when $\partial^+ \mathscr C$ and $\partial^- \mathscr C$ have constant curvature $K\in (-1,0)$ (Theorem \ref{tm:induced-hyp-K}).  
   Lorentzian versions of these results, in which the hyperbolic space is replaced by the anti de Sitter space $\AdS^3$, will be given as well (Theorems~\ref{tm:induced-ads} and \ref{tm:induced-ads-K}). We remark that, although there is not yet a well-developed analogue of the Alexandrov and Pogorelov theory in $\AdS^3$, the analogue of Rivin's result on induced metrics for ideal polyhedra were obtained by the last three authors~\cite{idealpolyhedra}.

\subsection{Quasicircles in $\CP^1$ and their convex hulls in $\HH^3$}\label{sec:intro-hyp}

We consider in this paper several natural constructions of \emph{gluing maps} associated to an oriented Jordan curve $C$ in $\CP^1$. The first construction comes from complex geometry and the others come from hyperbolic geometry. Since these constructions are invariant under the action of the conformal group $\PSL(2,\CC)$, there is no loss in generality in considering only the case that $C$ is a \emph{normalized} Jordan curve, meaning that $C$ contains the points $0,1, \infty \in \CP^1$ and these points appear in positive order in the orientation on $C$. We assume this is the case in the following discussion.

 The normalized, oriented Jordan curve $C$ divides $\CP^1$ into two connected components.
 We denote by $\Omega^+_C$ the component of $\CP^1 \setminus C$ on the positive, or ``upper", side of $C$ and by $\Omega^-_C$ the component on the negative, or ``lower" side. Then, by the Riemann Mapping Theorem, $\Omega^+_C$ is biholomorphic equivalent to the upper half-plane $\HH^2=\HH^{2+}$ in $\CC$ and $\Omega^-_C$ is biholomorphic equivalent to the lower half-plane $\HH^{2-}$. For each value of $\pm$, it follows from Caratheodory's theorem (see e.g.~\cite[Section 21]{Pommerenke}) that the biholomorphism $U^\pm_C\co \HH^{2\pm} \to \Omega^\pm_C$ extends to a homeomorphism $\partial U^\pm_C \co \RP^1 \cong \partial \HH^{2\pm} \xrightarrow{\cong} C$, which is well-defined upon imposing that $\partial U^\pm_C(i) = i$ for $i=0,1,\infty$.  The map $\varphi_C: \RP^1 \to \RP^1$ defined by $$\varphi_C =  (\partial U^-_C)^{-1}\circ \partial U^+_C$$ is a \emph{normalized} homeomorphism of $\RP^1$, meaning it is a homeomorphism that fixes $0,1, \infty$. It is called the \emph{gluing map between the upper and lower regions of the complement of $C$ in $\CP^1$}. The relationship between the properties of $C$ and the properties of $\varphi_C$ is in general mysterious. In particular, there seems to be no known good condition for a homeomorphism $\RP^1 \to \RP^1$ to be realized as the gluing map $\varphi_C$ associated to some Jordan curve $C$, see Thurston's comment \cite{thu_mathoverflow}. However, this gluing map is much better understood when the Jordan curves considered are restricted to the class of quasicircles.

An oriented Jordan curve $C$ in $\CP^1$ is called a \emph{quasicircle} if $C$ is the image of $\RP^1$ under a quasiconformal homeomorphism of $\CP^1$.
 Let $\qcm$ be the space of normalized quasicircles in $\CP^1$ with the Hausdorff topology. Then for any $C \in \qcm$, the gluing map $\varphi_C$ is a \emph{quasisymmetric} homeomorphism. The space of such normalized quasisymmetric homeomorphisms is called the \emph{universal Teichm\"uller space} and will be denoted $\mathcal T$. A classical result of Bers \cite{bers} states that the map $$\varphi_{.}: \qcm \to \mathcal T$$ is a bijection. In particular, every quasisymmetric homeomorphism is realized as the gluing map between the upper and lower regions in the complement of a unique normalized quasicircle $C$ in $\CP^1$, up to the action of the conformal group $\PSL(2,\CC)$.

In this paper, we study a second type of gluing map which is defined from hyperbolic geometry. Throughout the paper, we identify $\CP^1$ with the ideal boundary $\partial \HH^3$ of the hyperbolic three-space $\HH^3$ in the usual way.
Given a normalized Jordan curve $C$, let $\CH(C)$ denote the convex hull of $C$ in $\HH^3$, that is the smallest closed subset of $\HH^3$ whose accumulation set at infinity is $C$. There are two components of the boundary of $\CH(C)$, which we denote by $\partial^+ \CH(C)$ and $\partial^- \CH(C)$. By convention $\partial^+ \CH(C)$ is the component on the positive, or ``upper", side of $C$ and $\partial^- \CH(C)$ is the component on the negative, or ``lower" side. Each component is a pleated surface and inherits an induced path metric from $\HH^3$ which is isometric to the hyperbolic plane. The orientation-preserving isometries $V_C^\pm\co \HH^{2\pm} \to \partial^\pm \CH(C)$ extend (see Proposition~\ref{prop:extend}) to homeomorphisms $\partial V_C^\pm \co \RP^1 \cong \partial \HH^2 \xrightarrow{\cong} C$ and become well-defined upon imposing that $\partial V_C^\pm(i) = i$ for $i=0,1,\infty$. The map $\Phi_C\co \RP^1 \to \RP^1$ defined by $$\Phi_C = (\partial V_C^-)^{-1}\circ \partial V_C^+$$ is, similarly as above, a normalized homeomorphism which we call the \emph{gluing map between the upper and lower boundaries of the convex hull}. As above, in the case that $C$ is a quasicircle, $\Phi_C$ is quasisymmetric (Proposition~\ref{pr:main-properness}). However, the map $$\Phi_. : \qcm \to \mathcal T$$ is more mysterious than its counterpart $\varphi_.$ above. The first main goal of this paper is:

\begin{introthm} \label{tm:induced-hyp}
The map $\Phi_{.}\co \qcm \to \cT$ is surjective: Any normalised quasisymmetric homeomorphism of the circle is realized as the gluing map between the upper and lower boundary of the convex hull of a normalized quasicircle in $\CP^1$.
\end{introthm}

The image of $\Phi_{.}$ is already known to contain a large subset of $\cT$, namely the collection of quasisymmetric maps which are equivariant, conjugating one Fuchsian closed surface group action to another. Since this is important for both the context and for the proof of Theorem~\ref{tm:induced-hyp}, we make a short digression to explain this.
 
 Let $\Sigma = \Sigma_g$ be the closed surface of genus $g \geq 2$. Recall that a discrete faithful representation $\rho: \pi_1 \Sigma \to \PSL(2,\C)$ is called \emph{quasifuchsian} if its action leaves invariant an oriented quasicircle $C \subset \CP^1$, or alternatively if the convex core $\cC_\rho$ in $M_\rho = \rho(\pi_1(\Sigma)) \backslash \HH^3$ is compact, homeomorphic to $\Sigma \times [0,1]$ (except in the case that $\rho$ is Fuchsian, in which $\cC_\rho$ is a totally geodesic surface in $M_\rho$). In this case, $\cC_\rho$ is bounded by two convex pleated surfaces $\partial^+ \cC_\rho \cong \Sigma \times \{1\}$ and $\partial^-  \cC_\rho \cong \Sigma \times \{0\}$. Each surface inherits a path metric from $M$ which is locally isometric to the hyperbolic plane. Hence the quasifuchsian representation $\rho$ determines two elements $X^+$ and $X^-$ of the Teichm\"uller space $\mathcal T(\Sigma) = \mathcal T_g$, namely the induced path metrics on the top $\partial^+ \cC_\rho$ and bottom $\partial^-  \cC_\rho$ of the convex core respectively. 
Thurston conjectured that conversely any pair $(X^+, X^-) \in \mathcal T(\Sigma) \times \mathcal T(\Sigma)$ of hyperbolic metrics is realized as the metric data on the boundary of the convex core of a unique quasifuchsian manifold $M_\rho$ (up to isometry). 
The existence portion of this statement is due to Sullivan \cite{sullivan_travaux}, Epstein--Marden \cite{ep-ma} and Labourie \cite{L4}. The uniqueness remains an open question.
\begin{theorem}[]\label{thm:exist-thurston-induced}
Let $X^+, X^- \in \cT(\Sigma)$ be two hyperbolic structures on the closed surface $\Sigma$ of genus $g \geq 2$. Then there exists a quasifuchsian representation $\rho\co \pi_1 \Sigma \to \PSL(2,\C)$ for which $X^+$ and $X^-$ are realized respectively as the induced metrics on the top and bottom boundary components of the convex core of $M_{\rho}$.
\end{theorem}

In the context of the above discussion, the preimage in $\HH^3$ of the convex core $\cC_\rho$ of a quasifuchsian hyperbolic three-manifold $M_{\rho}$ is the convex hull $\CH(C)$ of the invariant quasicircle $C$. We call $C$ a \emph{quasifuchsian} quasicircle. For each value of $\pm$, the map $(V^\pm_C)^{-1} \co \partial^\pm \CH(C) \to \HH^{2\pm}$, defined above, conjugates the action of $\rho$ on $\partial^\pm \CH(C)$ to a properly discontinuous action by isometries on $\HH^{2\pm}$, namely the holonomy representation $\rho^\pm\co \pi_1 \Sigma \to \PSL(2,\RR)$ of the hyperbolic structure $X^\pm$. Hence the gluing map $\Phi_C\co \RP^1 \to \RP^1$, between the upper and lower boundaries of the convex hull of the $\rho$-invariant quasicircle $C$, is equivariant taking the action of one Fuchsian representation $\rho^+$ to another $\rho^-$. We call such a quasisymmetric homeomorphism \emph{quasifuchsian}. Theorem \ref{thm:exist-thurston-induced} implies that given a quasifuchisan quasisymmetric homeomorphism $v\co  \RP^1 \to  \RP^1$, there exists a quasifuchsian quasicircle $C$ such that $v = \Phi_C$. Hence in Theorem~\ref{tm:induced-hyp}, the image of the quasifuchsian quasicircles is precisely the set of quasifuchsian quasisymmetric maps. This fact, together with a density statement for quasifuchsian quasisymmetric maps in~$\mathcal T$, will be used to prove Theorem~\ref{tm:induced-hyp}. 
Indeed, we think of Theorem~\ref{tm:induced-hyp} as a universal version of Theorem \ref{thm:exist-thurston-induced}. We note that, similarly to Theorem \ref{thm:exist-thurston-induced}, in the context of Theorem~\ref{tm:induced-hyp}, the question of whether the quasisymmetric homeomorphism $\Phi_{C}$ uniquely determines the quasicircle $C$ remains open. Theorem~\ref{thm:parameterized-hyp} in Section~\ref{sec:param} discusses a slightly different version of Theorem~\ref{tm:induced-hyp} about parameterized quasicircles that more superficially resembles the statement of Theorem \ref{thm:exist-thurston-induced}. In fact, the injectivity of the map in Theorem~\ref{thm:parameterized-hyp} will imply Thurston's conjecture discussed above.

\subsection{Gluing maps at infinity for $K$--surfaces in $\HH^3$}\label{K_hyp}

For $K \in \RR$, a $K$-surface in a Riemannian manifold is a smoothly embedded surface whose Gauss curvature is constant equal to $K$. Let $C \subset \CP^1$ be an oriented Jordan curve. Then, Rosenberg--Spruck~\cite[Theorem 4]{RS} showed that for each $K \in (-1,0)$, there are exactly two complete $K$-surfaces embedded in $\HH^3$ which are asymptotic to $C$  (see Theorem~\ref{tm:K-surfaces-hyp}). They are each locally convex, but with opposite convexity, and together they bound a convex region $\mathscr C_K(C)$ of $\HH^3$ that contains the convex hull $\CH(C)$. By convention, for each value of $\pm$, we denote by $S^\pm_K = S^\pm_K(C)$ the $K$-surface spanning $C$ that lies between $\partial^\pm \CH(C)$ and $\Omega^\pm_C$. Note that for $K$ varying from $0$ to $-1$, the $K$-surfaces $S^+_K$ (respectively $S^-_K$) in fact form a foliation of the upper (respectively lower) component of $\HH^3\setminus \CH(C)$ which limits to $\Omega^+_C$ (respectively $\Omega^-_C$) as $K \to 0$ and to $\partial^+ \CH(C)$ (respectively $\partial^-\CH(C)$) as $K \to -1$.

Generalizing the above, we may consider, for each $K \in (0,1)$, a gluing map between the upper and lower $K$-surfaces spanning a normalized Jordan curve $C$ as follows. Let $\HH^{2+}_K$ (resp. $\HH^{2-}_K$) denote the upper (resp. lower) half-plane in $\CC$ equipped with the unique $\PSL(2,\R)$-invariant metric of constant curvature $K$. For each value of $\pm$, the $K$-surface $S^\pm_K(C)$ is orientation-preserving isometric to $\HH^{2+}_K$. The orientation-preserving isometry $V^\pm_{C, K}\co \HH^{2\pm}_K \to S^\pm_K(C)$ extends (Proposition \ref{prop:extend}) to a homeomorphism $\partial V^\pm_{C, K} \co \RP^1 \cong \partial \HH^2 \xrightarrow{\cong} C$ which becomes well-defined upon imposing that $\partial V^\pm_{C, K}(i) = i$ for $i=0,1,\infty$. The map $\Phi_{C, K}\co \RP^1 \to \RP^1$ defined by $$\Phi_{C, K} = (\partial V^-_{C, K})^{-1}\circ \partial V^+_{C, K}$$ is, similarly as above, a normalized homeomorphism which we call the \emph{gluing map between the upper and lower $K$-surfaces spanning $C$}. As above, in the case that $C$ is a quasicircle, $\Phi_{C, K}$ is quasisymmetric (Proposition~\ref{pr:main-properness}). Our second main result is:

\begin{introthm} \label{tm:induced-hyp-K} 
Given $K\in (-1,0)$, the map $\Phi_{\cdot, K} \co \qcm \to \cT$ is surjective: Any normalised quasisymmetric homeomorphism of the circle is realized as the gluing map between the upper and lower $K$--surfaces spanning some normalized quasicircle in $\CP^1$.
\end{introthm}

 Theorem~\ref{tm:induced-hyp} may be thought of as the limiting case $K = -1$ of Theorem~\ref{tm:induced-hyp-K}. Indeed, the proof of both theorems follow a similar general strategy. However, we keep the two statements separate since the technical tools required for the proofs are different. As for Theorem~\ref{tm:induced-hyp}, we do not determine whether $\Phi_{\cdot, K}$ is injective.

Next, for $K \in (-1,0)$, the third fundamental form $\III$ on the $K$-surface $S_K^\pm(C)$ (see Section \ref{sssc:surfaces}) is a positive definite symmetric two-tensor which has constant curvature $K^* = \frac{K}{K+1} \in (-\infty,0)$.

 Proposition \ref{principal_curv} shows that the principal curvatures of $S^\pm_{K}(C)$ are bounded away from $0$ and $\infty$, and so the third fundamental form $\III$ on $S^\pm_{K}(C)$ is a complete metric. The rescaled isometry  $V^{*\pm}_{C, K}\co \HH^{2\pm}_{K^*} \to (S^\pm_{K}(C), \III)$ extends to a homeomorphism $\partial V^{*\pm}_{C, K} \co \RP^1 \xrightarrow{\cong} C$ which is well-defined upon imposing that $\partial V^{*\pm}_{C, K}(i) = i$ for $i=0,1,\infty$. Proposition~\ref{pr:main-properness-III} shows that the normalized homeomorphism $\Phi_{C, K^*}\co \RP^1 \to \RP^1$ defined by 
 $\Phi^*_{C, K} = (\partial V^{*-}_{C, K})^{-1}\circ \partial V^{*+}_{C, K}$ is quasisymmetric if $C$ is a quasicircle. Our third main result is:
 
\begin{introthm} \label{tm:III-hyp-K}
Given $K \in (-1,0)$, the map $\Phi^*_{\cdot, K}\co \qcm \to \cT$ is surjective: Any normalised quasisymmetric homeomorphism of the circle is realized as the gluing map of the third fundamental forms of the $K$--surfaces spanning some normalized quasicircle in $\CP^1$.
\end{introthm}

As in the discussion of Theorem~\ref{tm:induced-hyp}, we note that the analogues of Theorem~\ref{tm:induced-hyp-K} and Theorem~\ref{tm:III-hyp-K} in the setting of quasifuchsian quasicircles are already known: The restrictions of the maps $\Phi_{\cdot, K}$ and  $\Phi^*_{\cdot, K}$ to the space of quasicircles invariant under some quasifuchsian representation $\rho\co \pi_1 \Sigma \to \PSL(2,\C)$ is surjective onto the space of quasisymmetric homeomorphisms $v\co \RP^1 \to \RP^1$ which conjugate one Fuchsian representation of $\pi_1\Sigma$ to another.
This follows from work of Labourie~\cite{L4} which shows, much more generally, that the convex hyperbolic structures on a compact hyperbolic manifold $M$, in particular $M = \Sigma \times [0,1]$, induce all possible metrics of curvature bounded between zero and $-1$ on the boundary $\partial M$.
Schlenker~\cite{hmcb} showed further that the convex hyperbolic structure on $M$ realizing any given metric on $\partial M$ is unique.
Similarly, any metric of negative curvature on $\partial M$ is realized uniquely as the third fundamental form on the boundary of a unique convex hyperbolic structure on $M$.

\subsection{Quasicircles in the $\mathrm{Ein}^{1,1}$ and their convex hulls in $\AdS^3$}\label{sec:intro-ads}

We will also prove analogues of Theorems~\ref{tm:induced-hyp}, \ref{tm:induced-hyp-K}, and~\ref{tm:III-hyp-K} in the setting of three-dimensional anti de Sitter geometry. Anti de Sitter space $\AdS^3$ is a Lorentzian analogue of hyperbolic space $\HH^3$. It is the model for Lorentzian geometry of constant negative curvature in dimension $2+1$. The natural boundary at infinity $\partial \AdS^3$ of $\AdS^3$ is the Einstein space $\mathrm{Ein}^{1,1}$, a conformal Lorentzian space analogous to the Riemann sphere $\CP^1$.

In this setting, it is natural to consider Jordan curves $C \subset \mathrm{Ein}^{1,1}$ which are \emph{achronal}, meaning that in any small neighborhood of a point $x$ of $C$, all other points of $C$ are seen only in spacelike (positive) or lightlike (null) directions for the Lorentzian metric. We will restrict further to the class of achronal Jordan curves of $\Ein^{1,1}$ which bound a topological disk in $\AdS^3$, calling these the \emph{achronal meridians}. Achronal meridians $C$ are precisely the curves for which we can make sense of a notion of convex hull $\CH(C)$ in $\AdS^3$. See Section~\ref{sc:prelim_ads}. Amongst the achronal meridians, we will distinguish those for which the relationship between nearby points is spacelike (positive) only, calling these the \emph{acausal meridians}.

The null lines on $\mathrm{Ein}^{1,1}$ determine two transverse foliations by circles which endow $\mathrm{Ein}^{1,1}$ with a product structure $\mathrm{Ein}^{1,1} \cong \RP^1 \times \RP^1$. The identity component of the isometry group of $\AdS^3$ is also a product $\mathrm{Isom}_0 \AdS^3 \cong \PSL(2,\RR) \times \PSL(2,\RR)$ acting factor-wise on $\mathrm{Ein}^{1,1} \cong \RP^1 \times \RP^1$ by M\"obius transformations. An acausal meridian in $\Ein^{1,1}$ is precisely one which arises as the graph of an orientation-preserving homeomorphism $f: \RP^1 \to \RP^1$. It is this map $f$ which plays the role of the gluing map between the top and bottom regions of the complement of a Jordan curve in $\CP^1$, although in this setting $f$ arises via the product structure rather than as a gluing map.

Given an orientation-preserving homeomorphism $f: \RP^1 \to \RP^1$, let $\Gamma(f) \subset \RP^1 \times \RP^1$ denote the graph of~$f$.
Since the constructions we consider are invariant under $\mathrm{Isom}_0 \AdS^3$, we restrict to {\em normalized homeomorphisms}, i.e. we assume that $f(i) = i$ for $i=0,1,\infty$.
Since $\Gamma(f) \subset \mathrm{Ein}^{1,1}$ is an acausal meridian, the convex hull $\CH(\Gamma(f)) \subset \AdS^3$ is well-defined. 
There are two components of the boundary of $\CH(\Gamma(f))$, which we denote by $\partial^+ \CH(\Gamma(f))$ and $\partial^- \CH(\Gamma(f))$ (unless $f$ is a M\"obius map, in which case $\CH(\Gamma(f))$ is a totally geodesic spacelike plane in $\AdS^3$, in which case $\partial^+ \CH(\Gamma(f)) = \partial^- \CH(\Gamma(f))$). By convention $\partial^+ \CH(\Gamma(f))$ is the component on the ``future'' side of $\CH(\Gamma(f))$, and $\partial_- \CH(\Gamma(f))$ is the component on the ``past'' side.
Each component inherits a path metric from $\AdS^3$ which is locally isometric to the hyperbolic plane. However, by contrast to the setting of hyperbolic geometry above, this induced metric need not be complete but may be isometric to any region of $\HH^2$ bounded by disjoint geodesics (see~\cite[Cor 6.12]{BonsanteSurvey} and~\cite[Prop 6.16]{BeBo}). 
We will focus here on a special class of acausal meridians which are the analogues of the quasicircles in $\CP^1$. These have many nice properties, in particular the induced metrics on the future and past boundary components of the convex hull are complete.

We define a \emph{quasicircle in $\Ein^{1,1}$} to be an acausal meridian which arises as the graph $\Gamma(f)$ of a quasisymmetric homeomorphism $f\co \RP^1 \to \RP^1$. Let $\qcm(\Ein^{1,1})$ denote the space of all normalized quasicircles $\Gamma(f)$ in $\Ein^{1,1}$, i.e. those for which $f$ is normalized. Then $\qcm(\Ein^{1,1})$ is in natural bijection with the universal Teichm\"uller space $\cT$. Assume that $f: \RP^1 \to \RP^1$ is quasisymmetric.  Then the induced metrics on $\partial^+ \CH(\Gamma(f))$ and $\partial^- \CH(\Gamma(f))$ are complete (Proposition~\ref{prop:extend-ads}) and the orientation-preserving isometries $\VCK{\Gamma(f)}^\pm\co \HH^{2+} \to \partial^\pm \CH(\Gamma(f))$ extend to homeomorphisms $\partial \VCK{\Gamma(f)}^\pm \co \RP^1  \xrightarrow{\cong} \Gamma(f)$ and become well-defined upon imposing that $\partial V_{\Gamma(f)}^\pm(i) = (i,i)$ for $i=0,1,\infty$. The map $\Psi_{\Gamma(f)}\co \RP^1 \to \RP^1$ defined by
$$\Psi_{\Gamma(f)} = (\partial \VCK{\Gamma(f)}^-)^{-1}\circ \partial \VCK{\Gamma(f)}^+$$
is, as in the hyperbolic case described above, a normalized homeomorphism which we call the \emph{gluing map between the future and past boundaries of the convex hull}. In fact, we show (Proposition~\ref{prop:well-defined-ads}) that $\Psi_{\Gamma(f)}$ is also quasisymmetric. Hence this construction gives a map $\Psi_{.}\co \qcm(\Ein^{1,1}) \to \cT$ analogous to the map $\Phi_{.}: \qcm \to \cT$ defined above in the context of hyperbolic geometry.

\begin{introthm} \label{tm:induced-ads} 
The map $\Psi_{.}\co \qcm(\Ein^{1,1}) \to \cT$ is surjective: Any normalized quasisymmetric homeomorphism of the circle is realized as the gluing map at infinity for the convex hull of a normalized quasicircle in $\Ein^{1,1}$.
\end{introthm}

As in the discussion following Theorem~\ref{tm:induced-hyp} about $\Phi_{.}$, the image of $\Psi_{.}$ is already known to contain a large subset of $\cT$, namely the collection of quasifuchsian quasisymmetric homeomorphisms. Indeed, this follows from the study of anti de Sitter analogues of quasifuchsian hyperbolic manifolds called \emph{globally hyperbolic maximal compact} (GHMC) $\AdS^3$ spacetimes. Such a spacetime is non-compact, homeomorphic to $\Sigma \times \RR$ for some closed surface $\Sigma = \Sigma_g$ of genus $g \geq 2$, and has holonomy representation $\rho: \pi_1 \Sigma \to \mathrm{Isom}_0(\AdS^3) \cong \PSL(2,\RR) \times \PSL(2,\R)$ for which the projections $\rho_L$ and $\rho_R$ to the left and right factors are Fuchsian.
Conversely, every such representation $\rho = (\rho_L, \rho_R)$, called a GHMC representation, determines a unique GHMC $\AdS^3$ manifold $M_\rho$.
Any GHMC $\AdS^3$ spacetime has a compact convex core $\cC_\rho$ homeomorphic to $\Sigma \times [0,1]$ (except when $\rho_L = \rho_R$, in which case $\cC_\rho$ is a totally geodesic spacelike surface). Much like the convex core of a quasifuchsian hyperbolic $3$-manifold, $\cC_\rho$ is bounded by two spacelike convex pleated surfaces $\partial^+ \cC_\rho \cong \Sigma \times \{1\}$ and $\partial^-  \cC_\rho \cong \Sigma \times \{0\}$. Each surface inherits a path metric from $M_\rho$ which is locally isometric to the hyperbolic plane. Hence the representation $\rho$ determines two elements $X^+$ and $X^-$ of the Teichm\"uller space $\mathcal T(\Sigma) = \mathcal T_g$, namely the induced path metrics on the top $\partial^+ \cC_\rho$ and bottom $\partial^-  \cC_\rho$ of the convex core respectively. 
Mess~\cite{mess} conjectured that conversely any pair $(X^+, X^-) \in \mathcal T(\Sigma) \times \mathcal T(\Sigma)$ of hyperbolic metrics is realized as the metric data on the boundary of the convex core of a unique GHMC spacetime $M_\rho$ (up to isometry). This is the analogue of the conjecture of Thurston described in Section \ref{sec:intro-hyp}. The analogue of Theorem \ref{thm:exist-thurston-induced}, that existence holds in Mess's conjecture, was proved by Diallo~\cite{diallo2013}.

\subsection{Gluing maps at infinity for $K$--surfaces in $\AdS^3$}\label{K_ads}

Given a quasicircle $\Gamma(f)$ in $\Ein^{1,1}$, Bonsante and Seppi \cite{bon_are} proved that the $K$--surfaces spanning $\Gamma(f)$, for $K$ varying in $(-\infty, -1)$ form a foliation of the complement of the convex hull of $\CH(\Gamma(f))$ in the \emph{invisible domain} $E(\Gamma(f))$ of $\Gamma(f)$, the maximal convex region of $\AdS^3$ consisting of points which see the curve $\Gamma(f)$ in spacelike directions. For each $K \in (-\infty, -1)$, there is exactly one $K$-surface $\Sa^+_K(\Gamma(f))$ (resp. $\Sa^-_K(\Gamma(f))$)  in $E(\Gamma(f))$ which is asymptotic to $\Gamma(f)$ and lies in the future (resp. past) of $\CH(\Gamma(f))$; it is convex toward the past (resp. future).

The \emph{gluing map $\Psi_{\Gamma(f), K}\co \RP^1 \to \RP^1$ between the future and past $K$-surfaces spanning $\Gamma(f)$} is defined in exactly the same way as the gluing map $\Phi_{C, K}$ between the top and bottom $K$-surfaces in $\HH^3$ that span a quasicircle $C$ in $\CP^1$. Indeed, $\Psi_{\Gamma(f)}$ is a normalized quasisymmetric homeomorphism (Proposition \ref{prop:well-defined-ads}).

\begin{introthm} \label{tm:induced-ads-K} 
Given $K\in (-\infty, -1)$, the map $\Psi_{\cdot, K}\co \qcm(\Ein^{1,1}) \to \cT$ is surjective: Any normalized quasisymmetric homeomorphism of the circle is realized as the gluing map between the future and past $K$--surfaces spanning some normalized quasicircle in $\Ein^{1,1}$.
\end{introthm}

As in the context of Theorem~\ref{tm:induced-hyp-K}, the map $\Psi_{\cdot, K}$ is known to take the set of GHMC quasicircles in $\Ein^{1,1}$, namely those that are invariant under a GHMC representation, surjectively onto the quasifuchsian quasisymmetric homeomorphisms. 
This follows from a more general theorem of Tamburelli \cite{tamburelli2016} which states that any two metrics of curvature less than $-1$ (in particular, metrics of constant curvature $K < -1$) on a closed surface $\Sigma$ of genus $g \geq 2$ are induced on the boundary of a convex GHMC AdS structure on $\Sigma \times [0,1]$.

The AdS geometry analogue of Theorem \ref{tm:III-hyp-K} is also true. Let $\Psi^*_{\cdot, K}\co \qcm(\Ein^{1,1}) \to \cT$ be the map assigning to a quasicircle $\Gamma(f)$ in $\Ein^{1,1}$ the gluing map at infinity between the third fundamental forms on the future and past $K$-surfaces spanning $\Gamma(f)$, defined analogously to the map $\Phi^*_{\cdot, K}$. Unlike in the setting of hyperbolic geometry, this statement is \emph{equivalent} to Theorem \ref{tm:induced-ads-K} by a simple argument using the duality in $\AdS$ between points and spacelike totally geodesic planes. 

\begin{introthm} \label{tm:III-ads-K}
Given $K \in (-\infty,-1)$, the map $\Psi^*_{\cdot, K}\co \qcm(\Ein^{1,1}) \to \cT$ is surjective: Any normalised quasisymmetric homeomorphism of the circle is realized as the gluing map at infinity between the third fundamental forms of future and past $K$--surfaces spanning some normalized quasicircle in $\Ein^{1,1}$.
\end{introthm}

\subsection{Acknowledgements}
The authors would like to thank Andrea Seppi for his generosity in sharing Figure \ref{fig:torus} with us and for allowing us to use the background material from \cite{bon_are} in Sections \ref{prel ads} and \ref{projec}.

\section{Preliminaries I: Hyperbolic geometry, quasicircles, quasisymmetric maps}

Here we collect some preliminaries relevant for Theorems~\ref{tm:induced-hyp}, \ref{tm:induced-hyp-K}, and \ref{tm:III-hyp-K}. Anti de Sitter geometry preliminaries, relevant for Theorems~\ref{tm:induced-ads}, \ref{tm:induced-ads-K}, and \ref{tm:III-ads-K}, will be given in Section~\ref{sc:prelim_ads}.

\subsection{Quasiconformal, quasi-isometric, quasisymmetric mappings}

Maps that are not structure preserving, but only quasi structure preserving play an important role in both conformal geometry and metric geometry. We begin with the definitions of such maps and then examine the important example of the hyperbolic plane.

Let $f\co X \to Y$ be a diffeomorphism between Riemann surfaces (not necessarily compact). Then $f$ is called \emph{$K$-quasiconformal} if the complex dilatation $K(f) = \frac{1+|| \mu ||_\infty}{1-|| \mu ||_\infty}$ is at most $K$, where here $\mu = \mu(f)$ is the Beltrami differential, defined by the equation $\partial f/ \partial \bar{z} = \mu(z) \partial f/ \partial z$.
This condition makes sense, more generally, in the setting that $f\co X \to Y$ is a (not necessarily $C^1$) homeomorphism between Riemann surfaces whose derivatives (in the sense of distributions) are in $L^2$. See~\cite{Lehto}.

Let $(X,d_X)$ and $(Y,d_Y)$ be metric spaces. For $A > 1$, a map $f: X \to Y$ is called an $A$-\emph{quasi-isometric embedding} if for all $x_1, x_2 \in X$,  
\begin{align*}
\frac{1}{A}d_X(x_1,x_2) - A &\leq d_Y(f(x_1), f(x_2)) \leq A d_X(x_1,x_2) + A.
\end{align*}
More typically, the multiplicative and additive constants are allowed to be different, but for simplicity we will work with this definition. A map is a quasi-isometric embedding if it is an $A$-quasi-isometric embedding for some $A$. The map $f\co X \to Y$ is called an $A$-\emph{quasi-isometry} if it is an  $A$-quasi-isometric embedding and is $A$-dense in $Y$ for some $A > 1$.
It is well-known that if $X$ and $Y$ are $\delta$-hyperbolic spaces, then any quasi-isometric embedding (resp. any quasi-isometry) $f\co X \to Y$ extends uniquely to an embedding (resp. a homeomorphism) $\partial f\co \partial_\infty X \to \partial_\infty Y$ of the visual boundaries.

The hyperbolic space $\HH^n$ is the unique simply connected, complete Riemannian $n$-manifold of constant curvature $-1$.
In dimension $n=2$, the hyperbolic plane $\HH^2$ serves as an important example both of a Riemann surface and of a $\delta$-hyperbolic metric space.
A common model, which we will use frequently in this paper, realizes the hyperbolic plane $\HH^2$ as the upper half-plane $\HH^2 = \HH^{2+}$ in the complex plane $\CC$ equipped with the metric $ds^2 = (dx^2 + dy^2)/y^2$. The visual boundary $\partial \HH^{2+}$ of $\HH^{2+}$ naturally identifies with the equator $\RP^1$ in $\CP^1$, and the natural orientation on $\HH^{2+}$, coming from restriction from the complex plane, induces an orientation on $\RP^1$ that agrees with the orientation coming from the ordering of the reals. 
The lower half-plane $\HH^{2-}$ in $\CC$, equipped with the metric $ds^2 = (dx^2 + dy^2)/y^2$, also gives a model for the hyperbolic plane. The visual boundary $\partial \HH^{2-}$ also identifies with $\RP^1$. The natural orientation on $\HH^{2-}$ coming from restriction from $\CC$ induces an orientation on $\RP^1$ which is opposite to that induced by $\HH^{2+}$. 
The action of the group $\PSL(2,\RR)$ of real fractional linear transformations on $\CP^1 = \HH^{2+} \sqcup \RP^1 \sqcup \HH^{2-}$
restricts to actions by orientation-preserving isometries on $\HH^{2+}$ and $\HH^{2-}$. Hence the orientation-preserving isometries of $\HH^{2+}$ are precisely the conformal automorphisms of $\HH^{2+}$, and the same is true for $\HH^{2-}$. Each such map extends to a M\"obius map of the projective line $\RP^1$. In fact, both quasi-isometries and quasiconformal homeomorphisms of the hyperbolic plane extend to homeomorphisms of the visual boundary which are not quite M\"obius transformations. These are called quasisymmetric homeomorphisms.

In what follows we will say that $F: \HH^{2+} \to \HH^{2+}$ \emph{extends to } $f: \RP^1 \to \RP^1$ or that $f$ \emph{extends to} $F$ if the map $\overline{F}: \overline{\HH^{2+}} \to \overline{\HH^{2+}}$ which restricts to $F$ on $\HH^{2+}$ and to $f$ on $\RP^1$ is continuous along $\RP^1$. The following is well known, see e.g.~\cite{FlMa}.

\begin{prop}\label{prop:extend-QC}
Any quasiconformal homeomorphism $F: \HH^{2+} \to \HH^{2+}$ extends to a homeomorphism $f:\RP^1 \to \RP^1$.
\end{prop}

\begin{defi}\label{def:k-qs}
An orientation-preserving homeomorphism $f:\RP^1 \to \RP^1$ is called \emph{$k$-quasisymmetric} if it admits an extension $F:\HH^{2+} \to \HH^{2+}$ to the upper half-space which is $k$-quasiconformal. We call $f$ quasisymmetric if it is $k$-quasisymmetric for some~$k$.
\end{defi}

From Definition~\ref{def:k-qs}, it is clear that the composition of a $k$-quasisymmetric homeomorphism with a $k'$-quasisymmetric homeomorphism is $kk'$-quasisymmetric. Quasisymmetric maps also satisfy a useful compactness result.
The following is an immediate consequence of well-known compactness results for $k$-quasiconformal mappings (see for instance \cite[Theorem 5.2]{Lehto}). In what follows, a homeomorphism of $\RP^1$ is called \emph{normalized} if it fixes $0,1,$ and $\infty$. 
\begin{lemma}\label{lm:cmpqs}
Let $f_n: \RP^1 \to \RP^1$ be a sequence of $k$-quasisymmetric homeomorphisms.
Then either there exists a subsequence converging uniformly to a $k$-quasisymmetric homemorphism, or there are two points $q,p\in\RP^1$ such that
$f_n(x)\to q$ uniformly on any compact subset of $\RP^1\setminus\{p\}$ and  $f_n^{-1}(x)\to p$ uniformly on any compact set of $\RP^1\setminus\{q\}$. In particular, if each $f_n$ is normalized, there exists a subsequence converging in the uniform topology to a normalized $k$-quasisymmetric homemorphism.
\end{lemma}

Quasisymmetric homeomorphims may also be characterized in terms of cross ratios.
The cross-ratio of four points $(a,b,c,d) \in (\RP^1)^4$ in general position is defined by the formula
\begin{align*}
\cro(a,b,c,d) := \frac{(c-a)(d-b)}{(b-a)(d-c)}
\end{align*}
so that in particular $\cro(0,1,y,\infty) = y$ holds for all $y$. It is well known that the cross-ratio is invariant under the diagonal action of $\PSL(2,\RR)$ on $(\RP^1)^4$.
A quadruple of points $Q = (a,b,c,d)$ is called symmetric if $\cro(Q) = -1$, or equivalently if there exists $g \in \PSL(2,\RR)$ so that $g(Q) = (0,1,-1,\infty)$. The following is well-known, see e.g.~\cite{FlMa}.

\begin{prop}\label{qs_cross}
For any $k \geq 1$, there exists $M \geq 1$ so that if $f: \RP^1 \to \RP^1$ is $k$-quasisymmetric then
\begin{align}\label{eqn:M}
-M & \leq \cro f(Q) \leq -1/M
\end{align}
holds for all symmetric quadruples $Q$. The constant $M$ goes to infinity with $k$. Conversely for any $M \geq 1$, there exists $k \geq 1$, so that if~\eqref{eqn:M} holds for some orientation homeomorphism $f: \RP^1 \to \RP^1$, then $f$ is $k$-quasisymmetric. The constant $k$ goes to infinity with $M$.
\end{prop}

Finally, we note that the quasisymmetric homeomorphisms of the projective line are also characterized as the boundary extensions of the quasi-isometries of the hyperbolic plane, see again~\cite{FlMa}.

\begin{prop}\label{prop:QI-extend}
Any $A$-quasi-isometry $F:\HH^{2+} \to \HH^{2+}$ extends to a $k$-quasisymmetric homeomorphism $f:\RP^1 \to \RP^1$ where the constant $k$ depends only on $A$. Any $k'$-quasisymmetric homeomorphism $f:\RP^1 \to \RP^1$ extends to an $A'$-quasi-isometry $F:\HH^{2+} \to \HH^{2+}$ where the constant $A'$ depends only on $k'$.
\end{prop}
\subsection{The Universal Teichm\"uller Space} \label{universal}

Let $\homeoqs$ be the group of quasisymmetric homeomorphisms of $\RP^1$. The {\em universal Teichm\"uller space} $\cT= \cT(\HH^{2+})$ is defined as the quotient of $\homeoqs$ by the group $\PSL(2,\R)$ of M\"obius transformations, acting by post-composition:
$$\cT = \PSL(2,\R) \backslash \homeoqs.$$
Alternatively we may (and often will) identify $\cT$ with the set of normalized quasisymmetric homeomorphisms of $\RP^1$. 

The universal Teichm\"uller space contains copies of the classical Teichm\"uller spaces. We briefly explain. Let $\Sigma$ be a closed orientable surface of genus $g \geq 2$. The Teichm\"uller space $\cT(\Sigma)$ has many guises. Let us work from the classical definition, that $\cT(\Sigma)$ is the space of all marked Riemann surface structures (i.e. complex structures) on $\Sigma$. To begin, fix one Riemann surface structure on $\Sigma$ (a basepoint of $\cT(\Sigma)$).
The universal cover $\widetilde \Sigma$ is conformally equivalent to the hyperbolic plane, so we identify $\widetilde \Sigma = \HH^{2+}$. The group of deck automorphisms of $\widetilde \Sigma$ then identifies with a \emph{Fuchsian group}, i.e. a discrete subgroup $\pi_1 \Sigma \cong \Gamma_0 < \PSL(2,\RR)$. Now, let $g: \Sigma \to X$ be a diffeomorphism to another Riemann surface. Then any conformal isomorphism $h: \widetilde X \to \HH^{2+}$ of the universal cover $\widetilde X$ of $X$ conjugates the deck group of $\widetilde X$ to a Fuchsian group $\pi_1 X \cong \Gamma < \PSL(2,\RR)$. Since $\Sigma$ is compact, $g$ is quasiconformal, hence $\widetilde g$ is quasiconformal. 
It follows that the composition $f = h \circ \widetilde g$ is a quasiconformal diffeomorphism of $\HH^{2+}$. By Proposition~\ref{prop:extend-QC}, $f$ extends uniquely to a quasisymmetric homeomorphism $\partial f: \RP^1 \to \RP^1$. Further, $\partial f$ is equivariant under the isomorphism $\pi_1 \Sigma \cong \Gamma_0 \to \Gamma \cong \pi_1 X$ of Fuchsian groups induced by $g$. We call such a quasisymmetric homeomorphism a \emph{quasifuchsian quasisymmetric homeomorphism}.
Adjusting $g$ by isotopy (leaving the Riemann surface structure $X$ fixed) does not change $\partial f$. The isomorphism $h$ is only well-defined up to post-composition with $\PSL(2,\RR)$, hence $\partial f$ is well-defined up to post-composition with $\PSL(2,\RR)$ as well. Hence each isotopy class of map $g: \Sigma \to X$ to a Riemann surface $X$ determines a well-defined element of the universal Teichm\"uller space $\cT$, represented by a quasifuchsian quasisymmetric homeomorphism $\partial f$.
In fact, this map $\cT(\Sigma) \to \cT$ is an embedding, for the simple reason that the representation $\Gamma_0 \to \Gamma$ induced by $g$ determines the map $\partial f$.

\subsection{Quasicircles in $\CP^1$}\label{sub:qc}

In this paper, we will focus on a special class of oriented Jordan curves in the complex projective line $\CP^1$, called quasicircles.
Since all the constructions that we consider are $\PSL(2,\CC)$ invariant, we will often restrict to working with oriented Jordan curves $C \subset \CP^1$ which pass through $0, 1, \infty$ in positive order. Such a Jordan curve is called \emph{normalized}.

Let $C \subset \CP^1$ be a normalized Jordan curve. The complement of $C$ consists of two regions, one called $\Omega^+_C$ on the positive side of $C$, and one called $\Omega^-_C$ on the negative side. By the Riemann mapping theorem, both $\Omega^+_C$ and $\Omega^-_C$ are conformally isomorphic to $\HH^{2+}$ and by the Caratheodory theorem \cite[Section 21]{Pommerenke} any such conformal isomorphism extends to a homeomorphism between $C$ and the boundary $\partial \HH^2 = \RP^1$. Note that, by definition, the orientation of the Jordan curve $C$ is compatible with the orientation of $\Omega^+_C$. We let $U^+_C: \HH^{2+} \to \Omega^+_C$ be the unique conformal isomorphism whose extension $\partial U^+_C: \RP^1 \to C$ satisfies $\partial U^+_C(i) = i$ for $i = 0,1,\infty$. On the other hand, we note that the orientation of the Jordan curve $C$ is not compatible with the orientation of $\Omega^-_C$. For this reason, it makes sense to identify $\Omega^-_C$ with $\HH^{2-}$ rather than $\HH^{2+}$. Let $U^-_C: \HH^{2-} \to \Omega^-_C$ be the unique conformal isomorphism whose extension $\partial U^-_C: \RP^1 \to C$ satisfies $\partial U^-_C(i) = i$ for $i = 0,1,\infty$.
 The {\em gluing map between the upper and lower regions of the complement of $C$} is $\varphi_C :=(\partial U^-_C)^{-1}\circ(\partial U^+_C)$. 
We have the following central result:

\begin{lemma}\label{lem:ahl} \cite{ahlfors-reflections}
The following properties are equivalent:
\begin{itemize}
\item $C$ is the image of $\RP^1$ under a $k$-quasiconformal homeomorphism of $\CP^1$; 
\item $U^+_C\co \HH^2 \to\CP^1$ extends to a $k$-quasiconformal map of $\CP^1$;
\item  $\varphi_C$ is $k$-quasisymmetric.
\end{itemize}
\end{lemma}

\begin{defi}
A $k$-quasicircle $C$ in $\CP^1$ is a Jordan curve that satisfies one of the equivalent conditions in Lemma \ref{lem:ahl}.
We denote by $\qcm$ the space of normalized quasicircles in $\CP^1$.
\end{defi}

Using the compactness properties of quasiconformal maps we have the following continuity result. Here we denote by $\overline{U^\pm_C}: \overline{\HH^{2\pm}} \to \overline{\HH^{2+}}$ the map which restricts to $U^\pm_C$ on $\HH^{2\pm}$ and to $\partial U^\pm_C$ on $\RP^1$.

\begin{lemma}\label{lm:int-conv}
Let $k > 1$, let $C_n$ be a sequence of normalized $k$-quasicircles, and suppose that $C_n$ converges to $C$ in the Hausdorff sense.
Then $C$ is a $k$-quasicircle and the maps $\overline{U^{\pm}_{C_n}}$ converge to $\overline{U^{\pm}_{C}}$ uniformly on the closed disk $\overline{\HH^{2\pm}}$.
\end{lemma}

\begin{proof}
First, we note that it is sufficient to prove that the statement holds on some subsequence.
By Lemma \ref{lem:ahl} the map $U^{+}_{C_n}$ extends to a normalised $k$-quasiconformal homeomorphism $g_n$ of $\CP^1$. By the normalization in the definition of $\partial U^+_C$, we have that $g_n(i) = i$ for all $i = 0,1,\infty$.
Hence, by standard results in the theory of quasiconformal mappings (see~\cite[Theorem 5.2]{Lehto}), up to extracting a subsequence, $g_n$ converges uniformly to a $k$-quasiconformal homeomorphism $g$ of $\CP^1$. Clearly $g(i) = i$ for $i = 0,1, \infty$. Since $g_n(\RP^1) = C_n$, we have $g(\RP^1) = C$. Hence $C$ is a $k$-quasicircle. 
Since $g_n$ is holomorphic on $\HH^{2+}$, the limit $g$ is as well, and so $\overline{U^+_C}$ is the restriction of $g$ to $\overline{\HH^{2+}}$.

A similar argument shows that $\overline{U^-_{C_n}}$ uniformly converges to $\overline{U^-_C}$.
\end{proof}

\begin{cor}
In the setting of Lemma \ref{lm:int-conv}, the gluing map $\varphi_{C_n}$ between the upper and lower regions of the complement of $C_n$ uniformly converges to the gluing map $\varphi_C$ between the upper and lower regions of the complement of $C$.
\end{cor}

\subsection{Hyperbolic geometry in dimension three}\label{sec:hyp-3}

For the most part, the arguments in this paper involving hyperbolic geometry are independent of any specific model of hyperbolic three-space. Nonetheless, for concreteness we introduce here a version of the projective model for hyperbolic three-space.
Consider the $2\times2$ matrices $M_2(\CC)$ with complex coefficients. 
Let $$\Herm(2, \CC) = \{ A \in M_2(\CC) \mid A^* = A\}$$ denote the $2\times 2$ Hermitian matrices, where $A^*$ is the conjugate transpose of $A$.  As a real vector space, $\Herm(2,\CC) \cong \mathbb R^4$. We define the following (real) inner product on $\Herm(2,\CC)$:
\begin{align}\label{eqn:3-1-product} \left\langle \bminimatrix{a}{z}{\bar{z}}{d}, \bminimatrix{e}{w}{\bar{w}}{h} \right\rangle = -\frac{1}{2}tr\left( \bminimatrix{a}{z}{\bar{z}}{d} \bminimatrix{h}{-w}{-\bar{w}}{e}\right).
\end{align}
We will use the coordinates on $\Herm(2,\CC)$ given by 
\begin{align} \label{coordinates-on-Herm}
X &= \bminimatrix{x_4+x_1}{x_2 - x_3 i}{x_2 + x_3i}{x_4 - x_1}.
\end{align} In these coordinates, we have that
$$\langle X, X\rangle = -\text{det}(X) = x_1^2 + x_2^2 + x_3^2 - x_4^2,$$
and we see that the signature of the inner product is $(3,1)$.

The coordinates defined in~\eqref{coordinates-on-Herm} together with the inner product~\eqref{eqn:3-1-product} naturally identify $\Herm(2,\CC)$ with the standard copy of $\RR^{3,1}$. We also identify the real projective space $\RP^3$ with the non-zero elements of $\Herm(2,\CC) = \RR^{3,1}$, considered up to multiplication by a real number. We define the three-dimensional hyperbolic space $\HH^3$ to be the region of $\RP^3$ consisting of  the negative lines with respect to $\langle \cdot, \cdot\rangle$:
$$\HH^3 = \left\{ X \in \Herm(2,\CC) \mid \langle X, X\rangle < 0 \right\} / \RR^*.$$
Note that in the affine chart $x_4 = 1$, $\HH^3$ is the standard round ball. In particular, $\HH^3$ is a properly convex subset of projective space.
There are several ways to define the hyperbolic metric $g_{\HH^3}$. The tangent space to a point $x = [X] \in \HH^3$ naturally identifies with the space $\operatorname{Hom}(x, x^\perp)$. We equip $\HH^3$ with the Riemannian metric defined by 
$$(g_{\HH^3})_x(v,w) = \left\langle v\left(\frac{X}{\|X\|}\right), w\left(\frac{X}{\|X\|}\right) \right\rangle$$
where $v,w \in \operatorname{Hom}(x, x^\perp) = T_x \HH^3$ and $\|X \| = \sqrt{-\langle X, X \rangle}$. This metric, known as the hyperbolic metric, is complete and has constant curvature equal to $-1$. 
Alternatively, the hyperboloid $\left\{ X \in \Herm(2,\CC) \mid \langle X, X\rangle = -1 \right\} $ projects two-to-one onto $\HH^3$ and the hyperbolic metric is just the push forward under this projection of the restriction of $\langle \cdot, \cdot \rangle$. Alternatively, the hyperbolic metric also agrees with (a multiple of) the Hilbert metric, defined in terms of cross-ratios, see e.g.~\cite{Benoist-survey}. From this last description, it is clear that the geodesics of $\HH^3$ are the intersections with $\HH^3$ of projective lines in $\RP^3$. The totally geodesic planes of $\HH^3$ are the intersections with $\HH^3$ of projective planes in $\RP^3$. Hence, the intrinsic notion of convexity in $\HH^3$, thought of as a Riemannian manifold, agree with the notion of convexity coming from the ambient projective space. Recall that a set $P \subset \RP^3$ is called convex if it is contained in some affine chart and it is convex there.

Next, the isometry group $\isom(\HH^3)$ is naturally the group of automorphisms of the vector space $\Herm(2,\CC)$ which preserve the bilinear form $\langle \cdot, \cdot\rangle$ up to projective equivalence, also known as the projective orthogonal group $\PO(3,1)$. The orientation-preserving subgroup $\isom^+(\HH^3)$ is the projective special orthogonal group $\PSO(3,1)$. 
However, in our coordinates, we may also describe the orientation-preserving isometries in terms of two by two complex matrices.
Indeed, an element $A \in \PSL(2,\CC)$ acts on $\Herm(2,\CC)$ by the formula
$$ A \cdot X := A X A^*.$$
This action preserves the bilinear form $\langle \cdot, \cdot \rangle$, and hence we have an embedding $\PSL(2,\CC) \to \PSO(3,1)$ which one easily checks is an isomorphism.

The visual boundary $\partial \HH^3$ of $\HH^3$ coincides with the boundary of $\HH^3$ in projective space. It is given by the null lines in $\Herm(2,\CC)$ with respect to $\langle \cdot, \cdot \rangle$. Thus $$\partial\HH^3 = \left\{ X \in \Herm(2,\CC) \mid \det(X) = 0, X \neq 0\right\}/\RR^*$$ can be thought of as the $2\times 2$ Hermitian matrices of rank one. This gives a natural identification $\partial \HH^3 = \CP^1$ since any rank one Hermitian matrix $X$ can be decomposed as 
\begin{equation} \label{eqn:decomposition}
X= v v^*,
\end{equation} where $v \in \CC^2\setminus \{0\}$ is a two-dimensional column vector unique up to multiplication by $\lambda \in \CC \setminus \{0\}$ and $v^*$ denotes the transpose conjugate.  The action of $\PSL(2, \CC)$ on $\CP^1$ by matrix multiplication extends the action of $\PSL(2,\CC)$ on $\HH^3$ described above. We note also that the metric on $\HH^3$ determines a compatible conformal structure on $\partial_{\infty} \HH^3 = \CP^1$ which agrees with the usual conformal structure on $\CC$. 

Given a subset $C \subset \CP^1$, we define its \emph{convex hull} $\CH(C)$ in $\HH^3$ to be the intersection with $\HH^3$ of the usual convex hull in (say, the affine chart $x_4 = 1$ of) projective space. Given $C \subset \CP^1$, a closed convex subset $\mathscr C \subset \HH^3$ is said to \emph{span} $C$ or to have \emph{boundary at infinity} equal to $C$, if the closure $\overline{\mathscr C}$ of $\mathscr C$ in $\RP^3$ is the union of $\mathscr C$ and $C$.

\subsection{Geometry of surfaces embedded in $\HH^3$} \label{sssc:surfaces}

Given a smooth surface $S$ embedded in $\HH^3$, recall that the restriction of the metric of $\HH^3$ to the tangent bundle of $S$ is a Riemannian metric on $S$ which is called the {\em first fundamental form}, or alternatively the \emph{induced metric}, and is denoted $\I$. Let $N$ be a unit normal vector field to $S$, and let $\triangledown$ be the Levi-Civita connection of $\HH^3$, then the {\em shape operator} $B\co T S \to T S$ is defined by $Bx = - \triangledown_x N$. 

The {\em second fundamental form} $\II$ of $S$ is defined by
$$\forall s \in S, \forall x, y \in T_s S,\;\; \II(x, y) := \I(Bx, y) = \I(x, By)~,$$
and its {\em third fundamental form} $\III$ by
$$\forall s \in S, \forall x, y \in T_s S,\;\; \III(x, y) := \I(Bx, By)~.$$ 

Given a surface $S$ immersed in a hyperbolic $3$--manifold $M$, the {\em extrinsic curvature} $K_{ext}$ is the determinant of the shape operator $B$, or equivalently, the product $\kappa_1 \kappa_2$ of the two principal curvatures of $S$. This quantity is related to the {\em intrinsic} or {\em Gaussian curvature} $K$ of the $S$ by the Gauss equation, which in hyperbolic geometry takes the form:
\begin{equation}\label{hypcurv}
  K = K_{ext}-1.
\end{equation}
A {\em $K$--surface} in a hyperbolic $3$--manifold $M$ is a surface in $M$ which has constant Gaussian curvature equal to $K$.

The shape operator $B$ of $S$ satisfies the {\em Codazzi equation}: when $B$ is a considered as a 1-form with values in $TS$, $d^DB=0$, where $D$ is the Levi-Civita connection of the induced metric $I$. If $D$ is non degenerate, a direct computation shows that, as a consequence of this Codazzi equation, the Levi-Civita connection $D^*$ of $\III$ is given, for two vector fields $u,v$ on $S$, by
$$ D^*_uv = B^{-1}D_u(Bv)~. $$
It then follows that the curvature 2-form of $D^*$ is equal to the curvature 2-form of $D$, and the curvature $K^*$ of $\III$ is equal to
\begin{equation}
  \label{eq:KK*}
  K^* = \frac{K}{K_{ext}} = \frac{K}{K+1}~. 
\end{equation}

\subsection{Polar duality between surfaces in $\HH^3$ and the de Sitter space}\label{sec:polar}

The third fundamental form can also be interpreted in terms of the {\em polar duality} between $\HH^3$ and the de Sitter space $\dS^3$. Recall that we can identify the real projective space $\RP^3$ with the non-zero elements of $\Herm(2,\CC) = \RR^{3,1}$, considered up to multiplication by a real number. We define the three-dimensional de Sitter space $\dS^3$ to be the region of $\RP^3$ consisting of the positive lines with respect to $\langle \cdot, \cdot\rangle$:
$$\dS^3 = \left\{ X \in \Herm(2,\CC) \mid \langle X, X\rangle > 0 \right\} / \RR^*.$$
The inner product $\langle \cdot, \cdot\rangle$ determines a metric on $\dS^3$, defined up to scale. We choose the metric with constant curvature $+1$.

Given a point $x = [X] \in \HH^3$, the orthogonal of the line $\R X$ in $\R^{3,1}$ is a spacelike hyperplane, which intersects $\dS^3$ along a totally geodesic spacelike plane $x^*$ of $\dS^3$, and any totally geodesic spacelike plane in $\dS^3$ is obtained uniquely in this manner. Conversely, given a point $y =[Y] \in \dS^3$, the orthogonal of the oriented line $\R Y$ is an oriented timelike hyperplane in $\R^{3,1}$, which intersects $\HH^3$ along an {\em oriented} totally geodesic plane $y^*$ in $\HH^3$, and each oriented totally geodesic plane in $\HH^3$ is dual to a unique point in $\dS^3$.

Now consider a smooth surface $S\subset \HH^3$. We can consider the `dual' set $S^*$ of points in $\dS^3$ which are dual to the oriented tangent planes of $S$. Some of the key properties of this duality map are:
\begin{itemize}
\item If $S$ is convex with positive definite second fundamental form at each point, then $S^*$ is a smooth, spacelike, convex surface, with positive definite second fundamental form at each point.
\item The pull-back by the duality map of the induced metric on $S^*$ is the third fundamental form $\III$ of $S$, and conversely. So it follows from \eqref{eq:KK*} that the dual of a $K$-surface is a $K^*$-surface.
\end{itemize}
In the same manner, given a smooth surface $S$ in $\dS^3$, we can consider the ``dual'' set $S^*$, defined as the set of points in $\HH^3$ dual to the totally geodesic planes tangent to $S$. As before we have:
\begin{itemize}
\item If $S$ is spacelike and convex  with positive definite second fundamental form at each point, then $S^*$ is a smooth, convex surface, with positive definite second fundamental form at each point.
\item The duality maps pulls back the induced metric on $S^*$ to the third fundamental form $\III$ of $S$, and conversely. So it follows from \eqref{eq:KK*} that the dual of a $K$-surface is a $K^*$-surface.
\end{itemize}
Finally, again if $S$ is a smooth surface in $\HH^3$ (resp. a spacelike smooth surface in $\dS^3$) with positive definite second fundamental form, then $S=(S^*)^*$. See Hodgson and Rivin \cite{HR} or \cite{shu,horo} for the proofs of the main points asserted here and a more detailed discussion.

\section{Gluing maps in hyperbolic geometry}
Here we carefully define the gluing maps $\Phi_{\cdot}, \Phi_{\cdot, K}\co \qcm \to \cT$ from the introduction, filling in the technical results needed for the definitions. We will also give a critical estimate needed for the proofs of Theorems~\ref{tm:induced-hyp}, and ~\ref{tm:induced-hyp-K}, (and eventually~\ref{tm:III-hyp-K}).

Recall that, given an oriented Jordan curve $C$ in $\CP^1$, the convex hull $\CH(C) \subset \HH^3$ of $C$ is the smallest closed convex subset of $\HH^3$ whose closure in $\overline{\HH^3}$ includes $C$. The boundary $\partial \CH(C)$ of $\CH(C)$ consists of two convex properly embedded disks, spanning $C$, which inherit an orientation from that of $C$. We call the component of $\partial \CH(C)$ for which the outward normal is positive the \emph{top boundary component} and denote it $\partial^+ \CH(C)$. Similarly, the other boundary component, whose outward pointing normal is negative, is called the \emph{bottom boundary component} and denoted $\partial^- \CH(C)$. Note that the surfaces $\partial^\pm \CH(C)$ are not smooth, but rather are each bent along a geodesic lamination.

In the case that $C$, and hence $\CH(C)$, is invariant under some quasifuchsian surface group $\Gamma < \PSL(2,\CC)$, then the the quotient $\Gamma \backslash \CH(C)$ is compact and is called the convex core of the quasifuchsian hyperbolic three-manifold $\Gamma \backslash \HH^3$. In this case, Labourie~\cite{L5} proved that the complement of $\Gamma \backslash \CH(C)$ in $\Gamma \backslash \HH^3$ admits a foliation by $K$--surfaces, i.e. surfaces whose Gauss curvature is constant equal to $K$. The following result of Rosenburg--Spruck generalizes that result to the context of interest here.

\begin{theorem}[Rosenberg and Spruck \cite{RS}] \label{tm:K-surfaces-hyp}
Let $C\subset \C \mathbb{P}^1$ be a Jordan curve, and let $K \in (-1,0)$. There are exactly two properly embedded $K$--surfaces in $\HH^3$ spanning $C$. These are each homeomorphic to disks, are disjoint, and bound a closed convex region $\mathscr C_K(C)$ in $\HH^3$ which contains a neighborhood of the convex hull $\CH(C)$. Further the $K$--surfaces spanning $C$, for $K\in (-1,0)$, form a foliation of $\HH^3 \setminus \CH(C)$. 
\end{theorem}

An orientation of $C$ induces an orientation of the boundary $\partial \mathscr C_K(C)$ and hence, as above for $\partial \CH(C)$, determines a \emph{top $K$-surface}, which we label $S_K^+(C)$, and a \emph{bottom $K$-surface}, which we label $S_K^-(C)$. Note that as $K \to -1^+$, $S_K^\pm(C)$ converges to the top/bottom boundaries $\partial^\pm \CH(C)$ of the convex hull $\CH(C)$. Hence, we will sometimes use the convention $S_{-1}^\pm = \partial^\pm \CH(C)$, even though these surfaces are not technically considered $K$-surfaces since they are not smooth.

For $K \in [-1,0)$, let $\HH^{2\pm}_K$ be a copy of $\HH^{2\pm}$ equipped with the conformal metric that has constant curvature equal to $K$. The induced metric on the $K$-surface $S^\pm_K(C)$ is locally isometric to $\HH^{2\pm}_K$. Since $S^\pm_K(C) \subset \HH^3$ is a properly embedded disk, its induced metric is complete, and hence is globally isometric to $\HH^{2\pm}_K$. To continue, we need the following basic proposition. The proof, which is slightly technical, will be given later in this section.

\begin{prop}\label{prop:extend}
For a Jordan curve $C$ and $K \in [-1,0)$, any isometry $V:\HH^{2\pm}_K\to S^\pm_{K}(C)$ extends to a homeomorphism of $\overline{\HH^{2\pm}_K} = \HH^{2\pm}_K \cup \RP^1$ onto $S^\pm_K(C) \cup C \subset \overline{\HH^3}$.
\end{prop}

Now, let us assume the Jordan curve $C$ is normalized, meaning it is oriented and passes through $0,1,\infty$ in positive order, and fix $K \in [-1,0)$. Then there are unique isometries $V^\pm_{C,K}: \HH^{2 \pm}_K \to S^\pm_{K}(C)$ whose extension to the boundary, given by Proposition~\ref{prop:extend}, satisfies $\partial V^\pm_{C,K}(i) = i$ for $i = 0,1,\infty$. The gluing map associated to $C$ and $K$ is simply the \emph{comparison map} between the two maps $V^+_{C,K}$ and $V^-_{C,K}$: 
\begin{equation}\label{eqn:PhiCK}
\Phi_{C, K} = \cmp(V^-_{C,K}, V^+_{C,K}) := (\partial V^-_{C,K})^{-1}  \circ \partial V^+_{C,K}.
\end{equation}

The main goals of this section, in addition to Proposition~\ref{prop:extend}, are to prove the following two statements.

\begin{prop}\label{pr:main-properness}
 Let $K \in [-1, 0)$. Then for each $k > 1$, there exists a constant $k' > 1$ depending only on $k$ and $K$, so that for any (normalized) Jordan curve $C$:
 \begin{enumerate}
 \item \label{item:well-defined}
If $C$ is a $k$--quasicircles, then $\Phi_{C, K}$ is a $k'$--quasisymmetric map. In particular $\Phi_{C, K} \in \mathcal T$.
\item \label{item:properness}
If $\Phi_{C, K}$ is $k$--quasisymmetric, then $C$ is a $k'$--quasicircle.
\end{enumerate}
\end{prop}

Statement~\ref{item:well-defined} shows that $\Phi_{\cdot, K}$ is a well-defined map taking normalized quasicircles in $\CP^1$ to the universal Teichm\"uller space $\mathcal T$.
Statement~\ref{item:properness} is a properness statement, showing that the quasisymmetric constant of $\Phi_{C, K}$ can not go to infinity unless the quasicircle constant for $C$ does. This will be a key ingredient for the proofs of Theorems~\ref{tm:induced-hyp}, \ref{tm:induced-hyp-K}, and~\ref{tm:III-hyp-K}.

\subsection{Comparison maps}\label{sec:comparison}
As the notion of comparison map, from Equation~\eqref{eqn:PhiCK} will come up again and again, let us introduce some notation and properties. 
We will often consider embeddings $f\co \overline{\HH^2} \to \overline{\HH^3}$   
restricting on $\RP^1=\partial \HH^2$ to a homeomorphism to some Jordan curve $C \subset \CP^1$. Given such an embedding, we will denote by $\partial f$ the restriction of $f$ to $\partial \HH^2$. Given two proper embeddings $f_1$ and $f_2$ whose boundary maps are both homeomorphisms from $\RP^1$ to the same Jordan curve $C_1 = C_2 = C$, the \textit{comparison map} between $f_1$ and $f_2$ is defined as
\[
\cmp(f_1, f_2)=(\partial f_1)^{-1}\circ \partial f_2 \co \RP^1 \to \RP^1.
\]
As an example, the gluing map between the upper and lower regions $\Omega^\pm_C$ of the complement of $C$ in $\CP^1$, from Section~\ref{sec:intro-hyp}, is just the comparison map $$\varphi_C = \cmp(U^-_C, U^+_C)$$
where $U^\pm_C: \HH^{2\pm} \to \Omega^\pm_C$ are the biholomorphisms whose extensions to $\partial \HH^{2\pm}$ satisfy $\partial U^\pm_C(i) = i$ for $i = 0,1,\infty$.

Clearly the comparison map is well defined and is a homeomorphism of $\RP^1$.
Moreover the following cocycle relations hold:
\[
\begin{array}{l}
\cmp(f_1, f_2)\circ\cmp(f_2, f_3)=\cmp(f_1, f_3)\,,\\
\cmp(f_1, f_2)^{-1} = \cmp(f_2, f_1), \\
\cmp(f_1, f_1)=\mathrm{Id}|_{\RP^1}\,.
\end{array}
\]

\subsection{Compactness statements following Labourie}

In this subsection, we give several useful compactness results for taking limits of $K$-surfaces. These will be proved using the following general result of Labourie about limits of surfaces in $\HH^3$:

\begin{theorem}[{Labourie~\cite[Thm D]{L1}}]\label{thm:lab-compact}
Let $f_n:S\to \HH^3$ be a sequence of immersions of a surface $S$ such that the pullback $f_n^*(h)$ of the hyperbolic metric $h$ converges smoothly to a metric $g_0$.
If the integral of the mean curvature is uniformly bounded, then a subsequence of $f_n$ converges smoothly to an isometric immersion
$f_\infty$ such that  $f_\infty^*(h)=g_0$.
\end{theorem}

\begin{remark}
Let us emphasise the local nature of  Theorem \ref{thm:lab-compact}: no global assumption, like compactness of $S$ or completeness of $f_n^*(h)$, is in fact needed.
\end{remark}

\begin{Proposition}\label{pr-compK}
Let $K \in (-1,0)$. Let $f_n:\HH^2_K\to\HH^3$ be a sequence of proper isometric embeddings. If there is a point $z\in\HH^2_K$ such that
$f_n(z)$ is contained in a compact subset of $\HH^3$ then a subsequence of $f_n$ converges smoothly on compact subsets to an isometric immersion $f:\HH^2_K\to\HH^3$. 
\end{Proposition}

A locally convex immersion $f\co S\to\HH^3$ is a \textit{convex embedding} if it is an embedding and $f(S)$ is contained in the boundary of $\CH(f(S))$. Notice that this is equivalent to asking that there is a convex set $\mathscr C$ such that $f$ takes values on the boundary of $\mathscr C$.
In particular if $f$ is a proper embedding, then it is convex if and only if it bounds a convex region of $\HH^3$.
We have that any proper locally convex embedding is in fact a convex embedding, and the restriction of a convex embedding to an open subset is still a convex embeddding.
Finally if $h\co S\to\HH^3$ is a convex embedding and $N$ is the normal pointing towards the concave side, then the map $h_t\co S\to\HH^3$, $h_t(x)=\exp_{h(x)}(tN(x))$  is a convex embedding.

To deduce Proposition~\ref{pr-compK} from Labourie's result, we have the following simple remark:

\begin{lemma}\label{menac:lm}
Let $h\co S\to\HH^3$ be a convex embedding and $R$ be the extrinsic diameter of $h(S)$. Denote by $H$ the mean curvature and by $da$ the area form induced by $h$. Then we have 
\[
   \int_{S}Hda<\frac{1}{\sh 1}A(R+1),
\]
where $A(\rho)$ denotes the area of the sphere of radius $\rho$ in the hyperbolic space.
\end{lemma}

\begin{proof}
First notice that the area element of the embedding $h_{t}\co S\to\HH^3$ is given by
\[
  (\ch^ 2 t+\sh t\ch t H +\sh^2 t)da>\sh(2t)Hda
\]
where $K_{ext}$ is the extrinsic curvature of the emebedding $h$.
So we have that
\[
    \int_{S}Hda<\frac{1}{\sh 1}\textrm{Area}(h_{1/2}(S))\,.
\]
On the other hand $\mathrm{diam}(h_{1/2}(S))\leq R+1$. Thus there is a closed ball $B$ of radius $R+1$ containing $h_{1/2}(S)$.
Consider now the retraction $r:B\to \CH(h_{1/2}(S))$. It is a $1$-Lipschitz map that restricts to a surjective map $\partial B\to\partial \CH(h_{1/2}(S))$.
Since $h_{1/2}(S)$ is contained in the boundary of its convex core, then its area is smaller than the area $\partial B$.
\end{proof}

\begin{proof}[Proof of Proposition~\ref{pr-compK}]
First we prove that $f_n$ uniformly converge on compact sets of $\HH^2_K$.
Gauss equation implies that the map $f_n$ is locally convex, so by properness assumption  $f_n$ is in fact a convex embedding.
Let $U$ be a bounded open subset in $\HH^2_K$ with diameter $R$. Notice that $f_n$ restricts to a convex isometric embedding of $U$.
In particular the extrinsic diameter of $U$ is bounded by $R$. By Lemma \ref{menac:lm} then the integral of the mean curvature of $f_n$ over $U$ 
is uniformly bounded.
By Theorem~\ref{thm:lab-compact} we conclude that, up to a subsequence, $f_n$ converges over $U$ to an isometric immersion.
\end{proof}

Here is a basic application of Proposition~\ref{pr-compK} that will be useful later on in the section.
\begin{Proposition}\label{principal_curv}
Let $K\in(-1,0)$.
 There is a constant $N=N(K)$ so that for any oriented Jordan curve $C$ in $\CP^1$, the principal curvatures of the $K$--surfaces $S^\pm_{K}(C)$ are contained in the interval $[1/N,N]$. Hence, the third fundamental form on $S^\pm_{K}(C)$ is complete.
\end{Proposition}

\begin{proof}
As the product of principal curvatures is equal to $1+K$, it is sufficient to point out an upper bound for the principal curvatures.
Assume by contradiction that there is a sequence of Jordan curves $C_n$ and a sequence of points $x_n\in S^{\pm}_{K}(C_n)$ such that one principal curvature at $x_n$ is bigger than $n$. Up to applying an isometry of $\HH^3$, we may assume that $x_n=x_0$ is a fixed point in $\HH^3$.
Fix a sequence of isometries $f_n:\HH^2_K\to S^{\pm}_{K}(C_n)$ sending a fixed point $p_0\in\HH^2_K$ to $x_0$.
Proposition~\ref{pr-compK} implies that, after taking a subsequence, $f_n$ smoothly converges to an isometric immersion. But this contradicts the fact that
the second fundamental form of $f_n$ at $x_0$ is unbounded as $n \to \infty$.
\end{proof}

\subsection{The nearest point retraction and the horospherical metric at infinity}\label{sec:npr}

Given a closed convex set $\mathscr C$ in $\HH^3$, let $\overline{ \mathscr C}$ denote its closure in $\overline{\HH^3}$ and let $\partial_\infty \mathscr C := \overline{\mathscr C} \cap \CP^1$.
It is a classical fact that a natural $1$-Lipschitz retraction is defined 
\[
   r=r_{\mathscr C}:\HH^3\to \mathscr C
\]
sending $x$ to the nearest point of $\mathscr C$. This map restricts to a  $1$-Lipschitz map $r:\HH^3\setminus\mathscr C \to\partial\mathscr C$.
Moreover $r$ extends to a retraction of $\overline{\HH^3}$ onto $\overline{\mathscr C}$: for every $x\in\CP^1$,  $r(x)$ is the intersection point of the smallest horoball centered
at $x$ which meets $\mathscr C$.
It easy to show that the retraction behaves well under limits: if a sequence of convex subsets $\mathscr C_n$ is such that $\overline{\mathscr C_n}$ converges to $\overline{\mathscr C}$ in the Hausdorff topology on closed sets of $\overline{\HH^3}$, then
$r_{\mathscr C_n}$ uniformly converges to $r_{\mathscr C}$ on $\overline{\HH^3}$.

The closed convex set $\mathscr C$ induces a natural metric, called the \emph{horospherical metric}, on $\CP^1 \setminus \partial_\infty \mathscr C$. We now recall the definition.
Let $\mathbb B$ denote the space of horospheres in $\mathbb H^3$ and let $\pi:\mathbb B\to\CP^1$ denote the natural projection sending each horosphere to its center. There is a natural section $\sigma_{\mathscr C}: \CP^1 \setminus \partial_\infty \mathscr C \to \mathbb B$ of $\pi$ which maps a point $z$ to the horosphere centered at $z$ passing through $r(z)$ (tangent to $\partial \mathscr C$). 

An important feature of $\mathbb B$ is that it naturally identifies with the total space of the fiber bundle of conformal metrics over $\CP^1$, namely the bundle whose fiber above a point $x \in \CP^1$ is the one-dimensional space of inner products on $T_x \CP^1$ in the correct conformal class.
To see this, recall that any point $x\in\HH^3$ induces a Riemannian conformal metric, called a  visual metric, on $\CP^1$.
The visual metrics induced by two different points $x_1, x_2\in\HH^3$ agree at a point $z$ if and only if $x_1$ and $x_2$ lie on the same horosphere centered at $z$.
This defines a canonical identifcation between $\pi^{-1}(z)$ and the space of metrics on $T_z \CP^1$ compatible with the conformal structur. We remark that bigger horospheres correspond to smaller conformal factors.
The section $\sigma_{\mathscr C}$ therefore determines a conformal metric, which we will denote by $\Istar = \Istar_{\mathscr C}$ on $\CP^1\setminus \partial_\infty \mathscr C$.  The \emph{Thurston metric} on the complement $\CP^1 \setminus C$ of a Jordan curve $C$ is precisely $\Istar_{\mathscr C}$ for the case $\mathscr C = \CH(C)$ is the convex hull of $C$ (see \cite{bridgeman_canary}).

\begin{remark}\label{rk:mon}
Note that if
$\mathscr C_1 \subset \mathscr C_2$, then $\Istar_{\mathscr C_1}|_{\CP^1\setminus \partial_\infty \mathscr C_2}\leq \Istar_{\mathscr C_2}$.
Conversely if two convex sets $\mathscr C_1$ and $\mathscr C_2$  share the same ideal boundary $\partial_\infty \mathscr C_1 = \partial_\infty \mathscr C_2$, then $\Istar_{\mathscr C_1}\leq \Istar_{\mathscr C_2}$
only if $\mathscr C_1\subset\mathscr C_2$. In fact each convex set $\mathscr C$ can be reconstructed as the intersection of the exterior of the horospheres of $\sigma_{\mathscr C}$.
\end{remark}

We list here some properties of the horospherical metric $I^*_{\mathscr C}$, referring to \cite{horo} for details:
\begin{lemma}\label{lm:atinfty}\cite{horo}
\begin{enumerate}
\item\label{item:s-nbhd} If $\mathscr C_s$ is the set of points at distance less than or equal to $s$ from $\mathscr C$ then $\mathscr C_s$ is a convex set and  $\Istar_{\mathscr C_s}=e^{s}\Istar_{\mathscr C}$.
\item\label{item:retraction-formula} If $\partial \mathscr C$ is of class $C^2$, then $r_{\mathscr C}:\mathbf CP^1\setminus \partial_\infty \mathscr C \to \partial \mathscr C$ is a $C^1$-diffeomorphism and $(r^{-1}_{\mathscr C})^*(\Istar_{\mathscr C})=\I+2\II +\III$.
\item\label{item:curvature} If $\partial \mathscr C$ is smooth, then the curvature of $\Istar_{\mathscr C}$ at $z \in \CP^1 \setminus \partial_\infty \mathscr C$ is $$K^*(z)=\frac{K(r_{\mathscr C}(z))}{(1+\mu_1(r_{\mathscr C}(z)))(1+\mu_2(r_{\mathscr C}(z)))},$$ where $K$ is the intrinsic curvature of $\partial \mathscr C$ and $\mu_1$ and $\mu_2$ denote the principal curvatures.
\end{enumerate}
\end{lemma}

\subsection{The nearest point retraction to $K$-surfaces and convex hulls}
Now consider a normalized Jordan curve $C$ and a value $K \in [-1,0)$. Then the nearest point retraction map $r = r_{\mathscr C_{C,K}}$ restricts to the maps $r^{\pm}_{C, K}\co\Omega^\pm_C\to S^{\pm}_{K}(C)$ on the upper and lower components $\Omega^\pm_C$ of the complement of $C$ in $\CP^1$. 
We equip each of $\Omega^\pm_C$ with the hyperbolic metric $h^\pm$ from uniformization.

\begin{prop}\label{lm:npr}
For any $K\in(-1,0)$, there is a constant $L = L(K)$ so that for any normalized Jordan curve $C$, the nearest point retraction maps $r^{\pm}_{C, K}\co\Omega^\pm_C\to S^{\pm}_{K}(C)$ are each $L$--bilipschitz taking the hyperbolic metric $h^\pm$ to the induced metric on $S^{\pm}_{C, K}$.
\end{prop}

\begin{remark}\label{rk:npr}
In the case $K = -1$, Bridgeman--Canary~\cite[Cor 1.3]{bridgeman_canary} show that the nearest point retractions $r^{\pm}_{C, K}\co\Omega^\pm_C\to S^{\pm}_{K}(C)$ are quasi-isometries with uniform constants independent of $C$.
\end{remark}

To prove the proposition, we first need a Lemma relating the conformal hyperbolic metric $h^\pm$ on $\Omega^\pm_C$ to the horospherical metric $\Istar_{\mathscr C_{K}(C)}$ on $\Omega^\pm_C$. 

\begin{lemma}\label{lem:biLip}
Let $K\in(-1,0)$. There is a constant $M=M(K)$ such that for any Jordan curve $C$, the conformal hyperbolic metric $h^\pm$ on $\Omega^\pm_C$ is $M$--bilipschitz to the horospherical metric $\Istar = \Istar_{\mathscr C_{K}(C)}$.
\end{lemma}

\begin{proof}
  The principal curvatures of $S^{\pm}_{K}(C)$ are positive and less than a uniform constant $N = N(K)$ by Lemma \ref{principal_curv}. It follows by the formula in Lemma \ref{lm:atinfty}.\eqref{item:curvature} that the curvature of $\Istar$ is bigger than $-M$ and less than $-1/M$ for a constant $M >1$ depending only on $K$. By Lemma~\ref{lm:atinfty}.\eqref{item:retraction-formula}, $\Istar$ is complete. So $h^\pm$ and $\Istar$ are two complete metrics on $\Omega^\pm_C$ with pinched negative curvature. It follows from a theorem of Yau~\cite{yau:schwarz} (see the first theorem stated in Troyanov~\cite{troyanov:schwarz}) that the identity map taking one metric into the other is bilipschitz with bi Lipschitz constant bounded in terms of the curvature bounds. In this case the constant may be taken to be equal to $M$.

\end{proof}

\begin{proof}[Proof of Proposition~\ref{lm:npr}]
By Lemma~\ref{lm:atinfty}.\eqref{item:retraction-formula}, we have that $(r_{\mathscr C_K(C)}^{-1})^*(\Istar)=\I+2\II+\III$. It follows from Lemma~\ref{principal_curv} that 
 $r_{C, K}\co(\Omega^{\pm}_C, \Istar)\to (S^{\pm}_{K}(C), \I)$ is uniformly bilipschitz with constant $B = B(K)$ depending only on $K$.
It then follows from Lemma~\ref{lem:biLip} that $r_{C, K}\co(\Omega^{\pm}_C, h^\pm)\to (S^{\pm}_{K}(C), \I)$ is also uniformly bilipschitz with constant $L = L(K)$ depending only on $K$.
\end{proof}

We now use Proposition~\ref{lm:npr} to deduce Propositions~\ref{prop:extend}. In fact, we will prove the following stronger statement:
\begin{prop}\label{prop:extend-bis}
Fix $K \in [-1,0)$. Given a Jordan curve $C$, any isometry $V\co\HH^{2\pm}_K\to S^\pm_{K}(C)$ extends to a homeomorphism of $\overline{\HH^{2\pm}_K} = \HH^{2\pm}_K \cup \RP^1$ onto $S^\pm_K(C) \cup C \subset \overline{\HH^3}$. Further, the comparison map $\cmp(V, U^\pm_C)$ is $\alpha$-quasisymmetric for some constant $\alpha = \alpha(K)$ independent of $C$.
\end{prop}

\begin{proof}[Proof of Proposition~\ref{prop:extend-bis}]
Consider first the case $K \in (-1,0)$. By Proposition~\ref{lm:npr}, the nearest point retraction maps $r^{\pm}_{C, K}\co\Omega^\pm_C\to S^{\pm}_{K}(C)$ are $L$--bilipschitz taking the conformal hyperbolic metric on $\Omega^\pm_C$ to the induced metric on $S^\pm_K(C)$.  Hence each map $V^{-1}\circ r^{\pm}_{C, K} \circ U^{\pm}_C\co\HH^2_K \to\HH^2$ is $L$--bilipschitz, so in particular it is quasiconformal with constant depending only on $K$. Thus each map admits a unique extension to a homeomorphism $\overline{\HH^2_K} \to \overline{\HH^2}$ which is $\alpha$-quasisymmetric at the ideal boundary for some constant $\alpha = \alpha(K)$. 
Similarly, if $K = -1$, by the Bridgeman--Canary result from Remark~\ref{rk:npr}, each map $r^\pm_{C, K}$ is a uniform quasi-isometry so that again  $V^{-1}\circ r^{\pm}_{C, -1} \circ U^{\pm}_C\co\HH^2 \to\HH^2$  admits a unique extension which is an $\alpha$-quasisymmetric homeomorphism with constant $\alpha$ independent of $C$.

Take a sequence $(x_n)$ in $\HH^2$ converging to $x_\infty\in\RP^1$.
Take a sequence $(y_n)$ in $\HH^2$ such that $r^\pm_{C, K}(U^\pm_C(y_n))=V(x_n)$, so that $x_n = V^{-1}\circ r^\pm_{C, K}\circ U^\pm_C(y_n)$. Since $x_n \to x_\infty$, and $V^{-1}\circ r^\pm_{C, K}\circ U^\pm_C$ is a homeomorphism $\RP^1 \to \RP^1$, we have that $y_n$ converges to some $y_\infty \in \partial \HH^2$. Hence the formula $\partial V(x_\infty) = r^\pm_{C, K}(U^\pm_C(y_\infty)) = U^\pm_C(y_\infty)$ defines the desired extension of $V$ to a homeomorphism of $\overline{\HH^{2\pm}_K} \to S^\pm_K(C) \cup C$. 
This concludes the proof of the first statement and the proof that Equation~\ref{eqn:PhiCK} well-defines a map $\Phi_{\cdot, K}\co \qcm \to \mathcal T$.

Observing that $\cmp(V, U^\pm_C) = (\partial V)^{-1} \circ r^\pm_{C, K} \circ U^\pm_C$, since the map $r^\pm_{C, K}$ is the identity on $C$, yields the second statement.
\end{proof}

Finally, we use Proposition~\ref{prop:extend-bis} to deduce Proposition~\ref{pr:main-properness}.
\begin{proof}[Proof of Proposition~\ref{pr:main-properness}]
 Simply decompose the map $\Phi_{C, K}$ as:
  \begin{align*}
\Phi_{C,K} &=  \cmp(V^-_{C, K}, V^+_{C, K})\\
&=
  \cmp(V^-_{C, K}, U^-_{C})\circ\cmp(U^-_{C}, U^+_{C})\circ(\cmp(V^+_{C, K}, U^+_{C}))^{-1}\\
  &=
  \cmp(V^-_{C, K}, U^-_{C})\circ \varphi_C \circ(\cmp(V^+_{C, K}, U^+_{C}))^{-1}.
  \end{align*}
  By Lemma~\ref{lem:ahl}, $C$ is a $k$--quasicircle if and only if $\varphi_C =\cmp(U^-_{C}, U^+_{C})$ is $k$--quasisymmetric. Hence, the basic properties of quasisymmetric maps under composition and inverse together with Proposition~\ref{prop:extend-bis} imply that $\Phi_{C, K}$ is $k'$--quasisymmetric for some constant $k'$ if and only if $\varphi_C$ is quasisymmetric for some constant $k$, and that $k'$ (resp. $k$) may be bounded in terms of $k$ (resp. $k'$) and the constant $\alpha(K)$ from~\ref{prop:extend-bis}.
\end{proof}

\section{Proofs of Theorems~\ref{tm:induced-hyp} and \ref{tm:induced-hyp-K}}

The proofs of Theorems~\ref{tm:induced-hyp} and~\ref{tm:induced-hyp-K} (and Theorem~\ref{tm:III-hyp-K} as well) will use the following general criterion for surjectivity of a map $\qcm \to \cT$.
We say that an element $t\in\cT$ is \emph{quasifuchsian} if there is a uniform lattice $\Gamma$ in $\PSL(2,\RR)$ such that $t\Gamma t^{-1}<\PSL(2,\mathbb R)$. Bonsante--Seppi~\cite{bon-sep} prove that the subset of quasifuchsian elements in $\mathcal T$ satisfies a certain density property. A slightly strengthened version of this statement (Proposition~\ref{pr:uniflim}), whose proof we defer until Section~\ref{approx}, implies the following: 

\begin{prop} \label{pr-surj}
Let $F\co\qcm \to \cT$ be a map satisfying the following conditions:
\begin{enumerate}[(i)]
\item\label{item:cont} If $(C_n)_{n\in \N}$ is a sequence of normalised $k$--quasicircles converging to a normalised $k$--quasicircle $C$, then $(F(C_n))_{n\in \N}$ 
converges uniformly to $F(C)$.
\item\label{item:prop} For any $k$, there exists $k'$ such that if $F(C)$ is a normalised $k$--quasisymmetric homeomorphism, then $C$ is a $k'$--quasicircle.
\item\label{item:qf} The image of $F$ contains all the quasifuchsian elements of $\cT$.
\end{enumerate}
Then $F$ is surjective.
\end{prop}

Proposition~\ref{pr:main-properness}.\eqref{item:properness} is exactly the statement that, for each $K \in [-1,0)$, the map $\Phi_{\cdot, K}$ satisfies Condition~\eqref{item:prop} of Proposition~\ref{pr-surj}.
For each $K \in [-1,0)$, the map $\Phi_{\cdot, K}$ also satisfies Condition~\eqref{item:qf}.

For $K= -1$, this is due to Epstein and Marden \cite{ep-ma}, see also Sullivan \cite{sullivan_travaux} and Labourie \cite{L4}. For $K \in (-1, 0)$, this due to Labourie \cite{L6}, see the discussion at the end of Section~\ref{K_hyp}. So to prove Theorems~\ref{tm:induced-hyp} and \ref{tm:induced-hyp-K}, we will show that for each $K \in [-1,0)$, the map $\Phi_{\cdot, K}$ satisfies Condition~\eqref{item:cont}. This requires some care and is the subject of the remainder of this section:

\begin{prop}\label{pr:main-continuity}
  Let $K \in [-1, 0)$ and let $(C_n)_{n\in\mathbb N}$ be a sequence of normalized $k$--quasicircles, which converges in the Hausdorff sense to $C$. Then
 the sequence  $ \Phi_{C_n, K} = \cmp(V^-_{C_n, K}, V^+_{C_n, K})$ converges uniformly to
$\Phi_{C, K} = \cmp(V^-_{C, K}, V^+_{C, K})$.
\end{prop}

Several lemmas are needed. The following lemma is a compactness statement for bending maps which is the analog of Proposition \ref{pr-compK} for $K=-1$.

\begin{lemma}\label{lem:bend-conv}
 Let $(C_n)_{n\in\mathbb N}$ be a sequence of normalized $k$-quasicircles converging in the Hausdorff topology to a Jordan curve $C$.
Let $V_n:\mathbb H^2\to \partial^{\pm}\CH(C_n)$ be a sequence of convex isometric embeddings. 
Assume that there is a bounded sequence $\{x_n\}$ in $\mathbb H^2$  such that $V_n(x_n)$  is contained in a compact subset of
$\mathbb H^3$ then, up to passing to a subsequence, $V_n$ converges uniformly on compact subsets to an isometric embedding.
\end{lemma}

The proof of this lemma is based on the relation between the Hausdorff convergence of convex sets and the Gromov-Hausdorff convergences of the path metrics on the boundaries. While this fact is well known to experts we provide a short proof for the sake of completeness. 

\begin{lemma}\label{lm:HvsGH}
  Let $\mathscr C_n$ be a sequence of convex subsets in $\overline{\HH^3}$ converging to $\mathscr C$ in the Hausdorff topology of $\overline{\HH^3}$.  
 Let $S_n$ and $S$ denote, respectively, the boundary in $\HH^3$ of $\mathscr C_n$ and $\mathscr C$, and denote by $d_n$ and $d$ the corresponding path distances. If $x_n, y_n\in S_n$ are sequences of points converging to $x, y$ respectively we have $d(x,y)=\lim d_n(x_n, y_n)$.
\end{lemma}

\begin{proof}
For each $n$, let $\alpha_n\co I_n\to S_n$ be an arc-length geodesic on $S_n$ joining $x_n$ to $y_n$. Notice that, as a function with values in $\HH^ 3$, $\alpha_n$ is $1$--Lipschitz, so, up to extracting a subsequence, we may assume $\alpha_n$ converges as $n \to \infty$ to a path $\alpha\co I\to S$. By the well-known properties of the length of curves, we have that $\ell(\alpha)\leq\liminf\ell(\alpha_n)$. This implies that, in general, if  $x_n\to x$ and $y_n\to y$, then $d(x, y)\leq \liminf d(x_n, y_n)$.

To prove the other inequality, fix $\epsilon>0$. We claim there is  a path $\gamma\co [0,1]\to\mathbb H^3\setminus\mathscr C$ such that $d_{\mathbb H^3}(x, \gamma(0))<\epsilon, d_{\mathbb H^3}(y, \gamma(1))<\epsilon$, 
$\ell(\gamma)<d(x,y)+\epsilon$, and $\gamma(0)$ (resp. $\gamma(1)$) is contained in the exterior half-space bounded by any support plane of $\mathscr C$ at $x$ (resp. at $y$). 

To prove the claim first consider the case where the intrinsic geodesic $\gamma_0$ joining $x$ to $y$ is contained in some coordinate chart $(U,(u_1,u_2,u_3))$ such that $S\cap U$ is the graph of a function $u_3=f(u_1,u_2)$ and $\{u_3>f(u_1,u_2)\}=\mathscr C\cap U$. We moreover suppose that the vector field $\frac{\partial\,}{\partial u_3}$ points inwards with respect to any support plane at both $x$ and $y$.
The arc $\gamma_0$ is Lipschitz so in coordinates we have $\gamma_0(t)=(u_1(t), u_2(t), u_3(t))$ with $u_i$ differentiable for a.e. $t$ and $|u_i'(t)|<C$ for some constant $C>0$.
Let us consider the arc $\gamma_s:[0,1]\to\mathbb H^3\setminus\mathscr C$ given by $\gamma_s(t)=(u_1(t), u_2(t), u_3(t)-s)$. Notice that at differentiable points of $\gamma_0(t)$, the speed of $\gamma_s(t)$ is given by $||\gamma_s'(t)||=\sqrt{\sum_{ij}g_{ij}(\gamma_s(t))u_i'(t)u_j'(t)}$, where $g_{ij}$ is the local expression of the hyperbolic metric. It turns out that $||\gamma_s'(t)||\to||\gamma_0'(t)||$ a.e. and as those functions are uniformly bounded by the Dominated Convergence Theorem we conclude 
that the length of $\gamma_s$ converges to the length of $\gamma_0$ as $s\to+\infty$.
Moreover the assumption on $\frac{\partial\,}{\partial u_3}$ at $x$ and $y$ implies that $\gamma_s(0)$ and $\gamma_s(1)$ are contained in the open cone $\mathcal C_x$ formed by the intersection of the exterior half-planes bounded by support planes at $x$ and $y$. So we can take $\gamma=\gamma_s$ for $s$ sufficiently small.

In the general case one subdivides the geodesic $\alpha$ joining $x$ to $y$ to some arcs $\alpha_1,\ldots, \alpha_N$ such that each $\alpha_i$ is contained in some chart as above. Then for each $i$ we have an arc $\gamma_i$ which satisfies the 
stated condition for $\epsilon/2N$. As  $\gamma_i(1)$ to $\gamma_{i+1}(0)$ are contained in the convex cone  $\mathcal C_x$, then the arc obtained by gluing the arcs $\alpha_i$ with the segments $[\gamma_i(1), \gamma_{i+1}(0)]$ is contained in the  concave side and its length is smaller than $d(x,y)+\epsilon$.

Now by compactness of $\gamma$, there is $n_0\in \mathbb{N}$ such that $\gamma$ is contained in $\mathbb H^3\setminus\mathscr C_n$ for $n\geq n_0$.
Up to taking a bigger $n_0$ if necessary, we may moreover assume that $d_{\mathbb H^3}(x_n,\gamma(0))<2\epsilon$ and $d_{\mathbb H^3}(y_n,\gamma(1))<2\epsilon$.
Moreover as support planes at $x_n$ converge to  support planes at $x$,  the segment $[x_n, \gamma(0)]$ is contained in $\overline{\HH^3\setminus\mathscr C_n}$ for $n$ sufficiently big. 
A similar statement  holds for the segments $[\gamma(1), y_n]$.

Then for $n$ big enough, the path $\gamma'$ obtained by gluing the geodesic segments $[x_n, \gamma(0)]$ and $[\gamma(1), y_n]$ to $\gamma$ is contained in $\overline{\HH^3\setminus\mathscr C_n}$ and connects $x_n$ to $y_n$.
Thus for $n\geq n_0$ we have $d_n(x_n, y_n)\leq\ell(\gamma')=4\epsilon+\ell(\gamma)<d(x,y)+5\epsilon$. The proof easily follows.
 \end{proof}

\begin{proof}[Proof of Lemma \ref{lem:bend-conv}]
We notice that the maps $V_n$ are $1$--Lipschitz with respect to the distance of $\HH^3$. As $V_n(x_n)$ is contained in a compact subset of $\HH^3$, we deduce that $V_n$ converges up to a subsequence to a map $V\co\HH^2\to\HH^3$, and the convergence is uniform on compact subsets of $\HH^2$.
Clearly the map $V$ takes value on $\partial^\pm\CH(C)$.

If $x,y\in\HH^2$ we have by definition
\[
    d_{\HH^2}(x,y)=d_n(V_n(x), V_n(y))\ \forall n\in\mathbb N\,,
\]
where $d_n$ denotes the intrinsic path distance on $\partial^\pm\CH(C_n)$. On the other hand, Lemma \ref{lm:HvsGH} states that $d_n(V_n(x), V_n(y))\to d(V(x), V(y))$ as $n\to+\infty$, where $d$ is the intrinsic path distance on $\partial^\pm\CH(C)$. So we deduce that $V\co\HH^2\to\partial^\pm\CH(C)$ is a distance preserving map. As $\HH^2$ is complete, $V$ turns out to be surjective, so it is a global isometry, as we wanted to prove.
\end{proof}

The next lemma will be useful to prove a Hausdorff convergence result (Lemma \ref{lm:convS_K}) for the $K$-surfaces spanning a converging sequence of Jordan curves.

\begin{lemma}\label{lm:within}
For any $K\in(-1,0)$, there exists $R = R(K)$ such that
for any Jordan curve $C$, the surface $S^{\pm}_{K}(C)$ is contained within the $R$-neighborhood of $\CH(C)$. 
\end{lemma}
\begin{proof}
As in Section \ref{sec:npr}, let $\Istar = \Istar_{\mathscr C_K(C)}$ denotes the horospherical metric induced by $\mathscr C_K(C)$ on $\Omega^{\pm}_{K}(C)$.
Then by Lemma~\ref{lem:biLip}, $\I^*\leq M h^\pm$, where $h^\pm$ is the complete hyperbolic metric in the conformal class of $\Omega^\pm$.
On the other hand, by Theorem 2.1 and Lemma 3.1 of \cite{herron},  $h^\pm <2 h^\pm_{Th}$, where $h_{Th}$ is the Thurston metric of $\Omega^{\pm}_{K}(C)$. Then the result follows from Remark \ref{rk:mon} and  Lemma \ref{lm:atinfty}.\eqref{item:s-nbhd}.
\end{proof}

\begin{lemma}\label{lm:convS_K}
Fix  $K\in[-1,0)$. Let $C_n$ be a sequence of Jordan curves converging to a Jordan curve $C$.
Then $\overline{S^{\pm}_{K}(C_n)}$ converges to $\overline{S^{\pm}_{K}(C)}$ in the Hausdorff topology of 
$\overline{\HH^3}$.
Further, if $V_n:\HH^{2\pm}_K\to S^{\pm}_{K}(C_n)$ is a sequence of isometries
sending a fixed point $x_0 \in \HH^{2\pm}$ into a bounded set of $\HH^3$, then, up to extracting a subsequence, $V_n$ converges smoothly on compact subsets of $\HH^{2\pm}_K$ to an isometry
$V:\HH^{2\pm}_K\to S^{\pm}_{K}(C)$.
\end{lemma}
\begin{proof}
In the case $K = -1$, $S^\pm_K(C_n) = \CH(C_n)$ is the convex hull of $C_n$. It is clear that $\overline{\CH(C_n)}$ converges to $\overline{\CH(C)}$ in the Hausdorff topology of $\overline{\HH^3}$ (for example by considering the projective model of $\HH^3$).
The convergence statement $V_n \to V$ in this case follows from Lemma~\ref{lem:bend-conv}.

Assume $K \in (-1,0)$. 
Let us argue in the case $\pm = +$, since the other case is the same.
We may assume without loss of generality that the closure in $\overline{\HH^3}$ of the convex domain bounded by $S^{+}_{K}(C_n)$ converges towards
a convex domain $\mathscr K_\infty$. Clearly $\Omega^{-}_{C}\cup C$ is contained in $\mathscr K_\infty$.  On the other hand 
Lemma \ref{lm:within} implies that no point of $\Omega^{+}_{C}$ is contained in $\mathscr K_\infty$. 
So the boundary of $\mathscr K_\infty \cap \HH^3$ is a convex topological disk $S_\infty$ that spans $C$.
Consider a sequence of isometric embeddings $V_n:\HH^{\pm}_K\to S^{\pm}_{K}(C_n)$ normalized so that $V_n(x_0)$ is converging in $\HH^3$.
By Proposition~\ref{pr-compK}, up to extracting a subsequence, we may assume that
the sequence of maps $V_n$ converges to a local isometric immersion $V_\infty:\HH^{2\pm}_K\to \HH^3$.
Clearly $V_\infty$ takes values in $S_\infty$. As $\HH^{2\pm}_K$ is complete the map $V_\infty$ is a covering. As $S_\infty$ is simply connected, $V_\infty$ must be an isometry. So $S_\infty$ is a $K$--surface spanning $C$ and both statements of the lemma are proved.
\end{proof}

In order to prove Proposition~\ref{pr:main-continuity}, we must show that for Jordan curves $C_n$ converging in the Hausdorff topology to a Jordan curve $C$, the comparison maps $\Phi_{C_n, K} = \cmp(V^-_{C_n, K}, V^+_{C_n, K})$ converge uniformly to the comparison map $\Phi_{C, K} = \cmp(V^-_{C, K}, V^+_{C, K})$. While Lemma~\ref{lm:convS_K} gives a tool for showing that that the maps $V^\pm_{C_n, K}$ converge to $V^\pm_{C, K}$, uniform convergence on compact sets is a priori not enough to control convergence of the extensions $\partial V^\pm_{C_n, K}$ to $\RP^1$. Instead of dealing with this directly, we will again, as in the proof of Proposition~\ref{pr:main-properness}, use the nearest point retraction map to control the terms of the factorization $\Phi_{C_n,K} =  \cmp(V^-_{C_n, K}, U^-_{C_n})\circ\cmp(U^-_{C_n}, U^+_{C_n})\circ(\cmp(V^+_{C_n, K}, U^+_{C_n}))^{-1}$.

\begin{lemma}\label{lm:tt}
Under the hypothesis of Proposition \ref{pr:main-continuity}, the maps 
$\cmp(V^\pm_{C_n, K}, U^{\pm}_{C_n})$  uniformly  converge
to $\cmp(V^\pm_{C, K}, U^{\pm}_{C})$.
\end{lemma}

In the proof of this Lemma we will use two well-known facts about quasi-isometries of the hyperbolic plane.
We provide a short proof of these facts for the sake of completeness.

\begin{lemma}\label{lm:int-vs-bound}
Let $f_n: \HH^2 \to \HH^2$ be a sequence of $A$--quasi-isometries converging pointwise to $f: \HH^2 \to \HH^2$ (in fact the limit $f$ need only be defined on a dense subset). 
Then $f$ is an $A$-quasi-isometry and $\partial f_n$ uniformly converges to $\partial f$.
\end{lemma}

\begin{proof}
Passing to the limit in the sequence of estimates
\[
      \frac{1}{A} d(x, y)-A < d(f_n(x), f_n(y)) < A d(x,y)+A,
\]
we get that $f$ is an $A$--quasi-isometry.

In order to prove the second statement, let $\ell = [x_0, \zeta)$ be a geodesic ray in $\HH^2$ with final point $\zeta\in\RP^1$ and initial point $x_0\in\HH^2$. 
 By the Morse lemma, $f_n(\ell)$ is contained in
a $D$--neighborhood of the geodesic ray $[f_n(x_0), \partial f_n(\zeta))$ for some $D$ depending only on the quasi-isometry constant $A$.  Up to a subsequence assume that $\partial f_n(\zeta)$ converges to $\eta\in\RP^1$. Then, passing to the limit, every point of $f(\ell)$ is within distance $D$ from the ray joining $f(x_0)$ to $\eta$.  Thus $\eta = \partial f(\zeta)$. 
This shows that $\partial f_n(\zeta)$ converges to $\partial f(\zeta)$ for any $\zeta$.

Uniform control of the convergence $\partial f_n(\zeta) \to \partial f(\zeta)$ as $\zeta$ varies in $\RP^1$ is obtained as follows. Without loss in generality we may assume $f_n(x_0) = f(x_0)$ is constant. 
Suppose uniform convergence fails. Then there exists points $\zeta_n \in \RP^1$ so that the ray from $[f(x_0), \partial f_n(\zeta_n))$ makes a positive angle bounded away from zero with the ray $[f(x_0), \partial f(\zeta_n))$. 
Let $x_n$ be the point along $[x_0, \zeta_n)$ at distance $R$ from $x_0$, and pass to a subsequence so that $x_n \to x$. 
If $d(x_n, x) < \epsilon$, then the distance $d(f_n(x), f(x))$ is bounded below by $d(f_n(x_n), f(x_n)) - 2A - 2 \epsilon A$ . 
However, by the Morse Lemma, $f_n(x_n)$ lies in a uniform neighborhood of the ray $[f(x_0), f_n(\zeta_n))$ and $f(x_n)$ lies in a uniform neighborhood of  $[f(x_0), f(\zeta_n))$. If $R > 0$ is large enough, then $f_n(x_n)$ and $f(x_n)$ are each far away from $f(x_0)$ and hence are far from each other since the rays  $[f(x_0), \partial f_n(\zeta_n))$ and $[f(x_0), \partial f(\zeta_n))$ make a positive angle bounded away from zero. This contradicts that $f_n(x)$ converges to $f(x)$.
\end{proof}

For the following lemma, recall that a quasi-isometry $f: \HH^2 \to \HH^2$ is called normalized if its extension to the boundary $\partial f: \RP^1 \to \RP^1$ satisfies that $\partial f(i) = i$ for $i = 0,1,\infty$.
\begin{lemma}\label{lm:norm}
For any constant $A>1$ and for any $x\in\mathbb H^2$ there exists a compact region $Q$ of $\;\HH^2$ such that if $f$ is a normalized 
$A$-quasi-isometry of $\HH^2$, then $f(x)\in Q$.
\end{lemma}
\begin{proof}
For any edge $l$ of the ideal triangle $T$ with vertices $0,1,\infty$ pick a point $x_l$.
By the Morse Lemma, as $\partial f(i)=i$, we have that $d(f(x_l), l)<D$ for some $D$ depending only on $A$.
Then we have $d(f(x), l)\leq d(f(x), f(x_l))+D\leq A\ d(x,x_l)+A+D =: M$.
Hence $f(x)$ is contained in $Q=\{w\in\HH^2| d(w, l)\leq M, \textrm{ for every edge } l \textrm{ of }T\}$, which is compact.
\end{proof}

\begin{proof}[Proof of Proposition~\ref{pr:main-continuity}]
Let $r_n^\pm = r_{C_n, K}\co\Omega^\pm_{C_n}\to S^{\pm}_{K}(C_n)$ denote the nearest-point retraction map from Section~\ref{sec:npr}.
By Lemma \ref{lm:npr} or Remark \ref{rk:npr}, the maps 
$f_n:=(V^\pm_{C_n, K})^{-1}\circ r^\pm_n\circ U^{\pm}_{C_n}:\HH^2_K\to\HH^2_K$ are normalized $A$--quasi-isometries for some constant $A=A(K)$. The extension to the boundary $\partial f_n = \cmp(V^\pm_{C_n, K}, U^{\pm}_{C_n})\co \RP^1 \to \mathbb \RP^1$ 
is, by Proposition~\ref{prop:QI-extend}, a normalized $H$-quasisymmetric homeomorphism for some fixed value of $H$.
By Lemma \ref{lm:int-vs-bound}, it suffices to show that $f_n$ converges pointwise to  
the map $(V^\pm_{C, K})^{-1}\circ r^\pm \circ U^{\pm}_{C}:\HH^2\to\HH^2$, where $r^\pm$ is the nearest-point retraction on
the convex set bounded by $S^{\pm}_{K}(C)$.

We first prove that, up to taking a subsequence,  $V^\pm_{C_n, K}$ converges to an isometry $V:\HH^2_K\to S^{\pm}_{K}(C)$.
Fix $x\in\HH^2$ and consider the sequence $x_n=f_n(x)$. By Lemma \ref{lm:norm} it is a bounded sequence in $\HH^2$ and $V^\pm_{C_n, K}(x_n)=r_n^\pm \circ U^{\pm}_{C_n}(x)$ which converges to $r\circ U^{\pm}_C(x)$. By Lemma \ref{lm:convS_K}, we may pass to a subsequence so that $V^\pm_{C_n, K}$ converges to an isometry $V:\HH^2_K\to S^{\pm}_{K}(C)$ uniformly on compact subsets.

Taking the limit of the identity $V^\pm_{C_n, K}\circ f_n(x)=r_n\circ U^{\pm}_{C_n}(x)$ and using that $f_n(x)$ is bounded in $\HH^2$, 
we get that $f_n(x)$  pointwise converges to $V^{-1}\circ r\circ U^{\pm}_{C}(x)$. By Lemma \ref{lm:int-vs-bound} the boundary map
$\partial f_n$ converges uniformly to $\partial V^{-1}\circ r\circ c^{\pm}_{C}$.
Hence $\partial V^{-1}\circ r\circ U^{\pm}_{C}$ is a normalized quasisymmetric map, hence we must have $\partial V(i)=i$ for $i=0,1,\infty$, which means that $V=V^{\pm}_{C, K}$
and the proof is complete.
\end{proof}

The proof of Theorems~\ref{tm:induced-hyp} and~\ref{tm:induced-hyp-K} are now complete modulo the proof of the general surjectivity criterion Proposition~\ref{pr-surj}, to be given in Section~\ref{approx}.

\section{The third fundamental form and Theorem~\ref{tm:III-hyp-K}}

Given a normalized quasicircle $C\subset \CP^1$ and $K\in(-1,0)$, we have a natural map
\[
  \nu=\nu^{\pm}_{C, K}\co S^{\pm}_{K}(C)\to \dS^3
\]
sending $x$ to the dual point of the support plane of $S^{\pm}_{K}(C)$ at $x$. The results recalled in Section \ref{sec:polar} show that the pull-back by $\nu$ of the de Sitter metric is $\III$. It follows that $\nu$ is a spacelike immersion of constant curvature $K^*=\frac{K}{K+1}\in(-\infty, 0)$, see Section \ref{sssc:surfaces}. Moreover by the global convexity of $S^{\pm}_{K}(C)$ we have that $\nu$ is an embedding whose image we denote by $S^{*\pm}_{K}(C)$. It follows from Proposition \ref{principal_curv} (the bound on the principal curvatures of $S^\pm_K(C)$) that the map $\nu\co S^{\pm}_{K}(C)\to S^{*\pm}_{K}(C)$ is $A^*$--bilipschitz for some constant $A^*$ depending only on $K$. In particular the induced metric on $S^{*\pm}_{K}(C)$ is complete. Finally we remark that $S^{*\pm}_{K}(C)$ is properly embedded. In fact it bounds the domain made of points that are dual to planes disjoint from the convex set $\mathscr C^{\pm}_{C, K}$ bounded by $S^{\pm}_{K}(C)$ in $\HH^3$.

The following simple geometric argument shows that if $x_n\in S^{\pm}_{K}(C)$ converges to $x\in C$, then $\nu(x_n)$ converges to $x\in C$. Since $x_n$ is diverging in $S^\pm_{K}(C)$, the sequence $\nu(x_n)$ is diverging in $S^{*\pm}_{K}(C)$ as $\nu$ is bilipschitz. As $S^{*\pm}_{K}(C)$ is properly embedded, up to passing to a subsequence, we may assume that $\nu(x_n)$ converges to a point $x'$ in $\CP^1 = \partial \mathbb{dS}^3$. This implies that the support planes $P_n$ of $S^{\pm}_{K}(C)$ at $x_n$ converge to $x'$ in the Hausdorff topology of $\overline{\HH^3}$. As $x$ is a limit of a sequence of points $x_n\in P_n$, we  deduce that $x=x'$.

Now let $V^{*\pm}\co\HH^{2\pm}_{K^*}\to S^{*\pm}_{K}(C)$ be any isometry. The map $f:=(V^{\pm}_{C, K})^{-1}\circ \nu^{-1}\circ V^*\co\HH^{2\pm}_{K^*}\to \HH^{2\pm}_{K}$ is $A^*$--bilipshitz, so it extends to a $H^*$-quasisymmetric homeomorphism $\partial f: \RP^1 \to \RP^1$, where the constant $H^*$ depends only on $K$. On the other hand we have $V^{*\pm} = \nu\circ V^{\pm}_{C, K}\circ f$. Observe that $\nu$ extends to the identity over $C$, while $V^{\pm}_{C, K}$ extends to a map from $\RP^1$ to $C$. Hence $V^{*\pm}$ extends at the boundary to a homeomorphism $\partial V^{*\pm}\co \RP^1 \to C$ and moreover
\[
  \cmp( V^{\pm}_{C, K},  V^{*\pm})= \partial f ~.    
\]
is $H^*$-quasisymmetric for some constant $H^*$ depending only on $K$.

We denote by $V^{*\pm}_{C, K}$ the isometry between $\HH^{2\pm}_{K^*}$ and $S^{*\pm}_{K}(C)$ normalized so that $\partial V^{*\pm}_{C, K}(i)=i$ for $i=0,1,\infty$. We then define the dual gluing map as $$\Phi^*_{C, K}=\cmp(V^{*-}_{C, K},V^{*+}_{C, K}).$$ 
Notice that
\begin{equation}\label{gl:eq}
   \Phi^*_{C, K}=\cmp(V^{*-}_{C, K}, V^-_{C, K})\circ \Phi_{C, K}\circ \cmp(V^{+}_{C, K}, V^{*+}_{C, K})~.
\end{equation}
Since $\cmp(V^{*-}_{C, K}, V^-_{C, K})$ and $\cmp(V^{+}_{C, K}, V^{*+}_{C, K})$ are $H^*$-quasisymmetric,
$\Phi^*_{C, K}$ is quasisymmetric if and only if $\Phi_{C, K}$ is and the quasisymmetric constant for $\Phi^*_{C, K}$ is bounded in terms of the quasisymmetric constant for $\Phi_{C, K}$ independent of $C$. Hence, Proposition~\ref{pr:main-properness} implies the analogous statement in this setting:

\begin{prop}\label{pr:main-properness-III}
 Let $K \in (-1, 0)$. Then for each $k > 1$, there exists a constant $k' > 1$ depending only on $k$ and $K$, so that for any (normalized) Jordan curve $C$:
 \begin{enumerate}
 \item \label{item:well-definedIII}
If $C$ is a $k$--quasicircles, then $\Phi^*_{C, K}$ is a $k'$--quasisymmetric map. In particular $\Phi^*_{C,K} \in \mathcal T$.
\item \label{item:propernessIII}
If $\Phi^*_{C, K}$ is $k$--quasisymmetric, then $C$ is a $k'$--quasicircle.
\end{enumerate}
\end{prop}

Hence, the gluing map for the third fundamental form determines a well-defined function 
$$\Phi^*_{\cdot, K}\co \qcm \to \cT$$
which satisfies the condition~\eqref{item:prop} of the surjectivity criterion Proposition \ref{pr-surj}. In fact, $\Phi^*_{\cdot, K}$ also satisfies condition~\eqref{item:qf} of Proposition \ref{pr-surj}:

\begin{prop}\label{pr:III-qF}
 Let $K \in (-1, 0)$. Then for every quasifuchsian quasisymmetric homeomorphism $t\in \cT$ , there exists a quasicircle $C \in \qcm$ such that $\Phi^*_{C, K} = \cmp(V^{*-}_{C, K},V^{*+}_{C, K}) = t$.
\end{prop}

\begin{proof}[Proof of Proposition \ref{pr:III-qF}]
   This is a special case of \cite[Theorem 0.2]{hmcb}, which ensures that given any two metrics $h_-, h_+$ of constant curvature $K^*\in (-\infty,0)$ on a closed surface of genus at least $2$, there exists a unique quasifuchsian manifold containing a convex subset $\mathscr C$ whose boundary is the disjoint union of two smooth, strictly convex surfaces whose third fundamental forms are given by $h_-$ and $h_+$. (Note that Theorem 0.2 of \cite{hmcb} is more general since it also applies to smooth metrics of varying curvature larger than $-1$.)
\end{proof}

Hence to apply the surjectivity criterion Proposition \ref{pr-surj} to prove Theorem~\ref{tm:III-hyp-K}, we have left to show that $\Phi^*_{\cdot, K}$ satisfies Condition~\eqref{item:cont}:
\begin{prop}\label{pr:III-continuity}
  Let $K \in (-1, 0)$. Given a sequence of uniformized $k$--quasicircles  $(C_n)_{n\in \N}$ which converges, in the Hausdorff topology, to a $k$--quasicircle $C$, $(\Phi^*_{C_n, K})_{n\in \N}$  converges uniformly as $n \to \infty$ to $\Phi^*_{C, K}.$
\end{prop}

\begin{proof}
 The proof is based on the same circle of ideas as the proof of Proposition \ref{pr:main-continuity}.
 First, by Equation \eqref{gl:eq} and Proposition \ref{pr:main-continuity} it is sufficient to check that the sequence  $\cmp(V^{*\pm}_{C_n, K}, V^{\pm}_{C_n, K})$ converges uniformly to $\cmp(V^{*\pm}_{C, K}, V^{\pm}_{C, K})$.
 The Hausdorff convergence $\overline{S^\pm_K(C_n)} \to \overline{S^\pm_K(C)}$ guaranteed by Lemma \ref{lm:convS_K} implies that similarly $\overline{S^{*\pm}_{K}}(C_n)$ converges in the Hausdorff topology to $\overline{S^{*\pm}_{K}(C)}$. Moreover, letting $\nu_n = \nu^\pm_{C_n, K}\co S^\pm_{K}(C_n) \to S^{*\pm}_{K}(C_n)$ be the duality map, we have that $\nu_n\circ V^{\pm}_{C_n, K}$ converges to $\nu\circ V^{\pm}_{C, K}$ since the convergence of $S^{\pm}_{K}(C_n)$ to $S^{\pm}_{K}(C)$ implies the convergence of their support planes.
 
On the other hand, 
 $$f_n=(V^{\pm}_{C_n, K})^{-1}\circ \nu^{-1}\circ V^{*\pm}_{C_n, K}\co\HH^{2\pm}_{K^*}\to \HH^{2\pm}_{K}$$ 
is a sequence of normalized  $A^*$--bilipschitz homeomorphisms. Hence by Lemmas~\ref{lm:norm} and~\ref{lm:int-vs-bound}, we can pass to a subsequence so that $f_n$ converges to a normalized $A^*$--quasi-isometry $f$ (in fact also a bilipschitz homeomorphism) and the extensions to the boundary converge uniformly: $\partial f_n \to \partial f$. It follows that $V^{*\pm}_{C_n, K} = \nu_n\circ V^{\pm}_{C_n, K}\circ f_n$ converges pointwise to the map $V^*_\infty=\nu\circ V^{\pm}_{C, K}\circ f$. However, by~\cite[Theorem 5.6]{these}, a sequence of isometric immersions $\HH^{2}_{K^*}\to\dS^3$ which does not degenerate converges smoothly  and uniformly on compact subsets to an isometric immersion, hence $V^*_\infty$ is an isometric immersion which extends to a homeomorphism $\partial V^*_{\infty}\co \RP^1 \to C$ fixing $0,1,\infty$. We deduce that $V^*_\infty=V^{*\pm}_{C, K}$.

Finally we have
\[
   \cmp(V^{*\pm}_{C, K}, V^{\pm}_{C, K})=\partial f=\lim\partial f_n=\lim\cmp(V^{*\pm}_{C_n, K}, V^{\pm}_{C_n, K})
\]
 and we conclude.
\end{proof}

\section{Preliminaries II: Anti de Sitter geometry} \label{sc:prelim_ads}

We now turn to the Lorentzian half of the paper. In this section we give some preliminaries on anti de Sitter geometry in dimension three and Einstein geometry in dimension two. We will follow the description used by Bonsante and Seppi \cite{bon_are} and refer the interested reader to \cite{idealpolyhedra} for a definition closer to the one we used in Section \ref{sec:hyp-3} using Hermitian matrices with  entries in the algebra of pseuso-complex numbers.

\subsection{Anti de Sitter space and isometries of $\Hyp^2$} \label{prel ads}

Let $\Hyp^2$ be the hyperbolic plane, which is the unique complete, simply connected Riemannian surface without boundary of constant curvature $-1$. We denote its visual boundary by $\partial\Hyp^2$, and the Lie group of its orientation-preserving isometries by $\isom(\Hyp^2)$. The Killing form $\kappa$ on the Lie algebra $\mathfrak{isom}(\Hyp^2)$ of $\isom(\Hyp^2)$, namely
$$\kappa(v,w)=\mathrm{tr}(\mathrm{ad}(v)\circ\mathrm{ad}(w))\;\; \forall v, w \in  \mathfrak{isom}(\Hyp^2),$$
is $\mathrm{Ad}$-invariant, and so defines a bi-invariant pseudo-Riemannian metric on $\isom(\Hyp^2)$ (still denoted by $\kappa$) which has signature $(2,1)$. 

\begin{defi}
  We define the anti de Sitter space $\AdS^3$ of dimension $3$ to be the Lie group $\isom(\Hyp^2)$ endowed with the bi-invariant metric $g_{\AdS^3}=\frac{1}{8}\kappa$. 
\end{defi}

The space $\AdS^3$ is orientable, and, under the normalization above, it has constant sectional curvature equal to $-1$. Since the metric on $\AdS^3$ has signature $(2,1)$, tangent vectors are partitioned into three types: \emph{spacelike}, \emph{timelike}, or \emph{lightlike}, according to whether the inner product is positive, negative, or null, respectively. In any tangent space, the lightlike vectors form a cone that partitions the timelike vectors into two components. Locally there is a continuous map which assigns the label \emph{future-pointing} or \emph{past-pointing} to timelike vectors. The space $\AdS^3$ is \emph{time-orientable}, meaning that the labeling of timelike vectors as future or past may be done consistently over the entire manifold. We will choose the time orientation of $\AdS^3$ as follows. Let $R_t$ be a rotation of positive angle $t$ around any point $x\in\Hyp^2$ with respect to the orientation of $\Hyp^2$. Given $\gamma \in \isom(\Hyp^2)$, we require the timelike vectors which are tangent to the differentiable curve 
$$t\in[0,\epsilon)\longrightarrow R_t\circ \gamma\,,$$
to be future-pointing. Moreover, we fix an orientation of $\AdS^3$ so that if $v, w$ are linearly independent spacelike elements of the tangent space $\mathfrak{isom}(\Hyp^2)$ to the identity in $\AdS^3 = \isom(\HH^2)$, then the triple $\{v,w,[v,w]\}$ is a positive basis of  $\mathfrak{isom}(\Hyp^2)$.

By construction, the group of orientation-preserving, time-preserving isometries of $\AdS^3$ satisfies:
$$\isom(\AdS^3)\cong\isom(\Hyp^2)\times\isom(\Hyp^2)\,,$$
where the left action on $\isom(\Hyp^2)$ is given by:
$$(\alpha,\beta)\cdot \gamma=\alpha\circ\gamma\circ \beta^{-1}\,.$$

The boundary at infinity of anti de Sitter space satisfies: 
$$\partial\AdS^3 \cong \partial\Hyp^2\times \partial\Hyp^2,$$
where a sequence $\gamma_n\in\isom(\Hyp^2)$ converges to a pair $(p,q)\in \partial\Hyp^2\times \partial\Hyp^2$ if there exists a point $x\in\Hyp^2$  such that 
\begin{equation} \label{defi convergence boundary}
\gamma_n(x)\to p\qquad \gamma_n^{-1}(x)\to q\,.
\end{equation}
Note that this condition does not depend on the point $x \in \Hyp^2$ chosen. The union $\AdS^3 \cup \partial \AdS^3$ is homeomorphic to a compact solid torus. The action of  $\isom(\Hyp^2)\times\isom(\Hyp^2)$ on $\AdS^3$ continuously extends to  the product action on $\partial\Hyp^2\times \partial\Hyp^2$ as follows: if $p,q\in\partial \Hyp^2$ and $\alpha,\beta \in \isom(\Hyp^2)$, then
$$(\alpha,\beta)\cdot(p,q)=(\alpha(p),\beta(q))\,.$$ The boundary at infinity $\partial\AdS^3$ is endowed with a conformal Lorentzian structure, in such a way that the group $\isom(\AdS^3)$ acts on $\partial\AdS^3$ by conformal transformations. The null lines of $\partial\AdS^3$ correspond to the lines $\partial\Hyp^2\times\{\star\}$ and $\{\star\} \times\partial\Hyp^2$. 

Since the exponential map  at the identity for the Levi-Civita connection of the bi-invariant metric coincides with the Lie group exponential map, the geodesics through the identity are precisely the $1$--parameter subgroups of $\isom(\Hyp^2)$. In particular elliptic subgroups correspond to timelike geodesics through the identity. Using the action of the isometry group, timelike geodesics have the form:
$$L_{x,x'}=\{\gamma\in\isom(\Hyp^2):\gamma(x')=x\}\,,$$ where $x, x' \in \Hyp^2$. These geodesics are closed and have length $\pi$. In addition, with this definition, the isometry group acts on timelike geodesics as follows: if $(\alpha,\beta) \in \isom(\Hyp^2)\times\isom(\Hyp^2)$, then
\begin{equation} \label{transf rule timelike geodesic}
(\alpha,\beta)\cdot L_{x,x'}=L_{\alpha(x), \beta(x')}\,.
\end{equation}

\begin{figure}[htb]
\centering
\includegraphics[height=7cm]{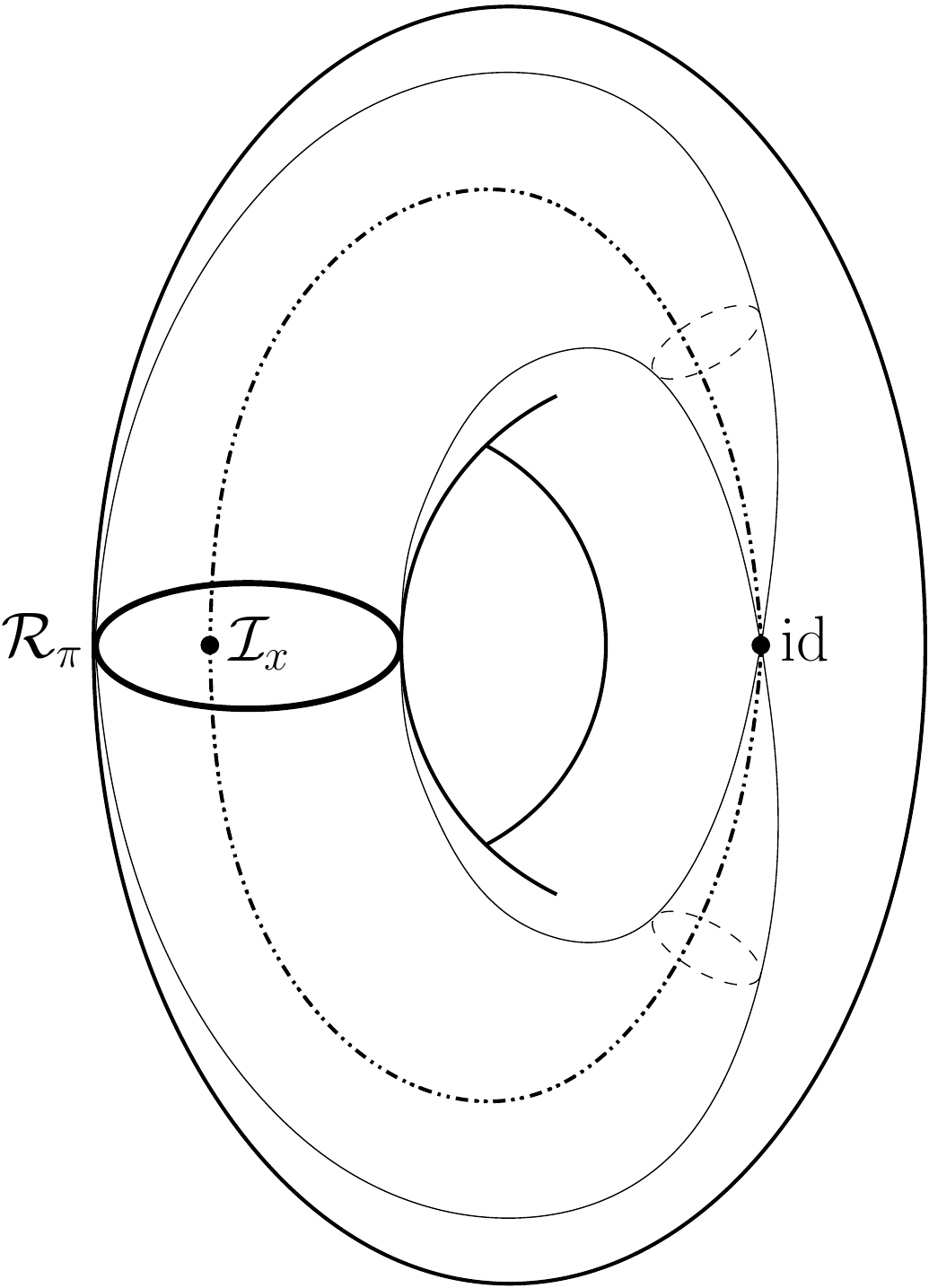}
\caption{A picture of $\AdS^3$ with the totally geodesic plane $\mathcal R_\pi$ defined by the midpoints of timelike geodesics from the identity, and its boundary at infinity corresponding to the tangency points between the double lightcone from the identity and $\partial\AdS^3$.} 
\label{fig:torus}
\end{figure}

Let $\mathcal{I}_x$ be the involutional rotation of angle $\pi$ around a point $x\in\Hyp^2$. These rotations $\mathcal{I}_x$ are the antipodal points to the identity in the geodesics $L_{x,x}$, and form a totally geodesic plane
$$\mathcal R_\pi=\{\mathcal I_x\,:\,x\in\Hyp^2\}\,.$$ 
See also Figure \ref{fig:torus}. By using the definition in Equation \eqref{defi convergence boundary}, one can check that its boundary at infinity $\partial \mathcal R_\pi$ is the diagonal in $\partial\Hyp^2\times \partial\Hyp^2$:
$$\partial \mathcal R_\pi=\{(p,p)\,:\,p\in\partial\Hyp^2\}\subset\partial\Hyp^2\times \partial\Hyp^2\,.$$
Given any point $\gamma$ of $\AdS^3$, the points which are connected to $\gamma$ by two timelike segments of length $\pi/2$ (whose union form a closed timelike geodesic) form a totally geodesic plane, called the \emph{dual} plane and denoted $\gamma^\perp$. For example, $\mathcal R_\pi=(\mathrm{id})^\perp$ and, more generally, if $(\alpha,\beta) \in \isom(\AdS^3)$ sends $\mathrm{id}$ to $\gamma$, then $\gamma^\perp=(\alpha,\beta)\cdot \mathcal R_\pi$. Then $\gamma^\perp=(\gamma,1)\cdot \mathcal R_\pi=(1,\gamma^{-1})\cdot \mathcal R_\pi$. Note that the restriction of the metric to $\gamma^\perp$ is Riemannian, and so the totally geodesic plane $\gamma^\perp$ is spacelike. In fact, any totally geodesic spacelike plane arises as the dual plane $\gamma^\perp$ to some point $\gamma \in \AdS^3$. Note also that the boundary at infinity of a totally geodesic spacelike plane $\gamma^\perp$ corresponds to the graph of the homeomorphism of $\partial\Hyp^2$ induced by $\gamma^{-1}$.

\subsection{The projective model}\label{projec}

In this section we will describe a more concrete realization of $\AdS^3$, by considering the upper-half space model of $\HH^2$: identify $\HH^2$ with
$$\Hyp^2=\{z\in\C:\im(z)>0\}\,,$$
endowed with the Riemannian metric ${|dz|^2}/{\im(z)^2}$. This metric makes every biholomorphism of the upper half-plane an isometry, so in this model $\isom(\Hyp^2)$ is naturally identified to $\PSL(2,\R)$ and the visual boundary $\partial\Hyp^2$ is identified with the real projective line $\RP^1$. 
The  Lie algebra $\mathfrak{sl}(2,\R)$ is the vector space of $2$-by-$2$ matrices with real entries and zero trace. In this model the anti de Sitter metric $g_{\AdS^3}$ at the identity is given by:
\begin{equation}\label{eq:minnie}(g_{\AdS^3})_{\mathrm{id}}( m,m')=\frac{1}{2}\tr(mm')\;\;\forall m, m' \in \mathfrak{sl}(2,\R).\end{equation}

On the space $\mathcal{M}_2(\R)$ of real $2$-by-$2$ matrices consider the quadratic form $q(M) = -\det M$. Its polarization, denoted $\langle \cdot, \cdot \rangle$, has signature $(2,2)$, providing an identification $\mathcal{M}_2(\R) = \R^{2,2}$. The restriction of
$\langle \cdot, \cdot \rangle$ to $\SL(2,\mathbb R)$ corresponds exactly to the double cover of the metric $g_{\AdS^3}$. In fact, left and right multiplication by elements in $\SL(2,\R)$ preserve $q$, so that the restriction of $\langle \cdot, \cdot \rangle$ to $\SL(2,\R)$ is a bi-invariant metric. In addition, at the identity Equation \eqref{eq:minnie} shows that it coincides with $g_{\AdS^3}$.
So we can identify $\AdS^3$ with the projective special linear group
$$\PSL(2,\R) = \{A\in \mathcal{M}_2(\R) \mid q(A)=-1\}/\{\pm 1\}$$
endowed with the pseudo-Riemannian metric which descends from $\langle \cdot, \cdot \rangle$. Note that the non-zero elements of the vector space $\mathcal{M}_2(\R)$ of $2 \times 2$ real matrices considered up to multiplication by a real number can be identified with the projective space $\RP^3$. So the space $\PSL(2,\RR)$ is naturally embedded in $\RP^3$. It will be sometimes useful to work in the coordinates,
\begin{equation} \label{coordinates sl2r}
(x_1,x_2,x_3,x_4)\in\R^4\longrightarrow M:=\begin{pmatrix} x_1-x_3 & -x_2+x_4 \\   x_2+x_4 & x_1+x_3  \end{pmatrix}\,,
\end{equation}
in which the quadratic form is given by 
\begin{align*}
q(M) = -\det(M) =-x_1^2-x_2^2+x_3^2+x_4^2
\end{align*}
In these coordinates $\AdS^3$ is the region of $\RP^3$ defined by $-x_1^2-x_2^2+x_3^2+x_4^2 < 0$. See Figure~\ref{fig:ruled} for a picture in the affine chart $x_4 = 1$.

Remarkably, the geometry of $\AdS^3$ is compatible with the geometry of $\RP^3$ in the sense that geodesics for the pseudo-Riemannian metric are precisely the intersections with $\PSL(2,\RR)$ of projective lines in $\RP^3$ and totally geodesic planes are the intersections with $\PSL(2,\RR)$ of projective planes, see e.g. \cite{fillastre-seppi}. Further, isometries of $\AdS^3$ are the restrictions of projective transformations and the isometry group $\isom \AdS^3$ naturally identifies with the subgroup of the projective linear transformations $\PGL(4, \RR)$ that preserves $\PSL(2,\RR)$, which is precisely the orthogonal group of the bilinear form $\langle \cdot, \cdot \rangle$, a copy of $\PO(2,2)$.

The embedding of $\AdS^3$ as an open set in $\RP^3$ naturally determines a compactification of $\AdS^3$, whose boundary at infinity is a copy of the \emph{Einstein space} $\Ein^{1,1}$, defined as the projectivized null cone of the quadratic form $q$.
In our framework $\Ein^{1,1}$ is precisely the space of $2 \times 2$ matrices which have rank one, up to scale. Such a matrix $A$ is of the form
\begin{align*}
A = \begin{bmatrix} a \\ b \end{bmatrix} \begin{bmatrix} d & -c \end{bmatrix}
\end{align*}
and is determined by its image, the point $a/b \in \RP^1$, and its kernel, the point $c/d \in \RP^1$. Hence $\Ein^{1,1}$ naturally identifies with a copy of $\RP^1 \times \RP^1$ and the action of $\PSL(2,\RR) \times \PSL(2,\RR)$ is factor-wise by M\"obius transformations.  Hence the identification of $\AdS^3 = \isom(\HH^2)$ with $\PSL(2,\RR)$ induces an identification of the ideal boundary $\partial \AdS^3$, defined above as $\partial \HH^2 \times \partial \HH^2$ with the projective space boundary  $\Ein^{1,1} \cong \RP^1 \times \RP^1$ in the obvious way. 
One pleasant feature of the projective model of $\AdS^3$ is that the product structure on $\partial \AdS^3$ is seen directly as the well-known double ruling of the hyperboloid $\Ein^{1,1}$ by projective lines. See Figure~\ref{fig:ruled}. The lines $\RP^1\times\{\star\}$ are referred to as the \emph{left ruling} and the lines $\{\star\} \times \RP^1$ are referred to as the \emph{right ruling}.
The form $\langle \cdot, \cdot \rangle$ defines a conformal Lorentzian structure on $\Ein^{1,1}$ whose lightlike directions are precisely the directions of the two rulings.
Directions in $\Ein^{1,1} = \RP^1 \times \RP^1$ for which both coordinates are increasing or both are decreasing are spacelike directions, while directions for which one coordinate is increasing and one is decreasing are timelike.

\begin{figure}[htb]
\centering
\includegraphics[height=5cm]{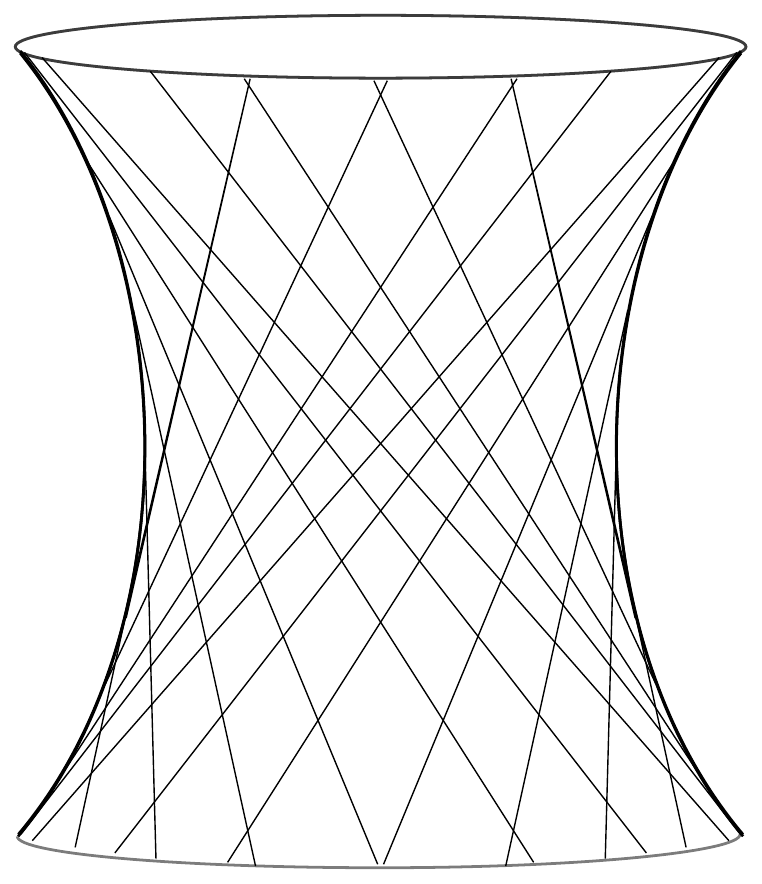}
\caption{In an affine chart, $\AdS^3$ is the interior of a one-sheeted hyperboloid. The null lines of $\partial\AdS^3$ coincide with the rulings of the hyperboloid. 
The intersection with the horizontal plane $z=0$ is a totally geodesic hyperbolic plane, in the Klein model. \label{fig:ruled}}
\end{figure}

Finally, we remark that the duality between points $\gamma$ of $\AdS^3$ and totally geodesic spacelike planes $\gamma^\perp$ described above is realized by the bilinear form $\langle \cdot, \cdot \rangle$, in the sense that $\gamma^\perp$ is precisely (the intersection with $\AdS^3$ of) the projectivization of the  subspace of $\RR^4$ defined by the equation $\langle \gamma, \cdot \rangle = 0$.
Similarly, any totally geodesic timelike plane, i.e. one whose signature is $(1,1)$, is realized as the orthogonal space to a point of projective space lying outside the closure of $\AdS^3$. A totally geodesic plane for which the restriction of the metric is degenerate is called a null plane, or a lightlike plane, and is realized as the orthogonal space to a point $(p,q)$ of $\partial \AdS^3 = \RP^1 \times \RP^1$. In this case the projective plane defined by $\langle (p,q), \cdot \rangle = 0$ is tangent to $\Ein^{1,1}$ at $(p,q)$ and the intersection of $\Ein^{1,1}$ with this projective plane, which is the boundary at infinity of $(p,q)^\perp$, is the union of a line $\{p\} \times \RP^1$ of the left ruling and a line $\RP^1 \times \{q\}$ of the right ruling.

\subsection{Achronal/acausal meridians and quasicircles in $\partial \AdS^3$}

Given a continuous curve $C$ in $\Ein^{1,1} = \partial\AdS^3$, we say that $C$ is \emph{achronal}, respectively \emph{acausal}, if for every point $p$ of $C$, 
there exists a neighborhood $U$ of $p$ in $\Ein^{1,1}$ such that $U\cap C$ is contained in the complement of the regions of $U$ which are connected to $p$ by timelike curves, respectively timelike and lightlike curves. A Jordan curve $C$ is called an \emph{achronal meridian}, respectively an \emph{acausal meridian}, if it is achronal, respectively acausal, and bounds a disk in $\AdS^3$.

\begin{lemma}[\cite{mess}]\label{lem:graph}
Any acausal meridian in $\Ein^{1,1} = \RP^1 \times \RP^1$ is the graph $\Gamma(f)$ of an orientation-preserving homeomorphism $f: \RP^1 \to \RP^1$. 
\end{lemma}

An acausal meridian $C = \Gamma(f)$ is \emph{normalized} if it contains the points $(0,0), (1,1),$ and $(\infty, \infty)$, or in other words the homeomorphism $f$ satisfies that $f(0) = 0, f(1) = 1, f(\infty) = \infty$. We call an acausal meridian $\Gamma(f)$ a \emph{quasicircle in $\Ein^{1,1}$} if $f$ is quasisymmetric. If further $f$ is $k$-quasisymmetric for some $k$, then we call $\Gamma(f)$ a $k$-quasicircle. Let $\qcm(\Ein^{1,1})$ denote the space of all normalized quasicircles $\Gamma(f)$ in $\Ein^{1,1}$. Then $\qcm(\Ein^{1,1})$ is in natural bijection with the universal Teichm\"uller space $\cT$.

The compactness statement Lemma \ref{lm:cmpqs}, for quasisymmetric maps, may be rephrased as follows to give an $\Ein^{1,1}$-analogue of the compactness statement Lemma~\ref{lm:int-conv} for quasicircles in $\CP^1$.
 \begin{lemma}\label{lm:cmpqcads}
 Let $C_n= \Gamma(f_n)$ be a sequence of $k$-quasicircles in $\Ein^{1,1} = \RP^1 \times \RP^1$. Then there is a subsequence which converges in the Hausdorff topology to either a $k$-quasicircle $\Gamma(f_\infty)$ or to the union  $\{p\} \times \RP^1 \cup \RP^1 \times \{q\}$ of a line of the left ruling and a line of the right ruling.
 \end{lemma}

\subsection{Convexity in $\AdS^3$}

We work now with the projective model (the $\PSL(2,\RR)$ model) of $\AdS^3$.
Recall that a subset of projective space is said to be \emph{convex} if it is contained and convex in some affine chart; in other words, any two points of the subset are connected inside the subset by a unique projective segment.
A subset of projective space is said to be \emph{properly convex} if its closure is convex.
Unlike the projective model of hyperbolic space $\HH^3$, the anti de Sitter space $\AdS^3$ not a convex subset of the projective space $\RP^3$, and the basic operation of taking convex hulls is not well defined.
Nonetheless, the notion of convexity in $\AdS^3$ makes sense: we shall say that a subset $\mathscr C$ of $\AdS^3$ is \emph{convex} if it is convex as a subset of $\RP^3$ or, from an intrinsic point of view, if any two points of~$\mathscr C$ are connected inside~$\mathscr C$ by a unique segment which is geodesic for the pseudo-Riemannian metric.
We shall say that $\mathscr C$ is \emph{properly convex} if its closure in $\RP^3$ is convex.
A projective plane $P$ is called a \emph{support plane} to a convex set $\mathscr C$ at a point $p \in \partial \mathscr C$ if $\mathscr C \cap P$ contains $p$ but contains no point of the interior of $\mathscr C$. In the case that $\mathscr C$ is contained in a plane, we call $P$ a support plane if it contains $\mathscr C$.

\begin{defi} \label{defi locally convex}
Let $\mathscr C$ be any convex subset of $\AdS^3$.
Then any connected region $S$ of the boundary $\partial \mathscr C$ such that the support planes to $\mathscr C$ at points of $S$ are all spacelike (resp. non timelike) is called a \emph{locally convex spacelike (resp. nowhere timelike)} surface. 
\end{defi}

Let $\mathscr C \subset \AdS^3$ be a properly convex subset and let $S \subset \partial \mathscr C$ be a locally convex spacelike  surface.
Then the outward pointing normals to the supporting hyperplanes at points of $S$ are timelike and must have consistent time orientation since we assume $S$ connected. If all of the outward normals point to the future, then we call $S$ \emph{past convex}, and if all of the outward normals point to the past, then we call $S$ \emph{future convex}.

In the case that $S$ is future convex, then $\mathscr C$ is contained in the future half-space of any support plane to $S$ (where to make sense of this, we restrict to an affine chart containing $\mathscr C$). Similarly, if $S$ is past convex, then $\mathscr C$ is contained in the past halfspace of any support plane to $S$.

We now introduce the natural convex sets associated to an acausal meridian.

\begin{prop}[\cite{BeBo}]\label{prop:AdS-CH}
Let $C \subset \Ein^{1,1} = \partial \AdS^3$ be an acausal meridian. Then:
\begin{enumerate}
\item There is a unique minimal closed properly convex subset $\CH(C) \subset \AdS^3$ which accumulates on $C$. It is called the convex hull of $C$.
\item There is a unique maximal open convex subset $E(C) \subset \AdS^3$ which contains $\CH(C)$. It is called the invisible domain (or sometimes the domain of dependence) of $C$. The invisible domain $E(C)$ is dual to the convex hull $\CH(C)$ in the sense that $x \in \CH(C)$ if and only if the dual plane $x^\perp$ is disjoint from $E(C)$ and $y \in E(C)$ if and only if the dual plane $y^\perp$ is disjoint from $\CH(C)$.
\item The boundary $\partial \CH(C)$ is the union of two disks $\partial^+ \CH(C)$ and $\partial^- \CH(C)$ each of which is a locally convex spacelike surface. One, which we call  $\partial^+ \CH(C)$ is past convex, and the other, which we call  $\partial^- \CH(C)$, is future convex. 
\item The $\AdS^3$ metric induces path metrics on $\partial^+ \CH(C)$ and $\partial^- \CH(C)$ each of which is locally isometric to the hyperbolic plane.
\end{enumerate}
\end{prop}

\subsection{Width of the convex hull}

We now give a useful criterion for an acuasal meridian $C$ in $\Ein^{1,1}$ to be a quasicircle in terms of the geometry of its convex hull. The following definition is due to Bonsante--Schlenker~\cite{maximal}.

\begin{defi}
Let $C\subset \partial \AdS^3$ be an achronal meridian. The {\em width} $w(C)$ of $C$ is the supremum of the time distance between a point of $\partial^- \CH(C)$ and $\partial^+ \CH(C)$. 
\end{defi} 

For reference, note that complete timelike geodesics in $\AdS^3$ are copies of $\RP^1$ which have timelike length $\pi$ in our normalization. Hence, the maximum time distance between two points of $\AdS^3$ is $\pi/2$, and hence the width $w(C)$ of an achronal meridian trivially satisfies $w(C) \leq \pi/2$. 
In fact, equality is realized in this bound for any achronal meridian containing a lightlike segment, see Claim 3.23 of~\cite{maximal}. Here is an important example.

\begin{example}[The rhombus curve]\label{ex:rhombus}
Let $a < a'$ and $b < b'$ be points of $\RP^1$
Let $\rhombus$ denote the piecewise linear curve consisting of four segments, two horizontal and two vertical, which connect the points $(a,b), (a',b), (a',b'), (a,b')$ in  $\Ein^{1,1} = \RP^1 \times \RP^1$ in this cyclic order using positive intervals of $\RP^1$. Note that all such curves, for varying values of $a,a',b,b'$ are equivalent by isometry of $\AdS^3$. We call $\rhombus$ the \emph{rhombus} example. The convex hull $\CH(\rhombus)$ is a tetrahedron and is easily seen to have width $\pi/2$. Indeed the spacelike line connecting $(a,b)$ to $(a',b')$ is dual to the spacelike line connecting $(a',b)$ to $(a,b')$ in the sense that any point on one is timelike distance $\pi/2$ from any point on the other. See Figure~\ref{fig:rhombus}.
\end{example} 

\begin{figure}
 
 \def\svgwidth{4.5cm}
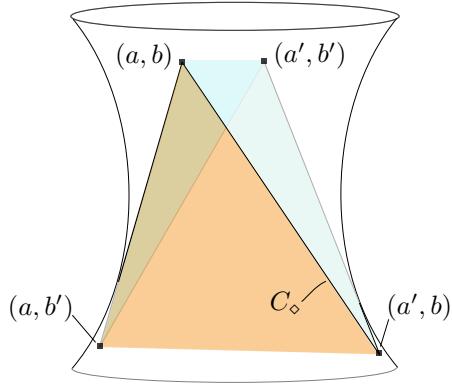

\caption{The convex hull of the achronal curve known as the \emph{rhombus} $\rhombus$ is a tetrahedron with two future oriented triangular faces (blue) and two past oriented triangular faces (orange). The tetrahedron has four edges contained in $\Ein^{1,1}$ which make up $\rhombus$ and two spacelike edges, one lying at the interface between the two future oriented triangles and one lying at the interface between the two past oriented edges. The two spacelike edges are dual.} \label{fig:rhombus}
\end{figure}

The following useful tool is a (very small) extension of Theorem 1.12 of~\cite{maximal}. 
In fact the methods of Theorem 1.12, which involve studying the relationship between the maximal surface spanning an acausal meridian and the associated minimal Lagrangian map of the hyperbolic plane, may be easily extended to prove this statement. We give a different proof here. 

\begin{prop}[Bonsante--Schlenker \cite{maximal}] \label{pr:w_ads}
Let $C\subset \partial \AdS^3$ be an acausal meridian. Then $C$ is a quasicircle if and only if $w(C)<\frac{\pi}{2}$.
Further, if $C_n$ is a sequence of quasicircles whose optimal quasisymmetric constant diverges to infinity, then there exist isometries $\phi_n \in \isom_0 \AdS^3$ so that $\phi_n C_n$ converges to the rhombus $\rhombus$, so that in particular $w(C_n) \to \pi/2$.
\end{prop}

We provide a proof for convenience, although the ideas are contained in~\cite{maximal}.

\begin{proof}
Suppose $C$ is an achronal meridian for which $w(C) = \pi/2$. 
The first thing to observe is that if the supremum in the definition of $w(C)$ is realized, then there is exactly one possibility for $C$, up to isometry, namely $C = \rhombus$ is the rhombus example (Example~\ref{ex:rhombus}). To see this, let $x \in \partial^+ \CH(C)$ and $y \in \partial^- \CH(C)$. Then $x$ is in the interior of the convex hull of points $c_1, \ldots, c_k \in C$ and $y$ is in the interior of the convex hull of points $d_1, \ldots, d_m \in C$. If $x$ and $y$ have timelike separation at distance $\pi/2$, then the dual plane $y^\perp$ to $y$ supports $\CH(C)$ at $x$ and hence $y^\perp$ contains $c_1, \ldots, c_k$. Similarly, the dual plane $x^\perp$ to $x$ contains $d_1, \ldots, d_m$. Hence $c_i$ is orthogonal to $d_j$ for all $1 \leq i \leq k$ and $1 \leq j \leq m$ and the only possibility is that $k = m =2$ and $C = \rhombus$ is a piecewise linear curve consisting of four lightlike segments connecting a $c_i$ to $d_j$ with two such segments from each of the two rulings.

If $w(C) = \pi/2$ but the supremum is not realized, then there are sequences $x_n \in \partial^+ \CH(C)$ and $y_n \in \partial^- \CH(C)$ for which the separation between $x_n$ and $y_n$ is timelike with distance converging to $\pi/2$ as $n \to \infty$. Then let $\varphi_n = (g_n, h_n) \in \PSL(2,\RR) \times \PSL(2,\RR)$ be such that $\varphi_n x_n$ and $\varphi_n y_n$ converge to  points $x_\infty$ and $y_\infty$ in $\AdS^3$ at timelike distance $\pi/2$. Then $\varphi_n C$ converges, up to taking a subsequence, to some limiting achoronal meridian $C_\infty$ whose width is $w(C_\infty) = \pi/2$ and the supremum is realized by $x_\infty$ and $y_\infty$ as in the previous paragraph, hence $C_\infty = \rhombus$ is composed of four lightlike segments. It then follows from Lemma~\ref{lm:cmpqs} that $\varphi_n C$, and hence $C$ does not have bounded quasicircle constant, and hence $C$ is not a quasicircle.

Conversely, let $f_n\co \RP^1 \to \RP^1$ be a sequence of orientation-preserving homeomorphisms whose optimal quasisymmetric constant diverges to infinity, where we allow for the case that some or all of the $f_n$ fail to be quasisymmetric (e.g. we allow $f_n = f$ to be some fixed orientation-preserving homeomorphism that is not quasisymmetric). Let $C_n = \Gamma(f_n)$ be the corresponding acausal meridians. Then there is a sequence of symmetric four-tuples of points $a_n, b_n, c_n, d_n$ (symmetric means $cr(a_n,b_n,c_n,d_n) = -1$) for which the image of the four-tuple $f(a_n), f(b_n), f(c_n), f(d_n)$  has cross-ratio (a negative number) converging to $0$ or $\infty$. For each $n$, let $g_n \in \PSL(2,\RR)$ map the points $a_n, b_n, c_n, d_n$ to $0,1,\infty, -1$ respectively, and let $h_n \in \PSL(2,\RR)$ map the points $f(a_n), f(b_n), f(c_n)$ to $0,1, \infty$. Then for each $n$, $f_n' := h_n f_n g_n^{-1}$ fixes the points $0,1,\infty$ (i.e. $f_n'$ is normalized). As $n \to \infty$, $f_n'(-1)$ converges to $0$ or $\infty$. Hence, after extracting a subsequence, the graph $C_n' = \Gamma(f_n')$ limits to an acausal meridian $C_\infty$ which contains either the lightlike interval from $(-1,0)$ to $(0,0)$ (in a line of the left ruling) or the lightlike interval from $(\infty,\infty)$ to $(-1, \infty)$ (also in a line of the left ruling). 

Next, consider any achronal meridian $C'$ (for example $C' = C_\infty$ above) which contains a maximal lightlike segment of the left ruling of the form $\{ (x,b): a \leq x \leq a'\} \subset \RP^1 \times \RP^1$ for some fixed $b \in \RP^1$. Let $G_n \in \PSL(2,\RR)$ be a hyperbolic element with repelling fixed point $b$, attracting fixed point some $b' \neq b$, and translation length going to infinity. Then $(G_n, \mathrm{id}) C'$ converges to $\rhombus$.
In particular, applying this to $C' = C_\infty$, we have a sequence $(G_n, \mathrm{id})$ so that $(G_n, \mathrm{id})C_\infty \to \rhombus$.
Finally, choosing a subsequence $(G_{k(n)}, \mathrm{id})$ as necessary, we have $(G_{k(n)}g_n,h_n) C_n \to \rhombus$. 
Since the width is a lower semicontinuous function of the achronal meridian, and $w(\rhombus) = \pi/2$, we have that $w(C_n)$ converges to $\pi/2$.
\end{proof}

We note that Proposition~\ref{pr:w_ads} does not have an analogue in hyperbolic geometry: Quasicircles in $\CP^1$ are not characterized by finiteness of the width of the convex hull in $\HH^3$. See~\cite{width-paper} for more about the behavior of the width of the convex hull in hyperbolic geometry.

\subsection{Differential geometry of surfaces embedded in $\AdS^3$}

Given a smooth spacelike immersed surface $S$ in $\AdS^3$, the \textit{first fundamental form} $\I$  on $S$ is defined as the pull-back of the anti de Sitter metric, which is Riemannian because $S$ is spacelike, while the \textit{second fundamental form} $\II$ is the normal part of the restriction of the Levi Civita connection of $\AdS^3$ to the tangent bundle of $S$. As in the Riemannian setting, the shape operator is defined as $B = \I^{-1}\II$, or, more explicitly, it is defined by the relation $\II(v,w)=\I(Bv, w)$ for any $v,w\in T_x S$. The \textit{third fundamental form} $\III$ is then defined as $\III(v,w)=\I(Bv, Bw)$. The eigenvalues of $B$ are called principal curvatures.
The Gauss equation in the anti de Sitter setting is $K = -1 -\det B$, where $K$ denotes the intrinsic curvature of $\I$. Notice that in this case the locally strictly convex surfaces
(i.e. those for which $\II$ is definite) are exactly the surfaces with intrinsic curvature less than $-1$.

\subsection{Polar duality for surfaces in $\AdS^3$} \label{ssc:polar-ads}

We will be using below a notion of polar duality for smooth, convex, spacelike surfaces in $\AdS^3$. This duality is well-known in the hyperbolic setting, where a convex surface in $\HH^3$ is dual to a convex, spacelike surface in the de Sitter space $\dS^3$, see Section \ref{sec:polar}, and conversely. However this phenomenon is also well-known in a much more general setting that includes $\AdS^3$, see e.g. \cite[Prop 3.3]{shu} or \cite[Section 5]{fillastre-seppi}. We recall its definition and key properties here, for completeness, in the specific case of $\AdS^3$.

A simple way to define this duality is to use the description of the double cover $\widetilde{\AdS^3}$ of $\AdS^3 = \PSL(2, \R)$ as $\widetilde{\AdS^3} = \SL(2, \R)$, see Section \ref{projec}.

Let $V\co\Sigma\to \widetilde{\AdS^3}$ be a smooth, spacelike oriented embedding of an oriented surface into the double cover of $\AdS^3$. At each point $x\in \Sigma$, we can consider the oriented unit normal $n_x$ to $\mathrm{Im}(dV_x)$. Since $V$ is spacelike, $n_x$ is timelike and $n_x$, considered as a unit timelike vector in $\mathrm{M}_2(\R) = \R^{2,2}$, is also contained in $\widetilde{\AdS^3}$. The map $V^*\co\Sigma \to\widetilde{\AdS^3}$ defined by $x\mapsto n_x$ is the (polar) dual map to $V$.

By construction, if $x\in \Sigma$ and $v\in T_x\Sigma$ then $dV^*_x(v)=\nabla_vn=-Bv$, where $\nabla$ is the  Levi-Civita connection of $\tilde{\AdS^3}$ and $B$ is the shape operator of $\Sigma$. As a consequence, if $B$ is non-degenerate at every point of $\Sigma$, then $V^*$ is a smooth embedding. Moreover, the metric $I^*$ induced by $V^*$ is by construction equal to the third fundamental form of $V$:
$$ I^*(v,w)=\langle \nabla_vn,\nabla_wn\rangle = \langle Bv,Bw\rangle = \III(v,w)~. $$
It follows that, as soon as $V$ has non-degenerate shape operator, $V^*$ is spacelike, because its induced metric is positive definite.

For all $x\in \Sigma$, the image $\mathrm{Im}(dV_x)$ can be identified, as a vector subspace of $\mathrm{M}_2(\R) = \R^{2,2}$, with $\mathrm{Im}(dV^*_x)$. By definition, the vector subspace normal to $\mathrm{Im}(dV_x)$ is spanned by $V(x)$ and $V^*(x)$, which are also orthogonal and of unit norm. It follows that $V(x)$ is also a unit normal to $\mathrm{Im}(dV^*_x)$ in $T_{V^*(x)}\widetilde{\AdS^3}$. As a consequence, with the correct choice of orientation, $V$ is the embedding dual to $V^*$ by the same polar duality (thus the term ``duality''). It also follows from this argument that $B^*$, the pull-back by $V^*$ of the shape operator of $V^*(\Sigma)$, is equal to $B^{-1}$.

Note, finally, that exactly as for the polar duality in hyperbolic space, and for the same reasons, if the curvature of the induced metric of $V$ is equal to $K$ (and $B$ is non-degenerate), then the curvature of the metric induced by $V^*$ is equal to $\frac{K}{\det(B)}$. Since $K=-1-\det(B)$ by the Gauss formula in $\AdS^3$, it follows that the curvature of the metric induced by $V^*$ is equall to $K^*=-\frac{K}{K+1}$.

\subsection{Earthquakes and measured geodesic laminations}

We conclude this (second) preliminaries section with some facts about earthquake maps of the hyperbolic plane which are relevant for the AdS geometry constructions in Section~\ref{sec:leftright} and also for the proof, to be given finally in Section~\ref{approx}, of the surjectivity criterion Proposition~\ref{pr-surj} used in the proofs of Theorems~\ref{tm:induced-hyp}, \ref{tm:induced-hyp-K} and~\ref{tm:III-hyp-K} and its AdS analogue Proposition~\ref{pr-surj-ads}, which will be needed for Theorems~\ref{tm:induced-ads} and~\ref{tm:induced-ads-K}.

A {\em (geodesic) lamination} $\mathcal L$ on $\HH^2$ is a closed subset of  $\HH^2$ which is the union of disjoint geodesics, called the leaves of $\mathcal L$. A {\em measured (geodesic) lamination} is a pair $(\mathcal L, \lambda)$ consisting of a lamination $\mathcal L$ on $\HH^2$ and a transverse measure $\lambda$, namely a measure defined on each arc $c$ with support in $c \cap \mathcal L$ which is invariant under a homotopy of $c$ that respects the leaves.  
We denote $\cML(\HH^2)$ the space of all measured (geodesic) laminations on $\HH^2$. The Thurston norm of $\lambda$, denoted by $||\lambda||_{\mathrm{Th}}$, is defined as the supremum of $\iota(\lambda, c)$ over all geodesic segments $c$ of length $1$ in $\mathbb H^2$, where $\iota(\lambda, c)$ is the total measure of $\mathcal L \cap c$ with respect to $\lambda$. Clearly the Thurston norm can be $+\infty$. We say that a measured geodesic lamination is \textit{bounded} if its Thurston norm is finite. 
 We may alternatively think of a measured geodesic lamination as a locally finite measure $\lambda$ on the space $\mathcal G$ of geodesics in $\HH^2$ whose support is simple, meaning no two geodesics in the support cross each other.
 The weak-* topology on $\cML(\HH^2)$ refers to the topology induced by embedding $\cML(\HH^2)$ into the dual of the compactly supported continuous functions on $\mathcal G$ via integration:
 a sequence of measured geodesic laminations $\lambda_n \in \ML(\HH^2)$ weak-* converges to $\lambda$ if for any compactly supported continuous function $f$ we have
\[
    \int_{\mathcal G} f(l)d\lambda_n(l)\rightarrow\int_{\mathcal G} f(l)d\lambda(l)~.
\]

A \textit{left (resp. right) earthquake along a geodesic lamination} $\mathcal L$ (called the fault locus) is a possibly discontinuous bijective map $E\co\mathbb H^2\to\mathbb H^2$ such that 
\begin{itemize}
\item the restriction of $E$ to any stratum $F$ of $\mathcal L$ (that is, a geodesic of $\mathcal L$ or a connected component of $\HH^2 \setminus \mathcal L$) extends to a global isometry $A(F)$ of $\mathbb H^2$,
\item for any  pair of strata $F,F'$ the comparison map $A(F)^{-1}A(F')$ is a hyperbolic transformation whose axis weakly separates $F$ from $F'$, and which moves $F'$ on the left (resp. right) as seen from $F$.
\end{itemize}

By Thurston's work (Proposition III.1.6.1 of \cite{th-earth}) it is possible to associate to every earthquake $E$ a measured geodesic lamination $\lambda$ whose support is the fault locus of $E$, and which encodes the amount of shearing of $E$. We will refer to $\lambda$ as the shearing measure associated to $E$. The measured  lamination  $\lambda$  determines $E$, up to post-composition by elements in $\PSL(2,\mathbb R)$ (Proposition III.1.6.1 of \cite{th-earth}).

Thurston proved that, though $E$ itself may not be continuous, $E$ extends uniquely to an orientation-preserving homeomorphism of the ideal boundary at infinity, which we denote by $E|_{\RP^1}$.  A key result of Thurston theory is that any orientation-preserving homeomorphism of the circle is the extension of an earthquake which is essentially unique (see Theorem III.1.3.1 of \cite{th-earth}). We say that $E$ is normalised if $E|_{\RP^1}$ is a normalised homeomorphism.

The following technical lemma will play an important role in our proof.

\begin{lemma}[Lemma II.3.11.5 of \cite{can_not}]\label{lm:continearth}
If the shearing measures  $\lambda_n$ of a sequence of normalised earthquakes $E_n$ weakly-* converges to
the shearing locus $\lambda_\infty$ of an earthquake $E_\infty$, then $E_n|_{\RP^1}$ uniformly converges to $E_\infty|_{\RP^1}$, while $E_n$ pointwise converges to $E_\infty$ off from weighted leaves of $\lambda_\infty$.
\end{lemma}

Gardiner, Hu and Lakic \cite{Ga-Hu} and Saric \cite[Theorem 1]{sa} proved the following characterization of quasisymmetric homeomorphisms in terms of earthquakes.

\begin{prop}\cite{Ga-Hu, sa}\label{pr:bounlam}
 An orientation-preserving homeomorphism of the circle $f$ 
 is quasisymmetric if and only if the shearing measure  $\lambda$ of the earthquake extension $E$ of  $f$ is bounded. More precisely for any $M$ there is $C=C(M)$ such that if $f$ is $M$-quasisymmetric then $||\lambda||_{\mathrm{Th}}\leq C$. 

Conversely any bounded lamination is the shearing measure of an earthquake $E$ such that $E|_{\RP^1}$ is quasisymmetric. In fact  for any $C$ there is $M=M(C)$ such that  if $||\lambda||_{\mathrm{Th}}\leq C$ then $E:\HH^2\to\HH^2$ is an $M$--quasi-isometry.
\end{prop}

\section{Gluing maps in AdS geometry}

Here we carefully define the gluing maps $\Psi_{\cdot}, \Psi_{\cdot, K}, \Psi_{\cdot, K}^*: \qcm(\Ein^{1,1}) \to \cT$ from the introduction, filling in the technical results needed for the definitions. We will also give a critical estimate needed for the proofs of Theorems~\ref{tm:induced-ads}, \ref{tm:induced-ads-K}, and~\ref{tm:III-ads-K}.
We will follow the same outline as in the setting of hyperbolic geometry with some notable differences along the way.

Let $C \subset \Ein^{1,1}$ be an acausal meridian. Then, recall from Lemma~\ref{lem:graph} that $C = \Gamma(f)$ is the graph of an orientation-preserving homeomorphism $f: \RP^1 \to \RP^1$, where here we identify $\Ein^{1,1}$ with the product $\RP^1 \times \RP^1$ as in Section~\ref{prel ads}.
Recall from Proposition~\ref{prop:AdS-CH} that there are convex subsets $\CH(C) \subset E(C) \subset \AdS^3$, where the convex hull $\CH(C)$ is the unique smallest 
 closed convex subset accumulating on $C$ at infinity, and the invisible domain $E(C)$ is the unique maximal open convex subset containing $\CH(C)$. 
The boundary $\partial \CH(C)$ of $\CH(C)$ consists of two spacelike convex properly embedded disks that span $C$. We call the component of $\partial \CH(C)$ for which the outward normal is future oriented the \emph{future boundary component} and denote it $\partial^+ \CH(C)$. Similarly, the other boundary component, whose outward pointing normal is past oriented, is called the \emph{past boundary component} and denoted $\partial^- \CH(C)$. Note that the surfaces $\partial^\pm \CH(C)$ are not smooth, but rather are each bent along a geodesic lamination.

Now assume that $C = \Gamma(f)$ is a quasifuchsian quasicircle, meaning $f: \RP^1 \to \RP^1$ is a quasifuchsian quasisymmetric homeomorphism, conjugating one surface group $\pi_1 \Sigma \xrightarrow[\cong]{\rho_R} G_1 < \PSL(2,\RR)$ to another $\pi_1 \Sigma \xrightarrow[\cong]{\rho_L} G_2 < \PSL(2,\RR)$, then the invisible domain $E(C)$ is a maximal open convex domain of discontinuity for the $(\rho_L, \rho_R)$ action of $\pi_1 \Sigma$  on $\AdS^3$ and $\pi_1 \Sigma \backslash \CH(C)$ is called the convex core of the maximal globally hyberbolic spatially compact spacetime $\pi_1 \Sigma \backslash E(C)$, see \cite{mess}.
In this case, Barbot-B\'eguin-Zeghib~\cite{BBZ} proved that the complement of $\pi_1 \Sigma \backslash \CH(C)$ in $\pi_1 \Sigma \backslash E(C)$ admits a foliation by $K$--surfaces, i.e. surfaces whose Guass curvature is constant equal to~$K$.
The following result of Bonsante--Seppi generalizes this result to the context of interest here and also generalizes Rosenberg--Spruck's Theorem~\ref{tm:K-surfaces-hyp} from the hyperbolic setting.

\begin{theorem}[Barbot-B\'eguin-Zeghib~\cite{BBZ}, Bonsante--Seppi~\cite{bon_are}] \label{tm:K-surfaces-ads}
Let $C\subset \Ein^{1,1}$ be an acausal meridian, and let
$K \in (-\infty,-1)$. Then:
\begin{enumerate}
\item There are exactly two properly embedded $K$--surfaces in $\AdS^3$ spanning $C$. These are each
homeomorphic to disks, are disjoint, and bound a closed properly convex region $\mathscr C_K(C)$ in $\AdS^3$ which contains a neighborhood of the convex hull $\CH(C)$.
\item Further the $K$--surfaces spanning $C$, for $K\in (-\infty,-1)$, form a foliation of $E(C) \setminus \CH(C)$. 
\item Moreover, $C$ is a quasicircle if and only if 
any spanning $K$-surface has bounded principal curvatures. 
\end{enumerate}
\end{theorem}

The last point of Theorem \ref{tm:K-surfaces-ads} shows that in the AdS setting the bound of the curvatures of a $K$-surface reflects some regularity property of the asymptotic boundary.
By contrast Lemma \ref{principal_curv} shows there is a uniform bound on the principal curvatures of  a  properly embedded $K$-surface in $\Hyp^3$ spanning any Jordan curve.  

As we did for $\partial \CH(C)$, we label the two $K$-surfaces bounding $\mathscr C_K(C)$ according to the time orientation of the outward pointing normal vectors, with $\Sa_K^+(C)$ denoting the $K$-surface having future pointing outward normals and $\Sa_K^-(C)$ the $K$-surface with past pointing outward normal vectors. Note that as $K \to -1^-$, $\Sa_K^\pm(C)$ converges to the top/bottom boundaries $\partial^\pm \CH(C)$ of the convex hull $\CH(C)$. Hence, we will sometimes use the convention $\Sa_{-1}^\pm = \partial^\pm \CH(C)$, even though these surfaces are not technically considered $K$-surfaces since they are not smooth. Further we orient the surfaces $\Sa_K^{\pm}$ using the outward pointing normal and the orientation on $\AdS^3$.

For $K \in (-\infty,-1]$, let $\HH^{2\pm}_K$ be a copy of $\HH^{2\pm}$ equipped with the conformal metric that has constant curvature equal to $K$. The induced metric on the $K$-surface $\Sa^\pm_K(C)$ is locally isometric to $\HH^{2\pm}_K$. We note however that, in contrast to the setting of hyperbolic geometry, the metric on $\Sa^\pm_K(C)$ need not be complete without further assumption on the acausal meridian $C$.

\begin{prop}
\label{prop:extend-ads}
Let $C = \Gamma(f) \subset \Ein^{1,1}$ be a quasicircle (i.e. $f$ is quasisymmetric) and let  $K \in (-\infty,-1]$. Then the induced metric on $\Sa^\pm_K(C)$ is complete, hence globally isometric to $\HH^{2\pm}_K$. 
Further any isometry $V:\HH^{2\pm}_K\to S^\pm_{K}(C)$ extends to a homeomorphism of $\overline{\HH^{2\pm}_K} = \HH^{2\pm}_K \cup \RP^1$ onto $S^\pm_K(C) \cup C \subset \overline{\AdS^3}$.
\end{prop}

The proof will require several tools developed later in this section.

Now, let us assume the meridian $C$ is normalized, meaning it is oriented and passes through $(0,0),(1,1),(\infty,\infty)$ 
in positive order, and fix $K \in (-\infty,-1]$. Then, assuming Proposition~\ref{prop:extend-ads}, there are unique isometries $\VCK{C,K}^\pm: \HH^{2 \pm}_K \to S^\pm_{K}(C)$ whose extension to the boundary satisfies that $\partial \VCK{C,K}^\pm(i) = (i,i)$ for $i = 0,1,\infty$. The gluing map associated to $C$ and $K$ is simply the comparison map (see Section~\ref{sec:comparison}) between the two maps $\VCK{C,K}^+$ and $\VCK{C,K}^-$: 
\begin{equation}\label{eqn:PsiCK}
\Psi_{C, K} = \cmp(\VCK{C,K}^-, \VCK{C,K}^+) := (\partial \VCK{C,K}^-)^{-1}  \circ \partial \VCK{C,K}^+.
\end{equation}

The main goal of this section is to prove that $\Psi_{C,K}$ is a quasisymmetric homeomorphism so that $\Psi_{\cdot, K}: \qcm(\Ein^{1,1}) \to \cT$ is well-defined. In fact, we prove the stronger statement:

\begin{prop}\label{prop:well-defined-ads}
Let $K \in (-\infty, -1]$. Then for each $k > 1$, there exists a constant $k' > 1$ depending only on $k$ and $K$, so that for any (normalized) $k$-quasicricle $C = \Gamma(f)$, the map $\Psi_{C, K}$ is $k'$--quasisymmetric. In particular $\Psi_{C, K} \in \mathcal T$.
\end{prop}

\subsection{The left and right projections}\label{sec:leftright}

Here we introduce natural projection maps associated to a convex spacelike surface is $\AdS^3$ and prove Proposition~\ref{prop:extend-ads}.

Recall from Section~\ref{prel ads} that any timelike geodesic in $\AdS^3 = \PSL(2,\R)$ is of the form $L_{x,x'} = \{A \in \PSL(2,\R): Ax' = x\}$. Hence a smoothly embedded spacelike surface $S$ determines a map $\Pi: S \to \HH^2 \times \HH^2$ defined by $\Pi(s) = (\Pi_l(s),\Pi_r(s))$ where $L_{\Pi_l(s),\Pi_r(s)}$ is the unique timelike geodesic orthogonal to $S$ at $s$. 
The maps $\Pi_l$ and $\Pi_r$ are local diffeomorphisms when the surface $S$ is locally convex (and in fact more generally when the intrinsic curvature does not vanish). Further, Krasnov--Schlenker~\cite{minsurf}  computed the pull-back of the hyperbolic metric through $\Pi_l$ and $\Pi_r$ in terms of the embedding data of the immersion (see also Lemma 3.16 in \cite{maximal}):
\begin{equation}\label{eq:leftright}
\begin{array}{l}
\Pi_l^*(g_{\mathrm{hyp}})(v, w)=\I((E+J_\I B)v, (E+J_\I B)w)\, \\
\Pi_r^*(g_{\mathrm{hyp}})(v, w)=\I((E-J_\I B)v, (E-J_\I B)w),
\end{array}
\end{equation}
where $E$ denotes the identity operator, and $J_\I$ is the complex structure over $TS$ induced by $\I$. (Note that while both $J_\I$ and $B$ depend on the choice
of an orientation on $S$, the product is independent of the orientation).
Here is a useful lemma about these projections.
\begin{lemma}[Lemma 3.18 and Remark 3.19 of \cite{maximal}]\label{lm:propconv}
Assume that $S\subset\AdS^3$ is a properly embedded spacelike convex surface whose accumulation set in $\partial\AdS^3=\Ein^{1,1} = \RP^1 \times \RP^1$ is an acausal meridian $C = \Gamma(f)$. Then 
the maps $\Pi_l$ and $\Pi_r$ are diffeomorphisms onto $\HH^2$ that continuously extend the canonical projections $\pi_l, \pi_r:C\to\mathbb \RP^1$ defined by $\pi_l(x,f(x)) = x$ and $\pi_r(x,f(x)) = f(x)$.
\end{lemma}

In fact, it is further true that:
\begin{prop}\label{prop:bilip}
Let $S\subset\AdS^3$ be a smooth properly embedded spacelike convex surface whose accumulation set $C = \Gamma(f)$ in $\partial\AdS^3=\Ein^{1,1} = \RP^1 \times \RP^1$ is a quasicircle and whose principle curvatures are bounded in absolute value between $D$ and $1/D$. Then the projection maps $\Pi_l$ and $\Pi_r$ are bilipschitz diffeomorphisms with bilipschitz constant depending only on $D$. 
\end{prop}

\begin{proof}
Let us prove the statement for $\Pi_l$ as the statement for $\Pi_r$ follows similarly. 
By  Formula~\eqref{eq:leftright} it is sufficient to prove that the eigenvalues of $(E+J_IB)^*(E+J_IB)$ are bounded from below and from above
by positive constants (here $A^*$ denotes the $\I$-adjoint of $A$).

A direct computation shows that $\det(E+JB)= 1+\det(B) = - K$, where the second equality is the Gauss formula in $\AdS^3$. On the one hand $K \leq -1$ since $S$ is convex, and on the other hand, as $-K$ is the product of the (absolute values of the) principle curvatures, we have that $-K$ is bounded above. Hence $\det(E+JB)$, and  hence $\det(E+JB)^*$ and the product $\det(E+J_IB)^*(E+J_IB)$ are each bounded from above and below.
Moreover, 
$$ \tr((E+JB)^*(E+JB))=2+\tr(B^2)~. $$
Again, since the eigenvalues of $B$ (the principle curvatures) are uniformly bounded, we conclude that the eigenvalues of $(E+J B)^*(E+J B)$
are bounded as desired.
\end{proof}

We now prove the $K < -1$ case of Proposition~\ref{prop:extend-ads}.
\begin{proof}[Proof of Proposition~\ref{prop:extend-ads} for $K < -1$.]
Suppose $C = \Gamma(f)$ is a quasicircle and let $K \in (-\infty, -1)$. Let us work with $\Sa^+_K(C)$, the arguments for $\Sa^-_K(C)$ are the same. By Theorem~\ref{tm:K-surfaces-ads}, the principle curvatures of the surfaces $\Sa^+_K(C)$ are uniformly bounded.  Hence by Proposition~\ref{prop:bilip}, the left projection $\Pi_l$ associated to $\Sa^+_K(C)$ is a bilipschitz diffeomorphism taking the induced metric on $\Sa^+_K(C)$ to the hyperbolic metric on $\HH^2$. Hence the induced metric is complete, hence globally isometric to $\HH^{2+}_K$. 

Next, let $V\co \HH^{2+}_K \to \Sa^+_K(C)$ be any orientation-preserving isometry. Then $\Pi_l \circ V$ is a bilipschitz diffeomorphism hence it extends uniquely to a homeomorphism $\partial (\Pi_l \circ V)\co \partial \HH^{2 +}_K \to \RP^1$ of the ideal boundary. We then observe that $V$ extends to $\partial V\co \partial \HH^{2+}_K \to C$ at the boundary by the formula 
$$\partial V = \pi_l^{-1}\circ \partial (\Pi_l \circ V).$$
It follows  from Lemma~\ref{lm:propconv} that $V \cup \partial V$ is a homeomorphism from $\HH^{2+}_K \cup \partial \HH^{2+}_K$ to $S^+_K(C) \cup C$.
\end{proof}

Let us now turn the $K = -1$ case of Proposition~\ref{prop:extend-ads}.
The surfaces $\Sa^\pm_{-1} = \partial^\pm \CH(C)$ are not smooth. Rather they are totally geodesic surfaces bent along geodesic laminations. Nonetheless, Mess~\cite{mess} in the equivariant case, and Benedetti--Bonsante~\cite{BeBo} in general case, described how to interpret the left and right projections above as earthquake maps. In general, if $C$ is any acausal meridian in $\Ein^{1,1}$, the induced metric on the future (resp. past) boundary $\partial^+ \CH(C)$ (resp. $\partial^- \CH(C)$) of the convex hull $\CH(C)$ is isometric to an open subset $\mathcal U^+$ (resp. $\mathcal U^-$) of the hyperbolic plane which is bounded by a (possibly empty) set of disjoint geodesics. An isometry $\VCK{C,-1}^+: \mathcal U^+ \to \partial^+\CH(C)$ (resp. $\VCK{C,-1}^-: \mathcal U^- \to \partial^-\CH(C)$) is given by a bending map, with bending determined by a measured geodesic lamination $\lambda^+$ of $\mathcal U^+$ (resp. $\lambda^-$ of $\mathcal U^-$). The projection maps $\Pi_l^+$ and $\Pi_r^+$ (resp. $\Pi_l^-$ and $\Pi_r^-$) for $\partial^+ \CH(C)$ (resp. for $\partial^- \CH(C)$) are well-defined only at the dense set of points along which $\partial^+ \CH(C)$ (resp. $\partial^- \CH(C)$) is $C^1$; this is the complement of the leaves of $\lambda^+$ (resp. of $\lambda^-$) which have positive measure. Mess observed the following relationship between the bending measure and earthquakes.

\begin{prop}[Mess~\cite{mess}, Benedetti-Bonsante~\cite{BeBo}]\label{prop:mess}
The compositions $$\Pi_l^+ \circ \VCK{C,-1}^+: \mathcal U^+ \to \HH^2,\;\; 
\Pi_r^+ \circ \VCK{C,-1}^+: \mathcal U^+ \to \HH^2$$
$$\Pi_l^- \circ \VCK{C,-1}^-: \mathcal U^+ \to \HH^2,\;  \text{ and } \;
\Pi_r^- \circ \VCK{C,-1}^-: \mathcal U^+ \to \HH^2$$ 
are surjective earthquake maps, with $\Pi_l^+ \circ \VCK{C,-1}^+: \mathcal U^+ \to \HH^2$ (resp $\Pi_r^- \circ \VCK{C, -1}^-: \mathcal U^- \to \HH^2$) shearing to the left along $\lambda^+$ (resp. along $\lambda^-$) and 
$\Pi_l^- \circ\VCK{C,-1}^-: \mathcal U^- \to \HH^2$ (resp. $\Pi_r^+ \circ \VCK{C,-1}^+: \mathcal U^+ \to \HH^2$) shearing to the right along $\lambda^-$ (resp. along $\lambda^+$). The diagram below records this information:
$$\xymatrix{ & & \mathcal U^+ \ar[d]^{\VCK{C,-1}^+} \ar@/^/[ddrr]^{E^l_{\lambda^+}} \ar@/_/[ddll]_{E^r_{\lambda^+}}
              \\
  & & \partial^+ \CH(C) \ar[dll]^{\Pi_l^+} \ar[drr]_{\Pi_r^+} 
                 &      \\
  \HH^2 & & & & \HH^2  \\
 & & \partial^- \CH(C) \ar[ull]_{\Pi_l^-}  \ar[urr]^{\Pi_r^-} & \\
&  & \mathcal U^- \ar[u]_{\VCK{C,-1}^-} \ar@/_/[uurr]_{E^r_{\lambda^-}} \ar@/^/[uull]^{E^l_{\lambda^-}} &            }
$$
Further, letting  $\lambda^+_l = \Pi_l^+ \circ \VCK{C,-1}^+(\lambda^+)$ (resp.  $\lambda^-_l = \Pi_l^- \circ \VCK{C,-1}^-(\lambda^-)$), the earthquake map $E^l_{2\lambda_l^+}: \HH^2 \to \HH^2$ (resp. the right earthquake map $E^r_{2\lambda_l^-}: \HH^2 \to \HH^2$) is the unique left (resp. right) earthquake map extending the homeomorphism $f: \RP^1 \to \RP^1$ to the hyperbolic plane.
\end{prop}

We will also need an analogue of Lemma~\ref{lm:propconv}.
Although the statement of Lemma~\ref{lm:propconv} from~\cite{maximal} is in the smooth category, the proof does not use smoothness and can be easily adapted to show the following.
\begin{lemma}\label{lem:boundary-project}
The (earthquake) projection maps $\Pi_l^+$ and $\Pi_r^+$ associated to the surface $\partial^+ \CH(C)$ extend respectively the canonical projections $\pi_l, \pi_r:C\to\mathbb \RP^1$ defined by $\pi_l(x,f(x)) = x$ and $\pi_r(x,f(x)) = f(x)$.
The projection maps $\Pi_l^-$ and $\Pi_r^-$ associated to the surface $\partial^- \CH(C)$ extend $\pi_l$ and $\pi_r$ respectively.
\end{lemma}

Let us now prove the $K =-1$ case of Proposition~\ref{prop:extend-ads}.
\begin{proof}[Proof of Proposition~\ref{prop:extend-ads} for $K=-1$]
Let us focus on the future surface $\partial^+ \CH(C)$. The argument for the past surface $\partial^- \CH(C)$ is the same.
We denote by $\lambda_l^+$ the image of the bending lamination $\lambda^+$ through the projection $\Pi^+_l$. Similarly we define $\lambda^+_r$, $\lambda^-_l$, and $\lambda^-_r$.
In the context of Proposition~\ref{prop:mess}, Proposition~\ref{pr:bounlam} implies that $f : \RP^1 \to \RP^1$ is quasisymmetric if and only if $2 \lambda_l^+$ (or $2 \lambda_l^-$) is a bounded lamination. In particular, if $f$ is quasisymmetric then $\lambda_l^+$ is a bounded measured lamination and hence $\mathcal U^+ = E^r_{\lambda_l^+}(\HH^2)$ is all of $\HH^2$. Hence $\partial^+ \CH(C)$ is complete. Let $\VCK{C,-1}^+: \HH^2 \to \partial^+ \CH(C)$ be an isometry. Then the composition $\Pi_l^+ \circ \VCK{C,-1}^+$ is an earthquake map with bounded shearing measure $\lambda^+$. Hence $\Pi_l^+ \circ \VCK{C,-1}^+$ extends uniquely to a quasisymmetric homeomorphism $\partial (\Pi_l^+ \circ \VCK{C,-1}^+): \RP^1 \to \RP^1$. By Lemma~\ref{lem:boundary-project}, this map factors as $\partial (\Pi_l^+ \circ \VCK{C,-1}^+) = \pi_l \circ \partial \VCK{C,-1}^+$, where $\pi_l: C \to \RP^1$ is the canonical left projection taking $(x,f(x)) \in C$ to $x \in \RP^1$ and $\partial \VCK{C,-1}^+: \RP^1 \to C$ is a homeomorphism extending~$\VCK{C,-1}^+$.
\end{proof}

Proposition~\ref{prop:extend-ads} is now proved. We next turn our focus to Proposition~\ref{prop:well-defined-ads}, 
For $K=-1$ the  proof is a based on the following stronger result.

\begin{lemma}\label{lem:unif-qi-ads}
For each $k > 1$, there exists $A > 1$ so that the earthquake maps $\Pi_l^\pm \circ \VCK{C,-1}^\pm: \HH^2 \to \HH^2$ associated to the boundary components $\partial^\pm \CH(C)$ of the convex hull $\CH(C)$ of any $k$-quasicircle $C$ are $A$--quasi-isometries.
\end{lemma}

\begin{proof}
Let $C = \Gamma(f)$ be a normalized $k$-quasicircle. The earthquake map $\Pi^+_l \circ \VCK{C,-1}^+$ has shearing measured lamination denoted $\lambda^+$ and the earthquake map $\Pi^-_l \circ \VCK{C,-1}^-$ has shearing lamination denoted $\lambda^-$. By Proposition~\ref{prop:mess}, the left earthquake map extending the $k$-quasisymmetric homeomorphism $f$ has shearing measure $2\lambda^+$. So by Proposition~\ref{pr:bounlam}, $\|2\lambda^+\|_{Th}$, and hence also $\|\lambda^+\|_{Th}$, is bounded in terms of~$k$. Similarly, $\|\lambda^-\|_{Th}$ is bounded in terms of~$k$. It then follows, again by Proposition~\ref{pr:bounlam}, that each of $(\Pi_l^- \circ \VCK{C,-1}^-)$, and  $(\Pi_l^+ \circ \VCK{C,-1}^+)$ is an $A$--quasi-isometry for some constant $A > 1$ depending only on $k$. \end{proof}

\begin{proof}[Proof of Proposition~\ref{prop:well-defined-ads} ($K = -1$):]
Let $\VCK{C, K}^\pm:\Hyp^{2\pm}\to\partial_{\pm}\CH(C)$ be orientation-preserving isometries which extend to homeomorphisms $\partial\VCK{C, -1}^{\pm}:\RP^1\to C$, guaranteed to exist by Proposition \ref{prop:extend-ads}.
Further, by pre-composing with isometries, we may adjust so that each $\VCK{C, -1}^{\pm}$ is normalized, in the sense that $\partial\VCK{C, -1}^{\pm}(i)=(i,i)$ for $i=0,1,\infty$.
Then $\Psi_{C, -1}=(\partial\VCK{C,-1}^-)^{-1}\circ\partial\VCK{C,-1}^+$ extends the composition $(\Pi_l^-\circ\VCK{C,-1}^-)^{-1}\circ(\Pi_l^+\circ\VCK{C,-1}^+)$ of right earthquake maps. 
Hence by Lemma \ref{lem:unif-qi-ads} $\Psi_{C,-1} = (\partial \VCK{C,-1}^-)^{-1} \circ \partial \VCK{C,-1}^+$ is $k'$-quasisymmetric for some $k'$ depending only on~$k$.
\end{proof}

 The $K < -1$ case will require several compactness arguments.

\subsection{Compactness statements}
As in the setting of hyperbolic geometry, we will use compactness arguments in $\AdS^3$ frequently.

\begin{lemma}\label{lm:continuityboundary}
Let $S_n$ be a sequence of properly embedded convex spacelike disks spanning a sequence of $k$-quasicircles $C_n$.
Then if $C_n$ converges to the union $\{p\} \times \RP^1 \cup \RP^1 \times \{q\}$ of a line of the left ruling and a line of the right ruling, then $S_n$ converge to the lightlike plane with boundary at infinity $\{p\} \times \RP^1 \cup \RP^1 \times \{q\}$.
If $C_n$ converges to a $k$-quasicircle $C$, then, up to a subsequence, $S_n$ converges to a locally convex properly embedded surface spanning the curve $C$.
\end{lemma}
\begin{proof}
Using coordinates \eqref{coordinates sl2r} on $\AdS^3$, notice that we can identify the Klein model of hyperbolic plane $\Hyp^2$  with the intersection of $\AdS^3$ with the projective plane given by the equation $x_1=0$. In these coordinates, $\Hyp^2$ is entirely contained in the affine chart $x_2\neq 0$.

We consider the universal covering $\pi:\overline{\HH^2}\times\RR\to\overline{\AdS^3}$ given by $\pi([0:1:x_3:x_4], t)=[\sin t: \cos t: x_3: x_4]$. The lift of the surface $S_n$  is properly embedded and spans the lift of the curve $C_n$. (Note that $C_n$ lifts because it bounds a disk in $\AdS^3$.) By Proposition 3.2 of \cite{maximal}, this lift is a global graph of a function
$u_n:\overline{\HH^2}\to \mathbb R$.  We can choose the lift so that $u_n[0:1:0:0]\in[0,\pi)$. The function $u_n$ is in general Lipschitz continuous and satisfies a gradient estimates $||du_n||^2_{\mathrm{Hyp}}<\phi^2$, where 
$\phi([0:1:x_3:x_4])=(1+x_3^2+x_4^2)$ almost everywhere. In particular considering the Euclidean metric on $\HH^2$ coming from the Poincar\'e model, the functions $u_n$ are $2$-Lipschitz
(see the discussion at page 295 of \cite{maximal}).
It follows that, up to a subsequence, $u_n$ converges on $\overline{\HH^2}$ to a Lipschitz function $u_\infty$, which satisfies the inequality $||du_\infty||^2\leq\phi$ almost everywhere. This shows that the accumulation points of the limit of $S_n$ is the limit $C$ of $C_n$.

It remains to show that if $C$ is the boundary of a lightlike plane $P$, then necessarily $S_n$ converges to $P$. The lift of $P$ on the universal covering is the graph of a function $u_P$ that satisfies $||du_P||^2_{\mathrm{hyp}}=\phi^2$. Let $u_\infty$ be the limit of $u_n$. Notice that $u_\infty= u_P$ on $\partial\HH^2$. Notice that the gradient of $u$ in the Euclidean metric of the Poincar\'e model is proportional to the gradient  of $u$ in the hyperbolic metric. So we conclude that $||du_\infty||_{\mathrm{Euc}}<||du_P||_{\mathrm{Euc}}$.

Let $c$ be the projection on $\HH^2$ of a lightlike ray of $P$. The curve $c$
connects two points $x_-$ and $x_+$ of $\partial \HH^2$, moreover it is parallel to the gradient of $u_P$.
Let $c(t)$ with $t\in[0,1]$ be a parameterization of such a line. Notice that $du_P(\dot c)=||du_P||_{\mathrm{Euc}}||\dot c||\geq du_\infty(\dot c)$ with the equality that
holds iff $du_P=du_\infty$ at $c(t)$.
On the other hand as $u_\infty$ is equal to $u_P$ at $x_-$ and $x_+$ we deduce that $\int_0^1 (du_P(\dot c(t))-du_\infty(\dot c(t))) dt=0$, thus
$du_P=du_\infty$. This implies that the graph of $u_\infty$ contain the line $c$

Since $P$ is foliated by lightlike rays, we conclude that $u_P=u_\infty$ everywhere.
\end{proof}

\begin{prop}\label{pr:cmpads}
Fix $k>1$ and $K<-1$.
Let $C_n = \Gamma(f_n)$ be a sequence of $k$-quasicircles defined by $k$-quasisymmetric homeomorphims $f_n:\RP^1\to\RP^1$. Then, up to taking a subsequence, either $\Sa^{\pm}_{K}(C_n)$ converges to a lightlike plane, or $C_n$ converges to a $k$-quasicircle $C$ and $\Sa^\pm_{K}(C_n)$ converges to $\Sa^\pm_{K}(C)$.

Moreover in the latter case, if $V_n:\HH^2_K\to\Sa^\pm_{K}(C_n)$ is a sequence of isometries, then it has a subsequence which smoothly converges uniformly on compact subsets to an isometry $V_\infty\co\HH^2_K\to\Sa_{K}(C)$ provided that there exists a bounded sequence $x_n\in\HH^2$ such that $V_n(x_n)$ is bounded in $\AdS^3$.
\end{prop}

\begin{proof}
We prove the statement for $\Sa^+_{K}(C_n)$. The other case is obtained by reversing time orientation.
 By Lemma \ref{lm:continuityboundary} we need to prove that if $C_n$ converges to a quasicircle $C$, 
then the limit $\Sa_\infty$ of $\Sa^+_{K}(C_n)$ has constant curvature $K$. 

By Lemma~\ref{lm:continuityboundary}, $\Sa_\infty$ is a convex surface different from a lightlike plane.
It turns out that there exists a sequence of points $p_n\in\Sa^+_{K}(C_n)$ converging to $p_\infty\in\Sa_\infty$ such that the tangent plane of $\Sa^+_{K}(C_n)$ at $p_n$  converges to a spacelike plane. We fix a sequence of isometries $\sigma'_n:\HH^2_K\to\Sa^+_{K}(C_n)$ sending a fixed point $x_0$ to $p_n$. We can then apply Th\'eor\`eme 5.6 of \cite{these} which shows that either $\sigma'_n$ converges smoothly to a isometric embedding $\sigma'_\infty$, or there exists a 
complete geodesic $\gamma\subset \HH^2_K$ which is sent by $\sigma$ to a spacelike geodesic $\Gamma\subset \AdS^3$. Moreover, the integral of the mean curvature of 
$\sigma_n$ tends to $\infty$ in the neighborhood of any point of $\gamma$. This implies that the length for the third fundamental form of $\sigma_n$ of any geodesic segment 
transverse to $\gamma$ goes to $\infty$ as $n\to \infty$. (This is, in fact, a direct consequence of Lemme 5.4 of \cite{these}.) As a consequence, the limit surface $\Sa_\infty$ 
must be contained in the union of the past-directed lightlike half-planes bounded by $\Gamma$. On the other hand, this contarsicts the fact that the boundary points of $\Sa_\infty$ form
a quasicircle. By Lemma \ref{lm:continuityboundary} the asymptotic boundary of $\Sa_\infty$ is $C$, so that $\Sa_\infty=\Sa^+_{K, C}$. 

Thus $\sigma'_n$ converges to an isometric immersion $\sigma'_\infty:\HH^2_K\to\AdS^3$ whose image is contained in $\Sa^+_{K, C}$. Since $\HH^2$ is complete, $\sigma'_\infty$ is a proper embedding, so $\Sa_\infty=\sigma'_\infty(\HH^2)$ and the conclusion follows.

The last part of the proof is proved in a similar way.
\end{proof}

We have seen in Proposition \ref{principal_curv} that in the setting of hyperbolic geometry, the principal curvatures of a properly embedded $K$-surface spanning any Jordan curve are uniformly bounded by a constant independent of the Jordan curve (but depending on $K$).
By contrast in the AdS setting Bonsante and Seppi proved that the principal curvature of a properly embedded $K$-surface are bounded if and only if the curve at infinity is a quasicircle \cite{bon_are}.
The following lemma refines this statement to give a bound on the principal curvatures of a $K$-surface in $\AdS^3$ spanning a $k$-quasicircle which only depends on $K$ and $k$.
 
\begin{lemma} \label{lm:unif-bound-II-ads}
Let $K\in (-\infty,-1)$, and let $k>1$. There exists a constant $D(K,k)>0$ such that any properly embedded $K$--surface in $\AdS^3$ with boundary at infinity a $k$--quasicircle has principal curvatures at most $D(K,k)$. 
\end{lemma}

\begin{proof} 
Suppose by contradiction that there exists a sequence of $k$-quasicircles $C_n$, and a sequence of points $p_n\in\Sa^{\pm}_{K}(C_n)$ 
such that the largest principal curvature at $p_n$  diverges. Replacing if necessary $C_n$ by its image through an isometry of $\AdS^3$, we may assume that $p_n$ is a fixed point in $\AdS^3$, and
$T_{p_n}\Sa^{\pm}_{K}(C_n)$ is a fixed spacelike plane. Hence no subsequence of $\Sa^{\pm}_{K}(C_n)$ converges to a lightlike plane. By applying Proposition \ref{pr:cmpads} we conclude that $\Sa^{\pm}_{K}(C_n)$ smoothly converges to a $K$-surface. On the other hand, this contradicts the assumption
that the principal curvatures are not bounded at $p_n$.
\end{proof}

We now prove Proposition~\ref{prop:well-defined-ads} for $K < -1$. As for $K=-1$ we will first prove a stronger statement.

\begin{lemma}\label{lem:unif-bilip-ads}
For each $K < -1$ and $k > 1$, there exists $M > 1$ so that the 
diffeomorphisms $\Pi_l^\pm \circ \VCK{C,K}^\pm: \HH^2_K \to \HH^2$ associated to the $K$-surfaces $\Sa^\pm_K(C)$ spanning any $k$-quasicircle $C$ are $M$--bilipschitz.
\end{lemma}
\begin{proof}
By Lemma~\ref{lm:unif-bound-II-ads}, the principle curvatures of the surfaces $\Sa^+_K(C)$ and $\Sa^-_K(C)$ are uniformly bounded by a constant depending only on $K$ and~$k$.  Hence by Proposition~\ref{prop:bilip}, the left projections $\Pi_l^+$  and $\Pi_l^-$ associated to $\Sa^+_K(C)$ and $\Sa^-_K(C)$ respectively are $M$--bilipschitz diffeomorphisms taking the induced metrics on $\Sa^+_K(C)$ and $\Sa^-_K(C)$ to the hyperbolic metric on $\HH^2$, where $M$ depends on $K$ and~$k$.  
Hence $\Pi_l^+ \circ \VCK{C, K}^+$ and $\Pi_l^- \circ \VCK{C, K}^-$ are each normalized $M$--bilipschitz diffeomorphisms of the hyperbolic plane. \end{proof}
\begin{proof}[Proof of Proposition~\ref{prop:well-defined-ads}: $K < -1$]
Let $K < -1$.
Suppose first that $C = \Gamma(f)$ is a normalized $k$-quasicircle. 
Let $\VCK{C, K}^\pm: \HH^{2\pm}_K \to \Sa^\pm_K(C)$ be orientation-preserving isometries which extend to homeomorphisms $\partial \VCK{C, K}^\pm : \RP^1 \to C$, guaranteed to exist by Proposition~\ref{prop:extend-ads}. Further by pre-composing with isometries, we may adjust so that each of $\VCK{C, K}^+$ and $\VCK{C, K}^-$ is normalized, in the sense that $\partial \VCK{C,K}^\pm(i) = (i,i)$ for $i = 0,1, \infty$.
By Lemma \ref{lem:unif-bilip-ads} the composition $(\Pi_l^- \circ \VCK{C,K}^-)^{-1} \circ (\Pi_l^+ \circ \VCK{C, K}^+)$ is then a $M^2$--bilipschitz diffeomorphism of the hyperbolic plane
which extends the map $\Psi_{C, K} = (\partial \VCK{C, K}^-)^{-1} \circ \partial \VCK{C, K}^+$. It follows that $\Psi_{C, K}$ is a $k'$-quasisymmetric homeomorphism where the constant $k'$ depends only on $M^2$ which depends only on $K$ and~$k$.
\end{proof}

\section{Proofs of Theorems~\ref{tm:induced-ads}, \ref{tm:induced-ads-K} and \ref{tm:III-ads-K}}

Similar to the proofs of Theorems~\ref{tm:induced-hyp} and~\ref{tm:induced-hyp-K} in the setting of hyperbolic geometry, we will use the following surjectivity criterion, analogous to Proposition~\ref{pr-surj}, to prove Theorems~\ref{tm:induced-ads} and \ref{tm:induced-ads-K}. Like Proposition~\ref{pr-surj}, the following proposition is a direct corollary of Proposition~\ref{pr:uniflim}.

\begin{prop} \label{pr-surj-ads}
Let $F\co\qcm({\Ein^{1,1}}) \to \cT$ be a map satisfying the following conditions:
\begin{enumerate}[(i)]
\item\label{item:contads} If $(C_n)_{n\in \N}$ is a sequence of normalised $k$--quasicircles converging to a normalised $k$--quasicircle $C$, then $(F(C_n))_{n\in \N}$ 
converges uniformly to $F(C)$.
\item\label{item:propads} For any $k$, there exists $k'$ such that if $F(C)$ is a normalised $k$--quasisymmetric homeomorphism, then $C$ is a $k'$--quasicircle.
\item\label{item:qfads} The image of $F$ contains all the quasifuchsian elements of $\cT$.
\end{enumerate}
Then $F$ is surjective.
\end{prop}

We begin by noting that for each $K \in (-\infty,-1]$, the map $\Psi_{\cdot, K}$ satisfies Condition~\eqref{item:qfads}. For $K= -1$, this is due to Diallo \cite{diallo2013}, while for $K \in (-\infty, -1)$, this is due to Tamburelli~\cite{tamburelli2016}. To prove Theorems~\ref{tm:induced-hyp} and~\ref{tm:induced-hyp-K}, we must now prove that for any $K \leq -1$, the map $\Psi_{\cdot, K}$ satisfies Conditions~\eqref{item:contads} and~\eqref{item:propads}. We treat each condition in its own subsection.

\subsection{Continuity of $\Psi_{\cdot, K}$}

In this section we prove that the map $\Psi_{\cdot, K}$ satisfies Condition~\eqref{item:contads} in Proposition \ref{pr-surj-ads} which we state as a stand-alone proposition for convenience:

\begin{prop}\label{pr:ads-continuity}
 Let $K \in (-\infty, -1]$ and let $C_n$ be a sequence of normalized $k$--quasicircles converging in the Hausdorff sense to a normalized $k$-quasicircle $C$. Then $\Psi_{C_n, K}$ converges uniformly to $\Psi_{C, K}$.
\end{prop}

\begin{proof}
Let $\VCK{C, K}^\pm\co \HH^2_K \to \Sa^\pm_K(C)$ be the unique isometry which is normalized, in the sense that its extension $\partial \VCK{C, K}^\pm: \RP^1 \to C$ satisfies that $\partial \VCK{C, K}^\pm(i) = (i,i)$, and let $\Pi_l^\pm: \Sa^\pm_K(C) \to \HH^2$ denote the left projection from $\Sa^\pm_K(C)$ to $\HH^2$, which is either an area preserving diffeomorphism if $K < -1$ or an earthquake map if $K = -1$.
Similarly, for each $n$, let $\VCK{C_n, K}^\pm: \HH^{2\pm}_K \to \Sa^\pm_K(C_n)$ denote the unique normalized isometry, let $(\Pi^\pm_l)_n: \Sa^\pm_K(C_n) \to \HH^2$ denote the left projection associated to $\Sa^\pm_K(C_n)$.
 Recall that $\Psi_{C_n, K} = (\partial \VCK{C_n, K}^{-1})^{-1} \circ \partial \VCK{C_n, K}^+$ extends the composition $((\Pi^-_l)_n \circ \VCK{C_n, K}^-)^{-1} \circ ((\Pi^+_l)_n \circ \VCK{C_n, K}^+)$.
By Lemma~\ref{lem:unif-qi-ads} in the case $K = -1$ or Lemma~\ref{lem:unif-bilip-ads} in the case $K < -1$, both $((\Pi^-_l)_n \circ \VCK{C_n, K}^{-1})$ and $((\Pi^+_l)_n \circ \VCK{C_n, K}^{+})$ are $A$--quasi-isometries for some $A > 1$ depending only on $k$.
By Lemma \ref{lm:int-vs-bound}, it is sufficient to show that the maps $((\Pi^\pm_l)_n \circ \VCK{C_n, K}^\pm)$ converge pointwise almost everywhere to the maps $(\Pi_l^\pm) \circ \VCK{C,K}^{\pm}$. By Lemma~\ref{lm:norm} and Proposition~\ref{pr:cmpads}, we only need to show that for some point $x \in \HH^2$, the image $\VCK{C_n, K}^{\pm}(x)$ remains in a compact set in $\AdS^3$. 
By Lemma~\ref{lm:norm}, the normalized uniform quasi-isometries $((\Pi^\pm_l)_n \circ \VCK{C_n, K}^{\pm})$ take any point $x$ into some compact subset $K_l$ of $\HH^2$. Similarly, the normalized uniform quasi-isometries 
$((\Pi^\pm_r)_n \circ \VCK{C_n, K}^{\pm})$ take any point $x$ into some compact subset $K_r$ of $\HH^2$, where here $(\Pi^\pm_r)_n: \Sa^\pm_K(C_n) \to \HH^2$ denotes the right projection associated to $\Sa^\pm_K(C_n)$. Hence $\VCK{C_n, K}^{\pm}(x)$ lies in the union of all timelike geodesics of the form $L_{x,y}$ with $x \in K_l$ and $y \in K_r$.  This is a compact family of compact subsets, whose union is therefore compact. This completes the proof.
  \end{proof}

\subsection{Properness of $\Psi_{\cdot, K}$}

In this section, we prove that the map $\Psi_{\cdot, K}\co \qcm(\Ein^{1,1}) \to \cT$ satisfies condition $(ii)$ of Proposition \ref{pr-surj-ads}, which we state here as a stand-alone proposition for convenience:

\begin{prop}\label{pr:ads_properness}
For each $K \in (-\infty, -1]$ and $k>1$, there exists $k'>1$ such that if $\Psi_{C, K}$ is $k$--quasisymmetric, then $C$ is a $k'$--quasicircle.
\end{prop}

The proof will use a simple geometric statement. Given a spacelike totally geodesic plane $P\subset \AdS^3$, we denote by $P^+$ the union of all future-oriented timelike geodesic segments of length $\pi/2$ starting orthogonally from $P$. Similarly, we denote by $P^-$ the union of all past oriented timelike geodesic segments of length $\pi/2$ starting orthogonally from $P$.

\begin{lemma} \label{st:convex}
\begin{enumerate}
\item
  Let $S\subset \AdS^3$ be a spacelike past convex surface, and let $P$ be a spacelike totally geodesic plane. In the neighborhood of all points $x\in S\cap P$, the intersection $S\cap P^+$ is locally convex in the induced metric on $S$. 
  \item   Let $S\subset \AdS^3$ be a spacelike future convex surface, and let $P$ be a spacelike totally geodesic plane. In the neighborhood of all points $x\in S\cap P$, the intersection $S\cap P^-$ is locally convex in the induced metric on $S$. 
  \end{enumerate}
\end{lemma}

\begin{proof}
 We prove the first statement, as the second is similar. We assume that $S$ is smooth, observing that the general case then follows by approximation. We also suppose that $P$ is not tangent to $S$ at $x$, since otherwise the result is clear.

  Since $P$ is not tangent to $S$, the intersection $P\cap S$ is a smooth curve $\gamma$ in the neighborhood of $x$. We denote by $\kappa$ its curvature vector. Since $\gamma\subset P$, the vector $\kappa$ is parallel to $P$. Denote by $\kappa_S$ the curvature vector of $\gamma$ as a curve in $S$, so that $\kappa_S$ is tangent to $S$ at $x$. 
  Then
  \begin{align*}
  \kappa_S &= \kappa + \langle \kappa, N_S \rangle N_S,
  \end{align*}
  where $N_S$ denotes the future-oriented unit normal vector to $S$.
  Then, if $N$ denotes the future oriented unit normal vector to $P$, we have that
  \begin{align*}
  \langle \kappa_S, N \rangle &=\langle \kappa, N_S \rangle \langle N_S, N \rangle,
  \end{align*}
  where we have used that $\langle \kappa, N \rangle = 0$ since $\gamma$ is contained in $P$. Finally observing that $\langle N_S, N \rangle < 0$ since both $N$ and $N_S$ are future oriented, and that $\langle \kappa, N_S \rangle > 0$ since $S$ is past convex, we conclude that $\langle \kappa_S, N \rangle < 0$, and hence $\kappa_S$ points toward the future side of $P$, i.e. into $P^+ \cap S$.
\end{proof}

\begin{proof}[Proof of Proposition \ref{pr:ads_properness}]
  We argue by contradiction. Suppose that $(C_n)_{n\in \N}$ is a sequence of normalized quasicircles, whose optimal quasisymmetric constant $k_n'$ tends to infinity. We will prove that the quasisymmetric constant of the gluing map $\Psi_{C_n, K}$ 
  between the future and past $K$-surfaces, denoted $\Sa^+_n = \Sa^+_K(C_n)$ and $\Sa^-_n = \Sa^-_K(C_n)$ respectively, must also go to infinity. Let $V_{C_n, K}^\pm\co \HH^{2\pm}_K \to \Sa^\pm_K(C_n)$ be the unique isometry whose extension 
  $\partial \VCK{C_n, K}^{\pm}\co \RP^1 \to C_n$ to the boundary satisfies $\partial \VCK{C_n, K}^+(i) = (i,i)$ for $i=0,1,\infty$. Then, recall that $\Psi_{C_n, K} = (\partial \VCK{C_n, K}^{-})^{-1} \circ \partial \VCK{C_n, K}^+$.

  Since $k_n'\to \infty$, Proposition~\ref{pr:w_ads} implies that the width $w(C_n)$ converges to $\pi/2$. Further, after adjusting by an appropriate sequence of isometries, $C_n$ converges in the Hausdorff topology to the rhombus example $\rhombus$ of Example~\ref{ex:rhombus}. We note that, even after adjusting by isometries we may assume the quasicircles $C_n$ remain normalized and that the limit $\rhombus$ is also normalized, although this is not important for the arguments to come.
  
 Let us work in the projective model for $\AdS^3$ in the coordinates~\eqref{coordinates sl2r} so that in the affine chart $x_4 = 1$, $\Ein^{1,1}$ is the hyperboloid $x_1^2 + x_2^2 = x_3^2 + 1$, with $\AdS^3$ seen as the region  $x_1^2 + x_2^2 < x_3^2 + 1$. 
 We may then arrange that the vertices of $\rhombus$ are the points $(\pm\sqrt{2},0,-1)$ and $(0,\pm\sqrt 2,1)$.  
 Since $\Sa^+_n$ is in the future of the $\CH(C_n)$ but contained in the invisible domain $E(C_n)$, and since $\overline{E(C_n)}$ and $\CH(C_n)$ both  converge to $\CH(\rhombus)$,  we have that
the limit as $n \to \infty$ of $\Sa^+_n$ is the union $\Sa^+_\infty$ of the two future faces of $\CH(\rhombus)$. Similarly, the limit as $n \to \infty$ of $\Sa^-_n$ is the union $\Sa^-_\infty$ of the two past faces of $\CH(\rhombus)$.
 
 \begin{figure}[htb]
 
 \def\svgwidth{4.6cm}
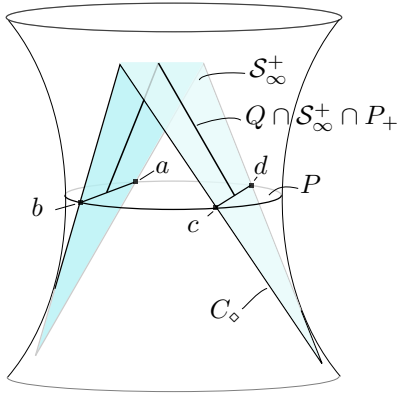
 
 \caption{ The path $Q \cap \Sa^+_\infty \cap P^+$ on the surface $\Sa^+_\infty$ is piecewise lightlike and connects a point of the $\AdS$ geodesic connecting $a$ to $b$ to a point on the $\AdS$ geodesic connecting $c$ to $d$. } \label{fig:proof-rhombus}
 \end{figure}

  Let $P$ be the spacelike totally geodesic plane equal to the intersection with $\AdS^3$ of the plane $z=0$ in $\RR^3$. Let $a,b,c,d$ be the intersection of $\partial P$ with $\rhombus$, occurring in this cyclic order, so that $a=(-\sqrt 2/2,\sqrt 2/2,0)$, $b = (-\sqrt{2}/2,-\sqrt{2}/2, 0)$, $c = (\sqrt{2}/2,-\sqrt{2}/2,0)$, and $d = (\sqrt{2}/2, \sqrt{2}/2, 0)$. For all $n\in \N$, we denote by $a_n,b_n, c_n,d_n$ the intersection points of $C_n$ with $P$ in the respective neighborhoods of $a,b,c,d$, so that $a_n \to a, b_n \to b, c_n \to c, d_n \to d$. Let $a_n^\pm = (\partial \VCK{C_n, K}^\pm)^{-1}(a_n)$, $b_n^\pm = (\partial \VCK{C_n, K}^\pm)^{-1}(b_n)$, $c_n^\pm = (\partial \VCK{C_n, K}^\pm)^{-1}(c_n)$, and $d_n^\pm = (\partial \VCK{C_n, K}^\pm)^{-1}(d_n)$.   We will show that the cross-ratio of the four points $(a_n^+, b_n^+, c_n^+, d_n^+)$ goes to zero, but that the cross-ratio of their images $(a_n^-, b_n^-, c_n^-, d_n^-)$ under $\Psi_{C_n, K}$ goes to infinity. 
 
 Consider the timelike plane $Q$ defined by $x_2 = 0$. Then the path $Q \cap \Sa^+_\infty \cap P_+$ from $Q \cap \overline{ab}$ to $Q \cap \overline{cd}$ along the piecewise lightlike geodesic $Q \cap \Sa^+_\infty$ has length zero in the $\AdS$ metric, where here $P_+$ denotes the future of $P$ as in Lemma~\ref{st:convex}. See Figure~\ref{fig:proof-rhombus}. It follows that the length of the path $Q \cap \Sa^+_n \cap P_+$ has distance converging to zero. By Lemma~\ref{st:convex}.(1), $\Sa^+_n \cap P_+$ has locally convex boundary with respect to the induced metric. In other words $U^+_n := (\VCK{C_n, K}^+)^{-1}(\Sa^+_n \cap P_+)$ corresponds to a region of $\HH^{2+}_K$ which has locally convex boundary and hence is globally convex. Since $\Sa^+_n \cap P_+$ has the points $a_n, b_n, c_n, d_n$ in its closure, Proposition~\ref{prop:extend-ads} implies that $U^+_n$ has the points $a_n^+, b_n^+, c_n^+,$ and $d_n^+$ in its ideal boundary. Since $U^+_n$ is convex, it contains the geodesics $\overline{a_n^+ b_n^+}$ and $\overline{c_n^+ d_n^+}$ and hence $\Sa^+_n \cap P_+$ contains $\gamma^{+}(a_n,b_n) := \VCK{C_n, K}^+(\overline{a_n^+ b_n^+})$ and $\gamma^+(c_n,d_n) := \VCK{C_n, K}^+(\overline{c_n^+ d_n^+})$. The path $Q \cap \Sa^+_n \cap P_+$ crosses both geodesics. Thus the distance between $\gamma^{+}(a_n,b_n)$ and $\gamma^+(c_n,d_n)$ in $\Sa^+_n$ is converging to zero, hence $\overline{a_n^+ b_n^+}$ and $\overline{c_n^+ d_n^+}$ are becoming arbitrarily close, and therefore the cross-ratio of $(a_n^+, b_n^+, c_n^+, d_n^+)$ converges to zero as $n \to \infty$. See Figure~\ref{fig:proof-rhombus-2}.
  
   \begin{figure}[htb]
 
 \def\svgwidth{9.5cm}
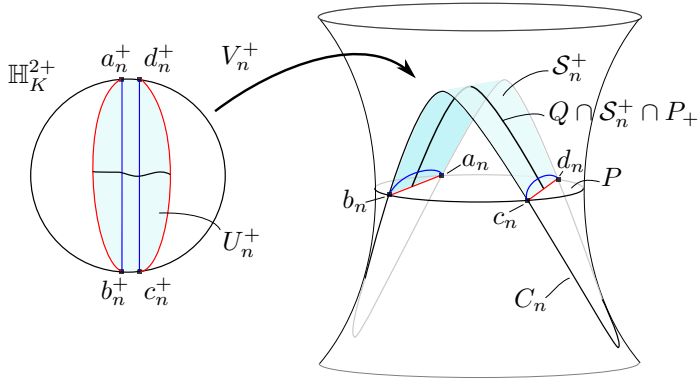

 \caption{As $n \to \infty$, the length of the path $Q \cap \Sa^+_n \cap P_+$ shrinks to zero. By Lemma~\ref{st:convex}, the region $U_n^+ \subset \HH^{2+}_K$ is convex, hence the geodesics $\overline{a_n^+ b_n^+}$ and $\overline{c_n^+ d_n^+}$ of $\HH^{2+}_K$ are becoming arbitrarily close.} \label{fig:proof-rhombus-2}
 \end{figure}

We apply a similar argument to the surfaces $\Sa^-_n$. Let $R$ denote the timelike plane defined by $x_1 = 0$. The path $R \cap \Sa^-_\infty \cap P_-$ joins $R \cap \overline{da}$ to $R \cap \overline{bc}$ and has length zero. So we similarly argue, using Lemma~\ref{st:convex}.(2), that the points $(d_n^-, a_n^-, b_n^-, c_n^-)$ at the ideal boundary of the convex set $U^-_n = (\partial \VCK{C_n, K}^-)^{-1}(\Sa^-_n \cap P_-) \subset \HH^{2-}_K$ have cross-ratio converging to zero. Hence, after rearranging, the cross-ratio of the points $(a_n^-, b_n^-, c_n^-, d_n^-)$ tends to infinity.  
Therefore, using Proposition \ref{qs_cross}, the quasisymmetric constants of $\Psi_{C_n, K}$ go to infinity as claimed. 
 \end{proof}

Propositions~\ref{pr:ads-continuity} and~\ref{pr:ads_properness}, plus the results of Diallo and Tamburelli imply Theorems \ref{tm:induced-ads} and Theorem \ref{tm:induced-ads-K}, via Proposition~\ref{pr-surj-ads}. 

\subsection{Proof of Theorem \ref{tm:III-ads-K}}

We now use the polar duality in $\AdS^3$, described in Section \ref{ssc:polar-ads}, to prove Theorem \ref{tm:III-ads-K} from Theorem \ref{tm:induced-ads-K}. We consider a fixed $K\in (-\infty, -1)$, and set $K^*=-\frac{K}{K+1}\in (-\infty, -1)$.

Let $v\in \cT$ be a normalized quasisymmetric homeomorphism $v\co\RP^1\to \RP^1$. By Theorem \ref{tm:induced-ads-K}, there is a normalized quasicircle $C \subset \partial \AdS^3$ such that the gluing map between the future-convex $(K^*)$--surface $S^-_{K^*}(C)$ and the past-convex $(K^*)$--surface $S^+_{K^*}(C)$ spanning $C$ is equal to $v^{-1}\in \cT$. 

Let $S^{*-}$ and $S^{*+}$ be the polar duals of $S^+_{K^*}(C)$ and $S^-_{K^*}(C)$, respectively. Then both surfaces have constant curvature $K$. Moreover, their induced metric is equal (under the identification by the duality map of $S^+_{K^*}(C)$ with $S^{*-}$, and of $S^-_{K^*}(C)$ with $S^{*+}$) to $\III$, the third fundamental form of $S^+_{K^*}(C)$ and $S^-_{K^*}(C)$. It follows that those two surfaces $S^{*-}$ and $S^{*+}$ have curvature $K$, and they are complete since $S^+_{K^*}(C)$ and $S^-_{K^*}(C)$ have bounded principal curvatures by Lemma \ref{lm:unif-bound-II-ads}.

Note also that $S^{*-}$ and $S^{*+}$ are contained in the invisible domain of $C$, and therefore have asymptotic boundary equal to $C$. Indeed, if $x^*$ is a point of $S^{*-}$, then it is dual to a support plane $P$ of $S^+_{K^*}(C)$. This plane $P$ is in the past of $S^+_{K^*}(C)$, and therefore also in the past of $C$, so that the past cone of $x^*$ is disjoint from $C$ and therefore $x^*$ is contained in the invisible domain of $C$. 

We also notice that the duality map between $S^+_{K^*}(C)$ and $S^{*-}$, and between $S^-_{K^*}(C)$ and $S^{*+}$ extends as the identity on their common boundary $C$. This follows from the fact that both are space-like, complete surfaces with the same boundary, and each point $x$ of $S^-_{K^*}(C)$ corresponds to a point $x^*$ of $S^{*+}$ which is in its future. As $x$ converges to a limit point $\xi\in C$, the intersection of the future cone of $x$ with $S^{*+}$ also converges to $\xi$, so that the point dual to $x$ on $S^{*+}$ must also converge to $\xi$.

Since the third fundamental forms of $S^{*\pm}$ are identified through the duality map with the induced metrics on $S^\mp_{K^*}(C)$, and since the gluing map at infinity between the induced metrics on $S^-_{K^*}(C)$ and $S^+_{K^*}(C)$ is $v^{-1}$, the gluing map at infinity between the third fundamental forms of $S^{*-}$ and $S^{*+}$ is equal to $v$. This concludes the proof of Theorem \ref{tm:III-ads-K}.

\section{Approximation of quasisymmetric maps}\label{approx}

In this section we show how to approximate quasisymmetric homeomorphisms by quasifuchsian ones. The main goal is to prove Proposition~\ref{pr-surj} which we used in the proofs of Theorems~\ref{tm:induced-hyp}, \ref{tm:induced-hyp-K}, and~\ref{tm:III-hyp-K} above, and its analogue, Proposition~\ref{pr-surj-ads}, used in the proofs of Theorem~\ref{tm:induced-ads} and~\ref{tm:induced-ads-K}.  Both propositions are direct corollaries of the following refinement of~\cite[Lemma 3.4]{bon-sep}.

\begin{prop}\label{pr:uniflim}
Any element $v\in\cT$ is the $C^0$--limit of a sequence of normalised uniformly quasisymmetric quasifuchsian homeomorphisms $v_n$.
\end{prop}

Letus first prove the surjectivity criteria Propositions~\ref{pr-surj} and~\ref{pr-surj-ads} assuming Proposition~\ref{pr:uniflim}. 

\begin{proof}[Proof of Propositions~\ref{pr-surj} and~\ref{pr-surj-ads}]
Let $\qcm(\mathbb X)$ denote the space of normalized quasicircles in $\mathbb X$, where $\mathbb X$ is either $\CP^1$ or $\Ein^{1,1}$. Let $F: \qcm(\mathbb X) \to \cT$ be a map that satisfies the three criteria:
\begin{enumerate}[(i)]
\item\label{item:cont_proof} If $(C_n)_{n\in \N}$ is a sequence of normalised $k$--quasicircles converging in the Hausdorff topology to a normalised $k$--quasicircle $C$, then $(F(C_n))_{n\in \N}$ converges uniformly to $F(C)$.
\item\label{item:prop_proof} For any $k$, there exists $k'$ such that if $F(C)$ is a normalised $k$--quasisymmetric homeomorphism, then $C$ is a $k'$--quasicircle.
\item\label{item:qf_proof} The image of $F$ contains all the quasifuchsian elements of $\cT$.
\end{enumerate}
Let $v \in \cT$. Then by Proposition~\ref{pr:uniflim}, $v$ is the $C^0$-limit of a sequence $(v_n)$ of normalized quasifuchsian homeomorphisms which are $k$-quasisymmetric for a fixed $k$ independent of $n$. By~\eqref{item:qf_proof}, let $C_n$ be normalized quasicircles so that $F(C_n) = v_n$. By~\eqref{item:prop_proof}, each $C_n$ is a $k'$-quasicircle, for some $k'$ independent of $n$. 
We pass to a subsequence so that $C_n$ converges in the Hausdorff topology to some limit $C_\infty$. By Lemma~\ref{lm:int-conv} if $\mathbb X = \CP^1$ or Lemma~\ref{lm:cmpqs} if $\mathbb X = \Ein^{1,1}$, $C_\infty$ is a $k'$--quasicircle.
By~\eqref{item:cont_proof}, $F(C_\infty) = v$. Hence $F$ is surjective.
\end{proof}

The rest of this section is devoted to the proof of Proposition \ref{pr:uniflim}.
Let $v \in \cT$. Let $\lambda$ denote the shearing lamination associated to the left earthquake extending $v$. It is a bounded measured geodesic lamination by Proposition \ref{pr:bounlam}. 
We will show (Lemma \ref{lem:approxlam}) that $\lambda$ is approximated by a sequence of uniformly bounded measured laminations $(\lambda_n)_{n\in \N}$, so that each $\lambda_n$ is  invariant under the action of a uniform lattice  $\Gamma_n$. Each $\lambda_n$ is the shearing locus of an earthquake whose boundary extension, say $v_n$ is a quasifuchsian homeomorphism. The proof is then complete upon remarking that the sequence $v_n$ uniformly converges to $v$ (Lemma \ref{lm:continearth}).

The proof of the key lemma (Lemma \ref{lem:approxlam}) is based directly on the same strategy used in \cite[Lem. 3.4]{bon-sep}. 
The main difference is that here we require the Thurston norm of $\lambda_n$ to be uniformly bounded, so some additional care is needed in the construction. 

First we state two technical facts we will need in order to prove Lemma \ref{lem:approxlam}.

\begin{lemma}\label{lem:tecn0}
Let $l_1,\ldots, l_k$ be pairwise disjoint complete geodesics in $\mathbb H^2$.
Then there exist $\epsilon>0$ and  smooth families of geodesics $(l_1(t), \ldots, l_k(t))_{t\in [0,\epsilon]}$,
 such that
\begin{itemize}
\item $l_i(0)=l_i$ for all $i\in \{ 1,\cdots, k\}$,
\item $d(l_i(t), l_j(t))>d(l_i, l_j)$ for every $t\in (0,\epsilon]$ and $i\neq j$, where $d$ denotes the hyperbolic distance between pairs of disjoint geodesics. 
\end{itemize}
\end{lemma}

\begin{proof}
We argue by induction on $k$. The statement is certainly true for $k=1$.
Let us consider the inductive step $k\Rightarrow (k+1)$. Up to renaming the geodesics we can assume that all $l_i$ are contained in a closed half-plane $P_0$ bounded by  $l_{k+1}$.
By the inductive hypothesis there is $\epsilon>0$ and a smooth family of geodesics $(l_1(t),\ldots, l_k(t))_{t\in[0,\epsilon]}$ which satisfies the stated properties.
After applying a smooth family of isometries we may assume that $l_i(t)\subset P_0$ for every $t\leq\epsilon$.

By smoothness there is a constant  $C>0$ such that $$d(l_i(t), l_{k+1})>d(l_i, l_{k+1})-Ct~$$
holds for all $1 \leq i \leq k$ and $t \in [0,\epsilon]$.
Let $(g(t))_{t\in [0,\epsilon]}$ be a $1$-parameter subgroup of hyperbolic transformations with axis orthgonal to $l_{k+1}$, and repulsive fixed point in $P_0$, normalized so that the translation length of $g(t)$ is $Ct$.
Set $l_{k+1}(t)=g(t)(l_{k+1})$. In this way $d(l_{k+1}(t), l_{k+1})=Ct$. Since $l_{k+1}$ separates $l_{k+1}(t)$ from $l_i(t)$ for $i=1,\ldots, k$, we deduce that
\[
    d(l_{i}(t), l_{k+1}(t))\geq d(l_i(t), l_{k+1})+d(l_{k+1}, l_{k+1}(t))>d(l_i, l_{k+1})~.
\]
This shows that the family $(l_1(t),\ldots, l_{k+1}(t))_{t\in[0,\epsilon]}$ satisfies the stated properties and completes the induction.
\end{proof}

\begin{lemma}\label{lem:tec}
Let $U$ be a region of $\mathbb H^2$ bounded by two ultraparallel geodesics $e_1, e_2$ and a geodesic arc $c$ joining $e_1$ and $e_2$. Fix also a point $x_0\in\mathbb H^2$ and a number $R>0$. There exists a piecewise geodesic arc $\alpha$ in $U$ with endpoints on $e_1$ and $e_2$ such that
\begin{itemize}
\item $\alpha$ meets $e_1$ and $e_2$ orthogonally, and the angle at each verte  of $\alpha$ is $\pi/2$.
\item The region of $U$ bounded by $c$ and $\alpha$ is convex. 
\item The distance from $x_0$ to $\alpha$ is larger than $R$.
\end{itemize}
\end{lemma}

\begin{proof}
In the upper half-plane model, we may suppose that the endpoints of $e_1$ are $0,\infty$ and the endpoints of $e_2$ are $1, a$ for some $a>1$. In this way the segment $[0,1]$ is in the asymptotic boundary of $U$. For any $n$ we take the partition of $[0,1]$ with nodes at $t_i=i/(1+2n)$ for $i=0,\ldots 1+2n$, and consider the geodesics $g_1,\ldots g_n$ of $\mathbb H^2$  so that the endpoints of $g_i$ are $t_{2i-1}, t_{2i}$.

 \begin{figure}
 \centering
 \def\svwidth{\columnwidth}
\def\svgscale{.3}
 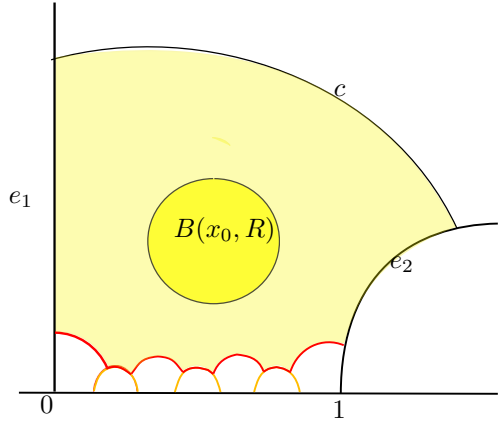
 \caption{The construction of the piecewise geodesic arc $\alpha$, in red in the picture}\label{lemme93:fig}
 \end{figure}

Notice that the endpoints of $g_i$ and $g_{i+1}$ are  $\PSL(2,\R)$-equivalent to the four points $-1$, $0$ $1$ and $2$, so the distance between $g_i$ and $g_{i+1}$ is some constant $d_0$ (independent of $n$). By comparing the distance between $g_1$ and $e_1$ with the distance between $g_1$ and the geodesic with endpoints $-1/(1+2n)$ and $0$ we see that the former distance is bounded by $d_0$ too. Analogously, the distance between $g_n$ and $e_2$ is bounded by $d_0$, 
if  $\frac{1}{2n+1}<a$.

On the other hand, the distance between $g_i$ and $c$  and between $g_i$ and $x_0$ diverges for $n\to+\infty$. Thus, we can choose $n$ big enough so that the distances between $g_i$ and $c$ are all bigger than $d_0$ 
and the distances between $g_i$ and $x_0$ are all bigger than $R+d_0$.

In this way the geodesics arc $a_i$ joining orthogonally $g_i$ to $g_{i+1}$  cannot meet $c$, and any point of $a_i$ is at distance at least $R$ from $x_0$.
The same facts  hold for the arcs $a_0$ joining orthogonally $e_1$ to $g_1$ and the arc $a_{n}$ joining orthogonally $g_n$ to $e_2$. 

Finally we can construct the arc $\alpha$ as union of the arcs $a_i$ and of geodesic segments of each of the $g_i$ (see Figure~\ref{lemme93:fig}).
\end{proof}

\begin{lemma}\label{lem:approxlam}
Let $\lambda$ be a bounded lamination on $\mathbb H^2$. There are a constant $C>0$, and 
 a sequence $(\lambda_n, \Gamma_n)$, where $\Gamma_n$ is a cocompact Fuchsian group and $\lambda_n$ is a $\Gamma_n$-invariant measured lamination, such that:
\begin{itemize}
\item $\lambda_n$ converges to $\lambda$ in the weak--$\ast$ topology.
\item The Thurston norm of $\lambda_n$ is bounded by $C$ for every $n$. 
\end{itemize}
\end{lemma}

\begin{proof}
In Theorem 5 of Miyachi and \v{S}ari\'c \cite{mi-sa} it is shown that $\lambda$ is the limit of discrete laminations $\mu_n$ for the uniform weak--$\ast$ convergence. Moreover in the proof of Theorem 5  (Claim 2) it is pointed out that $\mu_n$ have uniformly bounded Thurston norm. 

For a fixed $n$  we consider the leaves $l_1,\ldots, l_{s(n)}$ of $\mu_n$ which intersect a ball of radius $n$ centered at a fixed point $x_0\in \mathbb H^2$. Denote  
$a_i$  the  weight corresponding to $l_i$.
Fix a reference Riemannian distance $d_\mathcal G$ on the space of geodesics.
By Lemma~\ref{lem:tecn0} we may replace the geodesics $l_1,\ldots, l_{s(n)}$, some of which may be asymptotic, with  ultraparallel geodesics $l'_1,\ldots, l'_{s(n)}$ such that
\begin{itemize}
\item $d_{\mathcal G}(l_i, l'_i)<1/n$ for every $i=1,\ldots, s(n)$,
\item $d(l'_i, l'_j)>d(l_i, l_j)$ for every $i\neq j$.
\end{itemize}
Consider the finite lamination $\mu'_n=\sum_{i=1}^{s(n)}a_i\delta_{l'_i}$, where $\delta_l$ denotes the Delta measure centered at $l$.
As in the proof of Step 2 of Lemma 3.4 in \cite{bon-sep} we have that $\mu'_n$ is made by ultraparallel geodesics and $\mu'_n$ converges to $\lambda$
in the weak--$\ast$ topology.
Finally since $d(l'_i, l'_j)>d(l_i, l_j)$  the Thurston norm of $\mu'_n$ is smaller than the Thurston norm
of $\mu_n$. So there is a constant $C$ such that $||\mu'_n||_{Th}<C$ for any $n$.  

Now for each fixed $n$ we intend to construct a quasifuchsian approximation of $\mu'_n$, say $\lambda_n^{k}$, such that $||\lambda_n^{k}||_{Th}\leq 2||\mu'_n||_{Th}$. 
To this aim let us fix $n$. We claim that  for any $k$ there exists a polygon $P_k$ such that:
\begin{itemize}
\item $P_k$ contains the ball $B(x_0, k+n)$;
\item The interior angles of $P_k$ are all $\pi/2$;
\item The boundary $\partial P_k$ meets every leaf of $\mu'_n$ orthogonally in two points;
\item Any vertex of $P_k$ is at distance bigger than $1$ from the intersection points of any leaf of $\mu'_n$ with $\partial P_k$.
\end{itemize}

\begin{figure}
\centering
\def\svwidth{\columnwidth}
\def\svgscale{.3}
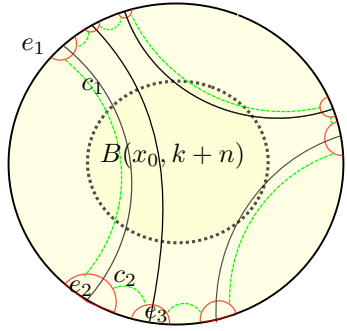
\caption{Geodesics $e_i$ and geodesic segments $c_i$ as in the proof of Lemma \ref{lem:approxlam}.}\label{polygon1:fig}
\end{figure}

 \begin{figure}
\centering
\def\svwidth{\columnwidth}
\def\svgscale{.3}
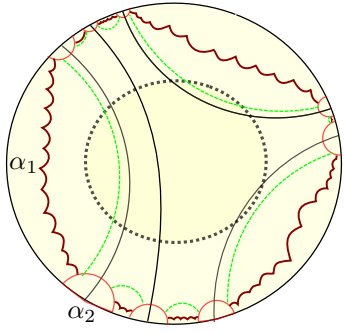
\caption{The construction of the piece-wise geodesics $\alpha_i$.}\label{polygon2:fig}
\end{figure}
  
  \begin{figure}[htb]
   \includegraphics[width = 2in]{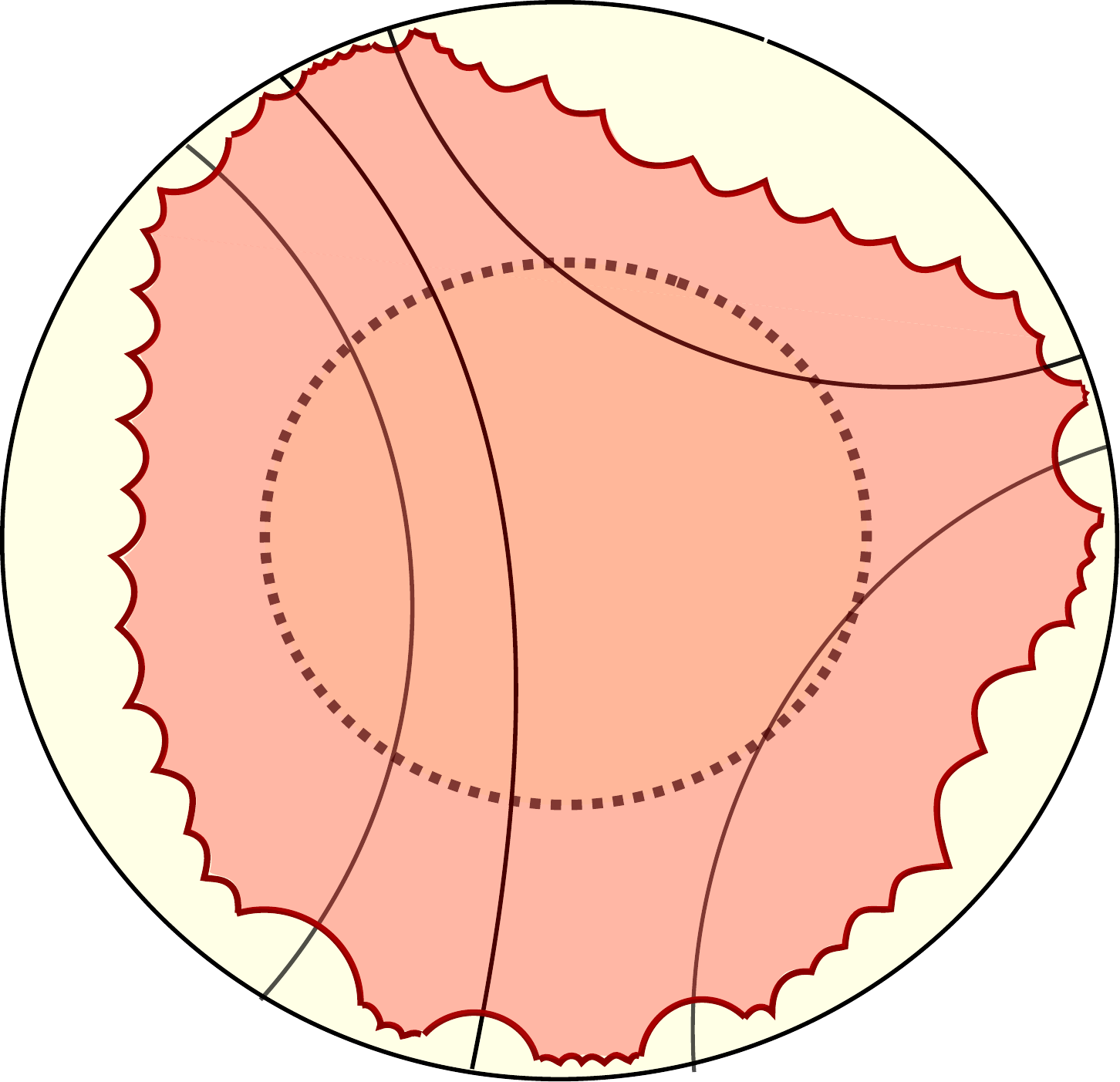}
   \caption{The polygon $P_k$.}\label{polygon3:fig}
   \end{figure}

To construct $P_k$ we first choose for each leaf $l$ of $\mu'_n$ two  geodesics  which orthogonally  meet $l$  at points in  different components of  $l\setminus B(x_0, k+n)$, in such a way
\begin{enumerate}
\item the distance of those geodesics to any other leaf of $\mu'_n$ is bigger than $1$;
\item all the geodesics constructed in this way are ultraparallel and bound a convex open subset of $\Hyp^2$ containing $B(x_0, k+n)$.
\end{enumerate}

Let us  enumerate those geodesics by $e_1, e_2,\ldots, e_N$ so that $e_i$ is adjacent to $e_{i-1}$ and $e_{i+1}$. 
We can choose a point $p_i$ on $e_i$ and $q_{i+1}$ on $e_{i+1}$ so that the distance between the segment $c_i=[p_i, q_{i+1}]$ and the support of $\mu'_n$ is bigger than $1$ (see Figure~\ref{polygon1:fig}).

Let $U_i$ be the region of $\mathbb H^2$ bounded by $e_i, e_{i+1}$ and $c_{i}$ which does not contain $e_j$ for $j\neq i, i+1.$ 
Applying Lemma \ref{lem:tec} to $e_i, e_{i+1}$ and $c_i$ we construct a piecewise geodesic path $\alpha_i$ contained in $U_i$ so that: 
\begin{itemize}
\item $\alpha_i$ meets orthogonally $e_i$ and $e_{i+1}$;
\item $\alpha_i$ stays outside $B(x_0, k+n)$ and at each vertex it forms a $\pi/2$-angle facing $x_0$;
\item The distance between $\alpha_i$ and the support of $\mu'_n$ is bigger than $1$. (see Figure~\ref{polygon2:fig}.)
\end{itemize} 
The union of $\alpha_i$ and the segments  along $e_i$ joining $\alpha_{i}$ to $\alpha_{i-1}$  bounds  a convex polygon $P_k$ (Figure~\ref{polygon3:fig}). It is immediate to check that $P_k$ satisfies the required properties.

Let $\Gamma_n^k$ be the reflection group generated by the reflections along the edges of $P_k$. For $k\geq n$, the lamination  $\mu'_n$  meets the boundary of $P_k$ orthogonally, so its $\Gamma_n^k$--orbit is a measured geodesic lamination $\lambda_n^k$ invariant under $\Gamma_n^k$, and $\lambda_n^k\cap B(x_0, k+n)=\mu'_n\cap B(x_0, k+n)$.

We want to prove that the Thurston norm of $\lambda_n^k$ is bounded by the Thurston norm of $\mu'_n$. If $c$ is a geodesic arc of length $1$, let $\hat c$ be the intersection of its $\Gamma_n^k$--orbit with $P_k$. Clearly $\iota(\lambda_n^k, c)=\iota(\lambda_n^k, \hat c)=\iota(\mu'_n, \hat c)$, where $\iota$ is the intersection form on measured laminations. If $\hat c$ is contained in the interior of $P_k$, then the statement is obvious. If $\hat c$ meets the boundary of $P_k$, by the last condition of $P_k$, it can meet at most one leaf of $\lambda_n^k$, so again its total intersection is less that the Thurston norm of $\mu'_n$.

Finally observe now that $\lambda_n^n$ weakly--$\ast$ converges to $\lambda$. Indeed if $f$ is a function with compact support on the space of geodesics of $\mathbb H^2$, then there exists $k$ such that the support of $f$ is contained in the set of geodesics that meet $B(x_0,k)$. For $n>k$ we have that
\[
  \int f(g)d\lambda_n^n(g)=\int f(g)d\mu'_n(g)
\] 
and the latter converges to $\int f(g)d\lambda(g)$ as $n\to+\infty$.
\end{proof}

\begin{proof}[Proof of Proposition \ref{pr:uniflim}]
We consider the shearing lamination $\lambda$ of the left earthquake extending $v$. This is a bounded measured  lamination by Proposition \ref{pr:bounlam}.
Applying Lemma \ref{lem:approxlam} we get a sequence of uniformly bounded  measured laminations $\lambda_n$ and  Fuchsian groups $\Gamma_n$,
 such that $\lambda_n$ is $\Gamma_n$ invariant and $(\lambda_n)_{n\in\mathbb N}$ weakly-* converges to $\lambda$.
 
 Let $v_n$ denote the restriction to $\RP^1$ of the normalised left erathquake along $\lambda_n$. 
 By Proposition \ref{pr:bounlam} the homeomorphisms $v_n$ are uniformly quasisymmetric.
 By Lemma \ref{lm:continearth} the sequence $v_n$ uniformly converges to $v$. 
\end{proof}

\section{Parametrized versions of the main theorems}\label{sec:param}

In the final section, we give versions of the main theorems for which the quasicircle $C$ is replaced by a parameterized quasicircle and the gluing map at infinity between the two $K$-surfaces spanning $C$ is replaced by a pair elements of the universal Teichm\"uller space. These ``parameterized" versions of the main theorems follow easily, and are appealing in that they more superficially resemble the analogous statements from the setting of quasifuchsian hyperbolic three-manifolds or globally hyperbolic maximal compact AdS spacetimes.

\subsection{Parameterized quasicircles in $\CP^1$}

\begin{defi}
An embedding $f\co \RP^1 \to \CP^1$ is called a \emph{parameterized quasicircle} if it is the restriction to $\RP^1$ of a quasiconformal homeomoprhism of $\CP^1$. Note that the image $C= f(\RP^1)$ is by definition, a quasicircle.
\end{defi}

Given a parameterized quasicircle $f\co \RP^1 \to \CP^1$ with image $C = f(\RP^1)$, let $\Omega^\pm$ be the top and bottom components of $\CP^1 \setminus C$ and let $U^\pm_C\co \HH^{2\pm} \to \Omega^\pm$ be the conformal homeomorphisms as in Section~\ref{sec:intro-hyp}. We do not require here that the extensions $\partial U^\pm\co \RP^1 \to C$ be normalized. Then, define $u^+_f\co \RP^1 \to \RP^1$ and  $u^-_f\co \RP^1 \to \RP^1$  by
\begin{align*}
u^+_f &= (\partial U^+_C)^{-1} \circ f\\\
u^-_f &= (\partial U^-_C)^{-1} \circ f
\end{align*}
Since $f$ is a parameterized quasicircle, both $u^+_C$ and $u^-_C$ are quasisymmetric. 
Since each of $U^\pm_C$ is well-defined only up to pre-composition by an element of $\PSL(2,\RR)$, each of $u^\pm_f$ is well-defined up to post-composition by an element of $\PSL(2,\RR)$, hence the pair $(u^+_f, u^-_f)$ is naturally an element of $\cT \times \cT$. The universal version of Bers' Simultaneous Uniformization Theorem \cite{bers} states that the map $f \mapsto u^+_f$ is a homeomorphism from the space $\pqc$ of parameterized quasicircles, up to post-composition by an element of $\PSL(2,\CC)$, to the product $\cT \times \cT$ of two copies of the universal Teichm\"uller space. Note that a gluing map between the upper and lower regions of the complement of $C$ is given by $\varphi_C = u^-_f \circ (u^+_f)^{-1}$.

For each $K \in [-1,0)$, let $V^\pm_{C, K}\co \HH^{2\pm}_{C, K} \to S^\pm_K(C)$ denote an isometry from the rescaled hyperbolic plane to the top/bottom $K$-surface spanning $C$ in $\HH^3$ (as usual $S^\pm_{-1}(C)$ refers to $\partial^\pm\CH(C)$), with $\partial \VCK{C,K}^\pm: \RP^1 \to C$ denoting the extension to the ideal boundary.
Then, define $v^+_{K, f}: \RP^1 \to \RP^1$ and  $v^-_{K, f}: \RP^1 \to \RP^1$  by
\begin{align}
v^+_{K,f} &= (\partial \VCK{C, K}^+)^{-1} \circ f \nonumber\\
v^-_{K,f} &= (\partial \VCK{C, K}^-)^{-1} \circ f. \label{eqn:two-maps}
\end{align}
Each of $v^\pm_{K,f}$ is well-defined, up to post-composition by $\PSL(2,\RR)$, and since $f$ is a parameterized quasicircle, each of $v^\pm_{K,f}$ is the boundary extension of a quasiconformal map of the disk, hence quasisymmetric. Therefore each of $v^\pm_{K,f}$ is naturally an element of $\cT$. 
Note that the gluing map between the top and bottom $K$-surfaces spanning $C = f(\RP^1)$ is given by $\Phi_{C, K} = v^-_{K,f} \circ (v^+_{K,f})^{-1}$. The following is the parameterized version of Theorems~\ref{tm:induced-hyp} and~\ref{tm:induced-hyp-K}.

\begin{theorem}\label{thm:parameterized-hyp}
For fixed $K \in [-1,0)$, the map $\pqc \to \cT \times \cT$ defined by $f \mapsto (v^+_{K,f}, v^-_{K,f})$ is surjective.
\end{theorem}

\begin{proof}
Let $(v^+, v^-) \in \cT \times \cT$. By Theorem~\ref{tm:induced-hyp} if $K = -1$, or Theorem~\ref{tm:induced-hyp-K} if $K \in (-1,0)$, there exists a normalized quasicircle $C \subset \CP^1$ so that $\Phi_{C, K} = v^- \circ (v^+)^{-1} \in \cT$. Let $V^\pm_{C, K}: \HH^{2\pm}_K \to S^\pm_K(C)$ denote the isometry whose extension $\partial V^\pm_{C, K}$ maps $i$ to $i$ for all $i = 0,1, \infty$. 
Now define $f: \RP^1 \to \CP^1$ by the formula
\begin{align*}
f = (\partial \VCK{C, K}^+) \circ v^+
\end{align*}
Then $v^+_{K,f}  = v^+$ clearly holds. Finally,  
\begin{align*}
v^-_{K,f} &=  \Phi^K_C \circ v^+_{K,f}\\
& = v^- \circ (v^+)^{-1} \circ v^+_{K,f} \\ &= v^-\\
\end{align*}
holds as well.
\end{proof}

Note that there is also a parametrized version of Theorem \ref{tm:induced-hyp-K} which can be proved using  similar arguments.

\subsection{Parameterized quasicircles in $\Ein^{1,1}$}

\begin{defi}
An embedding $f: \RP^1 \to \Ein^{1,1} = \RP^1 \times \RP^1$ is called a \emph{parameterized quasicircle} if both the left projection $\pi_l \circ f$ and the right projection $\pi_r \circ f$ are quasisymmetric. Note that the image $C= f(\RP^1)$ is then a quasicircle in $\Ein^{1,1}$ since it is the graph $\Gamma(g)$ of the quasisymmetric homeomorphism $g = (\pi_r \circ f) \circ (\pi_l \circ f)^{-1}$, 
\end{defi}
The space of parameterized quasicircles in $\Ein^{1,1}$ is denoted $\pqc(\Ein^{1,1})$. Clearly, $\pqc(\Ein^{1,1})$ is in natural bijection with $\cT \times \cT$ via the map $f \mapsto (\pi_l \circ f, \pi_r \circ f)$.

Let $f: \RP^1 \to \Ein^{1,1}$ be a parameterized quasicircle.
Exactly as in the previous section, for each value of $K \in (\infty, -1]$, $f$ determines two maps $v^+_{K,f}$ and $v^-_{K,f}$ by the same formula~\eqref{eqn:two-maps}, where now the maps $\VCK{C, K}^\pm$ denote isometries, well-defined, up to pre-composition with an element of $\PSL(2,\RR)$, between $\HH^{2\pm}_K$ and the future/past $K$-surfaces $\Sa^\pm_K(C)$ spanning $C$ in $\AdS^3$. The map $v^+_{C, K}$ is the composition of two homeomorphisms $\RP^1 \to \RP^1$,
\begin{align*}
v^+_{C, K} &= (\pi_r \partial \VCK{C, K}^+) \circ (\pi_r \circ f)
\end{align*}
the first of which is quasisymmetric because it is the boundary extension of a quasi-isometry (Lemma~\ref{lem:unif-qi-ads} if $K = -1$, or Proposition~\ref{prop:bilip} is $K < -1$), and the second of which is quasisymmetric because $f$ is a parameterized quasicircle.
Hence $v^+_{C, K}$ is quasisymmetric, and so it determines an element of $\cT$.  Similarly,  $v^-_{C, K}$ is quasisymmetric.
 The following is the parameterized version of Theorems~\ref{tm:induced-ads} and~\ref{tm:induced-ads-K}.

\begin{theorem}\label{thm:parameterized-ads}
For fixed $K \in (-\infty,-1]$, the map $\pqc(\Ein^{1,1}) \to \cT \times \cT$ defined by $f \mapsto (v^+_{K,f}, v^-_{K,f})$ is surjective.
\end{theorem}

The proof is identical to that of Theorem~\ref{thm:parameterized-hyp} using Theorems~\ref{tm:induced-ads} and~\ref{tm:induced-ads-K} in place of Theorems~\ref{tm:induced-hyp} and~\ref{tm:induced-hyp-K}. There is also a parametrized version of Theorem \ref{tm:III-ads-K} which can be proved with similar techniques.

\bibliographystyle{amsalpha}
\bibliography{adsbib}

\end{document}